\documentclass[10pt]{article}
\usepackage[top=1in, bottom=1in, left=1in, right=1in]{geometry}

\usepackage{euscript}
\usepackage{microtype,enumitem}
\usepackage{graphicx}
\graphicspath{{figs/}}

\usepackage{subfigure,algorithm,algorithmic}
\usepackage{booktabs} %
\usepackage[usenames,svgnames,dvipsnames]{xcolor}
\usepackage[linktocpage,colorlinks,linkcolor=blue,anchorcolor=blue,citecolor=blue,urlcolor=blue,pagebackref]{hyperref}
\usepackage{natbib}
\setcitestyle{authoryear,open={(},close={)}} 

\usepackage{amsmath}
\usepackage{amssymb}
\usepackage{mathtools}
\usepackage{amsthm}
\usepackage[textsize=tiny]{todonotes}

\usepackage[capitalize,noabbrev%
]{cleveref}

\usepackage{pgfplots}
\DeclareUnicodeCharacter{2212}{−}
\usepgfplotslibrary{groupplots,dateplot}
\usetikzlibrary{patterns,shapes.arrows,external}

\newcommand{\figdir}{-figs/}
\pgfplotsset{compat=newest,
    width=12cm,
    height=7cm,
    every axis plot/.append style={thick},
}

\theoremstyle{plain}
\newtheorem{theorem}{Theorem}[section]

\newtheorem{lemma}[theorem]{Lemma}
\newtheorem{corollary}[theorem]{Corollary}
\theoremstyle{definition}

\theoremstyle{remark}
\newtheorem{remark}[theorem]{Remark}

\renewcommand*{\backref}[1]{\ifx#1\relax \else Page #1 \fi}
\renewcommand*{\backrefalt}[4]{%
  \ifcase #1 \footnotesize{(Not cited.)}%
  \or        \footnotesize{(Cited on page~#2.)}%
  \else      \footnotesize{(Cited on pages~#2.)}%
  \fi
}

\usepackage{macros}

\usepackage[nolist]{acronym}
\begin{acronym}
    \acro{MCMC}{Markov Chain Monte Carlo}
    \acro{VI}{variational inference}
    \acro{KL}{Kullback--Leibler}
    \acro{OSI}{one-step inequality}
    \acro{BW}{Bures--Wasserstein}
    \acro{PGA}{Proximal Gradient Algorithm}
    \acro{SVGD}{Stein Variational Gradient Descent}
    \acro{SPGA}{Stochastic Proximal Gradient Algorithm}
    \acro{WPGA}{Wasserstein Proximal Gradient Algorithm}
    \acro{BWPGA}{Bures--Wasserstein Proximal Gradient Algorithm}
    \acro{BWSPGA}{Bures--Wasserstein Stochastic Proximal Gradient Algorithm}
\end{acronym}

\usepackage{mathrsfs}

\allowdisplaybreaks{}

\crefname{Inequality}{Inequality}{Inequalities}

\author{Michael Diao\thanks{Department of  Electrical Engineering and Computer Science, Massachusetts Institute of Technology. \texttt{diao@mit.edu} }
\and Krishnakumar Balasubramanian\thanks{Department of Statistics, University of California, Davis. \texttt{kbala@ucdavis.edu}} 
\and Sinho Chewi\thanks{Department of Mathematics, Massachusetts Institute of Technology. \texttt{schewi@mit.edu}}
\and Adil Salim\thanks{Microsoft Research. \texttt{adilsalim@microsoft.com} }
}

\begin{document}
\title{Forward-Backward Gaussian Variational Inference \\
via JKO in the Bures--Wasserstein Space}

\maketitle

\begin{abstract}
    \Ac{VI} seeks to approximate a target distribution $\pi$ by an element of a tractable family of distributions. 
    Of key interest in statistics and machine learning is Gaussian VI,
    which approximates $\pi$ by minimizing
    the \ac{KL} divergence to\ $\pi$ over the space of Gaussians.
    In this work, we develop %
    the (Stochastic) Forward-Backward Gaussian Variational Inference (FB--GVI) algorithm to solve Gaussian VI.
    Our approach exploits the composite structure of the \ac{KL} divergence, which can be written as the sum of a smooth term (the potential) and a non-smooth term (the entropy) over the \ac{BW} space of Gaussians endowed with the Wasserstein distance.
    For our proposed algorithm,
    we obtain state-of-the-art convergence guarantees when $\pi$ is log-smooth and log-concave,
    as well as the first convergence guarantees to first-order stationary solutions when $\pi$ is only log-smooth.
\end{abstract}

\section{Introduction}

Variational inference (VI)~\citep{blei2017variational,knoblauch2022optimization} has emerged as a tractable alternative to computationally demanding Monte Carlo Markov Chain (MCMC) methods. Of particular interest is the problem of Gaussian VI, in which we approximate a given distribution $\pi \propto \exp(-V)$, where $V$ is a smooth function, by the solution to
\begin{equation}
\label{eq:original}
    \argmin_{\mu\in \PG(\RR^d)} \KL{\mu}{\pi}\,,
\end{equation}
where $\mathsf{KL}$ denotes the Kullback--Leibler divergence and $\PG(\RR^d)$ the set of Gaussian distributions over $\RR^d$ (see Section~\ref{sec:pb} for formal definitions). Indeed, Gaussian VI has shown superior performance in practice, especially in the presence of large datasets, see, for example,~\citet{barber1997ensemble,seeger1999bayesian,honkela2004unsupervised,opper2009variational,quiroz2022gaussian}. 

In the literature on Gaussian VI, strong \textit{statistical} properties have been shown for the solutions to Problem~\eqref{eq:original}, see, for example,~\citet{cherief2019generalization, alquier2020concentration,katsevich2023approximation}. For instance, \citet{katsevich2023approximation} showed that Gaussian VI outperforms Laplace approximation for the estimation of the mean of $\pi$. Besides, consider the case where $\pi$ is the posterior distribution of a sufficiently regular Bayesian model. Then, the Bernstein--von Mises theorem (see~\citet[Chapter 10]{van2000asymptotic} and recent non-asymptotic results~\citep{kasprzak2022good,spokoiny2022dimension} state that $\pi$ is well-approximated by a Gaussian distribution, with the mean given by any asymptotically efficient estimator of the true parameter, and covariance matrix given by the inverse Fisher information matrix. These results collectively provide abundant motivation for efficiently computing the best Gaussian approximation of the target $\pi$ (Problem~\eqref{eq:original}).

We hence focus on the \textit{optimization} aspect of Gaussian VI, \textit{i.e.}, solving Problem~\eqref{eq:original}. Several approaches have been proposed which we summarize in the related works (Section~\ref{scn:related}). In particular,~\citet{lambert2022variational} recently proposed an algorithm for Gaussian VI that can be seen as an analog of stochastic gradient descent for Problem~\eqref{eq:original} over the space $\PG(\RR^d)$ endowed with the Wasserstein distance, called the Bures--Wasserstein (BW) space. This viewpoint was inspired from the theory of gradient flows over the Wasserstein space, \textit{i.e.}, the space of distributions endowed with the Wasserstein distance~\citep{jordan1998variational,ambrosio2008gradient}. Moreover, this viewpoint has been instrumental for many problems in probabilistic inference (see the related works in Section~\ref{scn:related}). However, from an optimization standpoint, the approach of~\citet{lambert2022variational} relying on the BW gradient of the objective is not the most natural one.
Indeed, over the BW space, the objective functional $\KL{\cdot}{\pi}$ is composite: it can be canonically decomposed as the sum of a ``smooth'' term called the potential and a ``non-smooth'' term called the entropy.

The composite nature of the KL divergence has inspired more than two decades of research on forward-backward methods on the Wasserstein space~\citep[see, for example,][]{jordan1998variational,bernton2018jko, sampling-as-opt, WPGA}.
Unfortunately, this line of work is obstructed by the computational intractability of the so-called JKO operator~\citep{jordan1998variational} (\textit{i.e.}, the analog of the proximal operator over the Wasserstein space) of the entropy.

In this paper, we introduce a novel algorithm called (Stochastic) Forward-Backward Gaussian Variational Inference (FB--GVI).
Similarly to~\citet{lambert2022variational}, the rich differential and geometric structure of the BW space comprises the linchpin of our approach.
However, (Stochastic) FB--GVI additionally incorporates celebrated ideas from the literature on composite and non-smooth optimization.
A key insight in this work is that the JKO operator for the entropy, when restricted to the BW space, admits a closed form~\citep{sampling-as-opt}, and hence leads to an implementable (stochastic) forward-backward (or proximal gradient) algorithm for Gaussian VI\@.
In turn, it yields new state-of-the-art computational guarantees for Gaussian VI under a variety of standard assumptions.

We summarize our \textbf{contributions} below.

\begin{itemize}
\item We propose a new (stochastic) forward-backward algorithm, (Stochastic) FB--GVI, to solve Problem~\eqref{eq:original}. The algorithm relies on a closed-form formula for the JKO operator of the entropy over the BW space.
\item We prove state-of-the-art convergence rates for Gaussian VI via our algorithm, leveraging recent techniques of optimization over the space of probability measures~\citep{ambrosio2008gradient}.  
\end{itemize}

The rest of our paper is organized as follows. In Section~\ref{sec:background}, we clarify our notation and provide some background material on stochastic and non-smooth composite optimization. In Section~\ref{sec:BW}, we describe the geometric and differential structure of the BW space, which is key to performing optimization over $\PG(\RR^d)$. Then, in Section~\ref{sec:sfbgvi} we focus on Problem~\eqref{eq:original} and propose our algorithms: FB--GVI and Stochastic FB--GVI.
The convergence of our algorithms is studied in Section~\ref{sec:convergence} and preliminary simulation results are provided in Section~\ref{scn:exp}.
Finally, we discuss related works in Section~\ref{scn:related} and conclude in Section~\ref{sec:ccl}. A Jupyter notebook containing code for our experiments can be found at
\begin{center}
 \url{https://github.com/mzydiao/FBGVI/blob/main/FBGVI-Experiments.ipynb}.   
\end{center}

\section{Background}
\label{sec:background}

In this section, we clarify our notation and provide background on (stochastic) forward-backward algorithms.

\subsection{Notation}

We will denote the space of real symmetric $d\times d$ matrices by $\mathbf{S}^d$ and the space of real positive definite $d\times d$ matrices by $\mathbf{S}_{++}^d$. Additionally, we denote the $d\times d$ dimensional identity matrix by $I$. Throughout, $\mP_2(\RR^d)$ is the set of probability measures $\mu$ over $\RR^d$ with finite second moment $\int \|x\|^2\, \dd\mu(x) < \infty$.
Let $\mu \in \mP_2(\RR^d)$. The space $L^2(\mu)$ is the Hilbert space of Borel functions $f: \RR^d \to \RR^d$ such that $$\EE_\mu \|f\|^2 = \int \|f(x)\|^2\,\dd\mu(x) < \infty\,,$$ endowed with the inner product $$\langle f,g \rangle_\mu \defeq \int \langle f(x),g(x) \rangle\,\dd\mu(x)$$ and the associated norm $\|f\|_\mu = \sqrt{\langle f, f \rangle_\mu}$. In particular, the identity map $\id: \RR^d \to \RR^d$ belongs to $L^2(\mu)$. If $\mu \in \mP_2(\RR^d)$ and $T \in L^2(\mu)$, the pushforward measure of $\mu$ by $T$ is denoted by $T_\# \mu$. This pushforward measure satisfies $\int \varphi\, \dd T_\# \mu = \int \varphi(T(x))\,\dd \mu(x)$ for any measurable function $\varphi : \RR^d \to \RR_+$.
The subset of $\mP_2(\RR^d)$ of all Gaussian distributions with positive definite covariance matrix is denoted by $\PG(\RR^d)$.
For an element $\mu \in \PG(\RR^d)$, we denote its mean by $m_\mu$ and its covariance matrix by $\Sigma_\mu$.
The notation $\mN(m,\Sigma)$ refers to the Gaussian distribution with mean $m \in \RR^d$ and covariance matrix $\Sigma \in \mathbf{S}_{++}^{d}$.

\subsection{Stochastic and non-smooth convex optimization over $\RR^d$}

Before introducing optimization concepts on the space of probability measures, we review some details of stochastic and convex non-smooth optimization over $\RR^d$. First, a function $V: \RR^d \to \RR$ is $\beta$-smooth if $V$ is twice continuously differentiable and its Hessian $\nabla^2 V(x)$ is bounded by $\beta$ in the operator norm, for every $x \in \RR^d$. In particular, $V$ is differentiable and its gradient $\nabla V$ is $\beta$-Lipschitz. In addition, $V$ satisfies the Taylor inequality
\begin{equation}
\label{eq:smooth}
\left| V(x + h) - V(x) - \langle \nabla V(x), h \rangle \right| \leq \frac{\beta}{2}\,\|h\|^2\,.
\end{equation}

Consider the optimization problem
\begin{equation}
\label{eq:optim-euc}
\min_{x \in \RR^d}{\{V(x) + H(x)\}}\,,
\end{equation}
where $V: \RR^d \to \RR$ is $\beta$-smooth and $H: \RR^d \to \RR$ is convex but potentially non-smooth. To simplify the presentation, we also assume that $V$ is convex. Because of the non-smoothness of $H$, the (sub)gradient descent algorithm applied to Problem~\eqref{eq:optim-euc} may not converge to a minimizer of $V+H$. However, in many situations of interest, the user has closed-form expressions\footnote{See \url{proximity-operator.net}.} for the proximal operator of $H$ defined by 
\begin{equation}
\label{eq:prox}
\prox_{H}(x) \defeq \argmin_{y \in \RR^d}{\Bigl\{H(y) + \frac12\, \|x-y\|^2\Bigr\}}\,.
\end{equation}

Given access to the proximal operator of $H$ and the gradient of $V$, the forward-backward algorithm~\citep{bauschke2011convex} is one of the most natural and efficient techniques to solve Problem~\eqref{eq:optim-euc}.
The forward-backward algorithm is written as
\begin{equation}
\label{eq:fb}
x_{k+1} = \prox_{\eta H}(x_{k} - \eta\, \nabla V(x_k))\,.
\end{equation}
In machine learning applications, the user often does not have direct access to $\nabla V(x_k)$ because computing the gradient of $V$ is expensive. Instead, the user has access to a cheaper stochastic estimator $\hat{g}_k$ of $\nabla V(x_k)$. In this situation, the stochastic forward-backward algorithm has been proven to be an efficient alternative to the forward-backward algorithm~\citep{atchade2017perturbed, bianchi2019constant, gorbunov2020unified}. In the stochastic forward-backward algorithm, $\nabla V(x_k)$ is replaced by $\hat{g}_k$ as follows:
\begin{equation}
\label{eq:sfb}
x_{k+1} = \prox_{\eta H}(x_{k} - \eta\, \hat{g}_k)\,.
\end{equation}
We will now adapt the (stochastic) forward-backward algorithm~\eqref{eq:sfb} to the BW space. In particular, the iterate at step $k$ will no longer be a random variable $x_k$, but rather a random Gaussian distribution $\iterate_k$. Despite the fact that the BW space is not a Euclidean space, its inherent structure still allows us to perform optimization, as explained next.

\section{The Bures--Wasserstein space}
\label{sec:BW}
A detailed presentation of the Wasserstein space and its geometry, which in turn enables optimization over that space, can be found in~\citet{ambrosio2008gradient}. In this section, we quickly review the BW space and its geometry, hence providing the requisite tools to perform optimization over the BW space and solve Problem~\eqref{eq:original}. %
We start with formal definitions of the Wasserstein and BW spaces.

\subsection{Geometry of the BW space}

The \textit{Wasserstein space} is the metric space $\mP_2(\RR^d)$ endowed with the $2$-Wasserstein distance $W_2$ (which we simply refer to as the Wasserstein distance). We recall that the Wasserstein distance is defined for every $\mu,\nu \in \mP_2(\RR^d)$ by
\begin{equation}
\label{eq:def-wass}
W_2^2(\mu,\nu) = \inf_{\gamma \in \eu C(\mu,\nu)}\int \|x-y\|^2\, \dd\gamma(x,y)\,,
\end{equation}
where $\eu C(\mu,\nu)$ is the set of couplings between $\mu$ and $\nu$. 
The BW space is the metric space $\PG(\RR^d)$ endowed with the Wasserstein distance $W_2$. In other words, the BW space is the subset of the Wasserstein space consisting of all Gaussian distributions with positive definite covariance matrix. 

Given $\mu,\nu \in \PG(\RR^d)$, there exists a unique optimal transport map from $\mu$ to $\nu$, \textit{i.e.}, a map $T : \RR^d \to \RR^d$ such that $T_\# \mu = \nu$ and 
\begin{equation}
W_2^2(\mu,\nu) = \int \|x-T(x)\|^2\,\dd\mu(x)\,.
\end{equation}
In other words, the coupling $(\id,T)_\# \mu$ belongs to $\eu C(\mu,\nu)$ and attains the infimum in~\eqref{eq:def-wass}. Moreover, since $\mu$ and $\nu$ are Gaussian, $T$ is an affine map with symmetric linear part, \textit{i.e.}, can be written as $T(x) = Sx + b$ where
$S \in \mathbf{S}^d$
and $b \in \RR^d$%
~\citep{olkpuk1982otgaussian}. In particular, the BW space is a Riemannian manifold where at each $\mu \in \PG(\RR^d)$, the tangent space $\Tan_{\mu} \PG(\RR^d)$ corresponds to the space of $d$-dimensional affine maps with symmetric linear part.
Using that $\mu \in \mP_2(\RR^d)$, $T_\# \mu \in \mP_2(\RR^d)$ implies $T \in L^2(\mu)$. Therefore, $\Tan_{\mu} \PG(\RR^d)$ is naturally endowed with the $L^2(\mu)$ inner product, making $\Tan_{\mu} \PG(\RR^d)$ a finite-dimensional subspace of $L^2(\mu)$.

\subsection{Optimization over the BW space}

In this section, we review the differential structure of the BW space. Further background on differential calculus over the BW space is provided in \Cref{sec:bw-grad-comp}.

\subsubsection{A ``smooth'' functional}

Consider a functional $\mF : \PG(\RR^d) \to \RR$. We say that $\mF$ is differentiable at $\mu$ if there exists $g_{\mu} \in \Tu$ such that for every affine map $h$, 
\begin{equation}
\label{eq:BWdef}
\mF((\id + t h)_{\#} \mu) = \mF(\mu) + t\, \langle g_{\mu}, h \rangle_{\mu} + o(t)\,.
\end{equation}
In this case, $g_{\mu}$ is unique, called the Bures--Wasserstein gradient of $\mF$ at $\mu$, and denoted $\nabla_{\mathsf{BW}} \mF(\mu) = g_{\mu}$. Given a $\beta$-smooth function $V : \RR^d \to \RR$, the potential energy functional $\mV \colon \PG(\RR^d)\to\RR$, defined by
\begin{equation}\label{eq:potential_energy}
\mV(\mu) \defeq \int V\, \dd \mu
\end{equation}
for every $\mu \in \PG(\RR^d)$,
is a useful example of a differentiable functional over the BW space.

The next result states that the potential is differentiable and gives a formula for its BW gradient. The formula for the BW gradient can be obtained by a straightforward adaptation of~\citet[Section C.1]{lambert2022variational} (see also~\Cref{sec:BWgradpot}).
Besides, the differentiability of $\mV$ is well-established in the literature on the Wasserstein space~\citep[Theorem 10.4.13]{ambrosio2008gradient}; we adapt this to the BW space.

\begin{lemma}[BW gradient of the potential]
\label{lem:smooth}
    Consider the potential functional $\mV$ in~\eqref{eq:potential_energy} where $V$ is $\beta$-smooth. Then, $\mV$ is differentiable at $\mu$ and the following Taylor inequality holds: for $h$ affine,
    \begin{equation}
    \label{eq:Taylor2}
\left| \mV((\id+ h)_\# \mu) - \mV(\mu) - \langle \nabla_{\mathsf{BW}}\mV(\mu),h \rangle_{\mu} \right| \leq \frac{\beta}{2}\,\|h\|_{\mu}^2\,.
\end{equation}

Moreover, the BW gradient of $\mV$ is known in closed form:
    \begin{equation}
    \label{eq:BWgrad}
    \nabla_{\mathsf{BW}}\mV(\mu): x \mapsto \EE_\mu \nabla V + (\EE_\mu \nabla^2 V)(x - m_{\mu})\,,
    \end{equation}
    where $m_{\mu} = \int x \, \dd\mu(x)$ is the mean of $\mu$.
\end{lemma}
\begin{proof}
The proof of Equation~\eqref{eq:BWgrad} can be found in~\citet[Section C.1]{lambert2022variational} (see also~\Cref{sec:BWgradpot}), and the proof of Inequality~\eqref{eq:Taylor2} can be found in \Cref{sec:proof-lem-pot}.
\end{proof}
Inequality~\eqref{eq:Taylor2} is stronger than differentiability and can be interpreted as the potential $\mV$ being $\beta$-smooth over the BW space (note the analogy with Inequality~\eqref{eq:smooth}). Inequality~\eqref{eq:Taylor2} is a consequence of the smoothness of $V$. As in optimization over $\RR^d$, when dealing with a convex but potentially non-smooth functional, the user may prefer to handle it through its proximal operator.

\subsubsection{A ``convex'' functional}

We say that $\mF$ is geodesically convex if for all $\mu_0,\mu_1 \in \PG(\RR^d)$, 
\begin{equation}
\label{eq:BWcvx}
\mF(\mu_0) + \langle \nabla_{\BW} \mF(\mu_0), T - \id \rangle_{\mu_0} \leq \mF(\mu_1)\,,
\end{equation}
where $T$ is the optimal transport map from $\mu_0$ to $\mu_1$.
In this case, we can introduce an analog of the proximal operator of $\mF$ over the BW space, called the Bures--Wasserstein JKO operator of $\mF$~\citep{jordan1998variational}\footnote{In~\citet{jordan1998variational} the authors define the JKO operator as an analog of the proximal operator over the Wasserstein space. Inspired by their definition, we define the JKO operator over the BW space, and we call it the BW JKO operator.}, and defined by (note the analogy with~\eqref{eq:prox})
\begin{equation}
\label{eq:jko}
\JKO_{\mF}(\mu) \defeq \argmin_{\nu \in \PG(\RR^d)}{\Bigl\{\mF(\nu) + \frac{1}{2}\,W_2^2(\mu,\nu)\Bigr\}}\,.
\end{equation}
We want to emphasize that, in our definition~\eqref{eq:jko}, the BW JKO operator is defined over the BW space.

The entropy $\mH$ is a useful example of a geodesically convex functional over the BW space. More precisely, the entropy is defined by
\begin{equation}\label{eq:entropy}
\mH(\mu) = \int \log \mu(x)\, \dd \mu(x)\,,
\end{equation}
for every $\mu \in \PG(\RR^d)$, where we identify $\mu$ with its density w.r.t.\ Lebesgue measure. 

The next lemma states that the entropy is geodesically convex and gives a formula for its BW JKO operator. The formula for the BW JKO operator can be obtained by a straightforward adaptation of~\citet[Example 7]{sampling-as-opt}. Besides, the geodesic convexity of $\mH$ is well-established in the literature on the Wasserstein space~\citep[Remark 9.3.10]{ambrosio2008gradient}; we adapt this to the BW space.

\begin{lemma}[BW JKO of the entropy]
\label{lem:cvx}
    Consider the entropy functional $\mH$ defined in~\eqref{eq:entropy}. Then, $\mH$ is geodesically convex and the following stronger inequality holds: for all $\nu, \mu_0, \mu_1 \in \PG(\RR^d)$, 
\begin{equation}
\label{eq:BWcvx2}
\mH(\mu_0) + \langle \nabla_{\BW} \mH(\mu_0) \circ T_0, T_1 - T_0 \rangle_{\nu} \leq \mH(\mu_1)\,,
\end{equation}
where $T_0$ (resp.\ $T_1$) is the optimal transport map from $\nu$ to $\mu_0$ (resp.\ $\mu_1$). 

Moreover, the BW JKO operator of $\mH$ is known in closed form: $\JKO_{\eta\mH}(\mu)$ (where $\eta > 0$) is a Gaussian distribution with same mean as $\mu$ and with covariance matrix $\Sigma_1$ where
    \begin{equation}
    \label{eq:BWjko}
    \Sigma_1 = \half \,\bigl( \Sigma + 2 \eta I + \left[ 
        \Sigma\,(\Sigma + 4\eta I)\right]^{1/2} \bigr)\,,
    \end{equation}
    where $\Sigma$ is the covariance matrix of $\mu$.
\end{lemma}
\begin{proof}
The proof of Equation~\eqref{eq:BWjko} can be found in \citet[Example 7]{sampling-as-opt}, and the proof of Inequality~\eqref{eq:BWcvx2} can be found in \Cref{sec:proof-lem-ent}.
\end{proof}

Inequality~\eqref{eq:BWcvx2} is stronger than geodesic convexity and is a consequence of the generalized geodesic convexity of the entropy~\citep[Remark 9.3.10]{ambrosio2008gradient}. \textit{Finally,~\eqref{eq:BWjko} which gives the BW JKO of the entropy in closed form is remarkable, and is at the core of our approach.} As a comparison, \citet{WPGA} proposed an algorithm relying on the JKO of the entropy over the whole Wasserstein space, but the latter JKO is not implementable.

\section{(Stochastic) Forward-backward Gaussian variational inference}
\label{sec:sfbgvi}

\subsection{Revisiting Gaussian VI}
\label{sec:pb}

We now restate Problem~\eqref{eq:original} more formally. We assume that the target distribution $\pi$ admits a positive density w.r.t.\ Lebesgue measure, denoted $\pi$ as well in an abuse of notation. We write $\pi$ in the form $\pi \propto \exp(-V)$. Moreover, we assume that the function $V : \RR^d \to \RR$ is $\beta$-smooth. Recall that the KL divergence is defined for every $\mu \in \PG(\RR^d)$ as
\begin{equation}
\KL{\mu}{\pi} = \int \log\frac{\mu(x)}{\pi(x)}\, \dd\mu(x)\,.
\end{equation}
Recall that our goal is to solve Problem~\eqref{eq:original}. We denote $\mF \defeq \mV + \mH$ as the sum of the potential (associated to the function $V$) and the entropy.
Then, a quick calculation reveals that $\mF(\mu) - \mF(\pi) = \KL{\mu}{\pi}$. Since $\mF(\pi)$ is a constant (\textit{i.e.}, does not depend on $\mu$),
Problem~\eqref{eq:original} is equivalent to
\begin{equation}
\label{eq:original2}
\min_{\mu \in \PG(\RR^d)}{\{\mV(\mu) + \mH(\mu)\}}\,. 
\end{equation}

\subsection{Proposed algorithm}

Recall that the potential $\mV$ is ``smooth'' over the BW space and that the BW gradient of $\mV$ admits a closed form (Lemma~\ref{lem:smooth}). Recall also that the entropy $\mH$ is ``convex'' over the BW space and that the BW JKO of $\mH$ admits a closed form (Lemma~\ref{lem:cvx}). To solve the equivalent problem~\eqref{eq:original2}, a natural idea is to adapt the forward-backward algorithm to the BW space. This leads to the following Forward-Backward Gaussian Variational Inference (FB--GVI) algorithm (note the analogy with~\eqref{eq:fb}):
\begin{align}
\iterate_{k + \half} &= (\id - \eta\, \nabla_{\mathsf{BW}}\mV(\iterate_k))_{\#} \iterate_k\,, \label{eq:dfb1}\\
\iterate_{k + 1} &= \jko_{\eta \mH}(\iterate_{k + \half})\,. \label{eq:dfb2}
\end{align}
The backward step~\eqref{eq:dfb2} is tractable using~\eqref{eq:BWjko}. Although the forward step~\eqref{eq:dfb1} also admits a closed form, the forward step involves computing integrals of $\nabla V$ and $\nabla^2 V$ w.r.t.\ $\iterate_k$, see~\eqref{eq:BWgrad}. These integrals can be intractable. Therefore, we propose to build a stochastic estimate $\hat{g}_k$ of $\nabla_{\BW} \mV(\iterate_k)$, \textit{i.e.}, a stochastic gradient, by drawing a random sample from $\iterate_k$. The resulting algorithm is called Stochastic FB--GVI, and can be written as
\begin{align}
\label{eq:bwsfb}
\iterate_{k + \half} = (\id - \eta\, \hat{g}_k)_{\#} \iterate_k\,, \nonumber\\
\iterate_{k + 1} = \jko_{\eta \mH}(\iterate_{k + \half})\,,
\end{align}
where $\hat{g}_k$ is the random affine function defined by
\begin{equation}
\label{eq:sto-grad}
\hat{g}_k: x \mapsto 
\nabla V(\hat{X}_k) + \nabla^2 V(\hat{X}_k)\,(x - m_k)\,,  
\end{equation}
where $\hat{X}_k$ is sampled from $\iterate_k$ (\textit{i.e.}, $X_k \sim \iterate_k$) and 
$m_{k} = \int x \, d\iterate_k(x)$ is the mean of $\iterate_k$. 

Stochastic FB--GVI is an analog, over the BW space, of the stochastic forward-backward algorithm (note the analogy with~\eqref{eq:sfb}). In particular, $(\iterate_k)_{k\in\NN}$ defined by~\eqref{eq:bwsfb} is a sequence of random Gaussian distributions, \textit{i.e.}, random variables with values in $\PG(\RR^d)$. We denote the mean (resp.\ covariance matrix) of $\iterate_k$ by $m_k$ (resp.\ $\Sigma_k$). FB--GVI and Stochastic FB--GVI can be implemented in terms of the means and the covariance matrices of the iterates $\iterate_k$. The iterations of FB--GVI and Stochastic FB--GVI in terms of $m_k$ and $\Sigma_k$ are given in Algorithm~\ref{alg:FBGVI}. Efficient algorithms developed for computing the matrix square-root (see, for example,~\citet{pleiss2020fast,song2022fast}) can be leveraged to improve the per-iteration complexity.  

\begin{algorithm}[tb]
    \caption{FB--GVI and Stochastic FB--GVI}
    \label{alg:FBGVI}
    \begin{algorithmic}
        \REQUIRE
            Step size $\eta > 0$;
            Iteration count $N$;
            Initial distribution $\iterate_0 = \mN(m_0,\Sigma_0)$
        \FOR{$k=0$ {\bfseries to} $N-1$}
        \IF {FB--GVI}
        \STATE $b_k \leftarrow \EE_{\iterate_k} \nabla V$, $S_k \leftarrow \EE_{\iterate_k} \nabla^2 V$
        \ELSIF{Stochastic FB--GVI}
        \STATE draw $\hat{X}_k \sim \mN(m_k,\Sigma_k)$
        \STATE $b_k \leftarrow \nabla V(\hat{X}_k)$, $S_k \leftarrow \nabla^2 V(\hat{X}_k)$
        \ENDIF
        \STATE $m_{k+1} \leftarrow m_k - \eta \, b_k$
        \STATE $M_{k+1} \leftarrow I- \eta\, S_k$
        \STATE  $\Sigma_{k + \half} \leftarrow M_{k+1} \Sigma_k M_{k+1}$
        \STATE $\Sigma_{k + 1} \leftarrow \half (\Sigma_{k + \half} + 2 \eta I + [ 
        \Sigma_{k + \half}(\Sigma_{k + \half} + 4\eta I)]^{1/2})$
        \ENDFOR
        \OUTPUT $p_N = \mN(m_N, \Sigma_N)$
    \end{algorithmic}
\end{algorithm}

\section{Convergence theory}\label{sec:convergence}

In this section, we study the convergence of FB--GVI and Stochastic FB--GVI using their equivalent forms~\eqref{eq:dfb1}--\eqref{eq:dfb2} and~\eqref{eq:bwsfb}. We will make use of standard complexity notations, such as $\gtrsim, \asymp$. We also denote by $\hat{\pi} = \mN(\hat{m},\hat{\Sigma})$ a solution of Problem~\eqref{eq:original} (\textit{i.e.}, a minimizer of the KL objective), and we let $\SF_k$ denote the $\sigma$-algebra generated up to iteration $k$ (but not including the random sample $\hat{X}_k \sim p_k$ in Stochastic FB--GVI).

We consider several assumptions on $V$. Given $\alpha \in \RR$, $V$ is $\alpha$-convex if $\alpha I \preceq \nabla^2 V$. If $\alpha = 0$, $V$ is said to be convex, and if $\alpha > 0$, $V$ is said to be ($\alpha$-)strongly convex. For either algorithm, define the (random) error function (see the definitions of $b_k$ and $S_k$ in \Cref{alg:FBGVI}) as
\begin{align*}
    e_k \colon x \mapsto (S_k - \EE_{\iterate_k} \nabla^2 V) (x - m_k) + (b_k - \EE_{\iterate_k} \nabla V)\,,
\end{align*}
and denote its expected $L^2(\iterate_k)$ norm by $\sigma_k^2 \defeq \EE[\norm{e_k}^2_{\iterate_k}\mid \SF_k]$. The expectation is taken over the possible randomness of $(b_k, S_k)$ (i.e., over the randomness of $\hat{X}_k$). For Stochastic FB--GVI, $e_k = \hat{g}_k - \nabla_{\BW}\mV(\iterate_k)$, where $\hat{g}_k$ is defined by~\eqref{eq:sto-grad}. Since $\EE [e_k\mid \SF_k] = 0$ (\textit{i.e.}, the BW stochastic gradient is unbiased),  $\sigma_k^2$ is the conditional variance of the BW stochastic gradient at iteration $k$. For FB--GVI, $e_k = \nabla_{\BW}\mV(\iterate_k) - \nabla_{\BW}\mV(\iterate_k) = 0$, hence $\sigma_k = 0$.  
Our analysis of FB--GVI and Stochastic FB--GVI relies on the following unified one-step-inequality for the iterates $(\iterate_k)_{k\in \NN}$ of both~\eqref{eq:dfb1}--\eqref{eq:dfb2} and~\eqref{eq:bwsfb}.

\begin{lemma}[\Acl{OSI}]
    \label{lemma:discrete-evi}
    Suppose that $V$ is $\alpha$-convex and $\beta$-smooth.
    Let $(\iterate_k)_{k \in \NN}$ be the iterates of FB--GVI~\eqref{eq:dfb1}--\eqref{eq:dfb2} or Stochastic FB--GVI~\eqref{eq:bwsfb}. Let $\stepsize > 0$ be such that
    \[
    \stepsize \leq \begin{cases}
        \frac{1}{\beta} &\text{if}~\sigma_k = 0~~\text{(FB--GVI)}\,,\\
        \frac{1}{2\beta} &\text{else}\,.
    \end{cases}
    \]
    Then, for all $\nu \in \mathsf{BW}(\RR^d)$,
    \begin{align}
    \label{eq:OSI-main}
        \EE
        W_2^2(\iterate_{k+1}, \nu) 
        &\leq 
        (1 - \alpha\stepsize)\, \EE W_2^2(\iterate_k, \nu)     - 2\stepsize\,\EE [\mF(\iterate_{k+1}) - \mF(\nu)]
        + 2\stepsize^2\,\EE\sigma_k^2\,. 
    \end{align}
\end{lemma}

\begin{proof}
    The proof is given in \Cref{section:discrete-evi-proof}.
\end{proof}
This one-step inequality is similar, by replacing the Wasserstein distance by the Euclidean distance, to the one-step inequality classically used in the analysis of the stochastic forward-backward algorithm~\citep{gorbunov2020unified}. The proof of Lemma~\ref{lemma:discrete-evi} heavily employs the differential and geometric structure of the BW space presented in Section~\ref{sec:BW}.

\subsection{Convergence of FB--GVI}

In this section, $(\iterate_k)_{k\in\NN}$ is the sequence of iterates defined by FB--GVI (\eqref{eq:dfb1}--\eqref{eq:dfb2}). We obtain corollaries of Lemma~\ref{lemma:discrete-evi} by setting $\sigma_k = 0$ in~\eqref{eq:OSI-main}, when $V$ is convex or strongly convex.

\begin{theorem}[Convex case, FB--GVI]
    \label{thm:noiseless-wc-rate}
    Suppose that $V$ is convex and $\beta$-smooth and that $0 < \eta \le \frac{1}{\beta}$. Then,
    \begin{align*}
        \mF(p_{N}) - \mF(\hat{\pi})
        &\leq
        \frac{W_2^2(p_0, \hat{\pi})}{2N\eta}\,.
    \end{align*}
    In particular, when $\eta = \frac{1}{\beta}$ and $N \gtrsim \frac{\beta W_2^2(p_0, \hat{\pi})}{\varepsilon^2}$,
    we obtain the guarantee
    \begin{align*}
        \mF(p_{N}) - \mF(\hat{\pi}) 
        &\leq 
        \varepsilon^2\,.
    \end{align*}
\end{theorem}

\begin{proof}
    The proof is given in \Cref{sec:noiseless-wc-rate-proof}.
\end{proof}
Next, we establish the convergence of $(p_k)$ to a minimizer of Problem~\eqref{eq:original}.
\begin{theorem}[Asymptotic convergence]
    \label{thm:noiseless-wc-no-rate}
    Suppose that $V$ is convex and $\beta$-smooth and that $0 < \eta < \frac{1}{\beta}$. Then, there exists a minimizer $\pi_\star$ of Problem~\eqref{eq:original} such that $W_2(\iterate_k,\pi_\star) \to 0$ as $k\to\infty$.
\end{theorem}
\begin{proof}
    The proof is given in \Cref{sec:noiseless-wc-no-rate-proof}.
\end{proof}
Note that Theorem~\ref{thm:noiseless-wc-no-rate} does not follow from Theorem~\ref{thm:noiseless-wc-rate}. Indeed, Theorem~\ref{thm:noiseless-wc-rate} only implies the weak convergence of ${(p_k)}_{k\ge 0}$ to the \textit{set} of minimizers of Problem~\eqref{eq:original}. The existence of a single minimizer $\pi_\star$ to which the sequence ${(p_k)}_{k\ge 0}$ converges follows from the one-step inequality (Lemma~\ref{lemma:discrete-evi}) applied with $\eta < \frac{1}{\beta}$ and from recent results on the topology of the Wasserstein space~\citep{naldi2021weak}.

\begin{theorem}[Strongly convex case, FB--GVI]
    \label{thm:noiseless-sc-rate}
    Suppose that $V$ is $\alpha$-strongly convex and $\beta$-smooth, and that $0 < \eta \leq \frac{1}{\beta}$. Then,
    \begin{align*}
        W_2^2(p_{N}, \hat{\pi})
        &\leq
        \exp\left( -\alpha N \eta \right)\,
        W_2^2(p_0, \hat{\pi})\,.
    \end{align*}
    In particular, when $\eta = \frac{1}{\beta}$ and $N\gtrsim \frac{\beta}{\alpha} \log \frac{W_2(p_0, \hat{\pi})}{\varepsilon}$,
    we obtain the guarantees
    \begin{align*}
        \alpha\, W_2^2(\mu_{N}, \hat{\pi}) 
        \leq 
        \varepsilon^2\,,
        \quad\text{and}\quad
        \mF(p_{2N}) - \mF(\hat{\pi}) 
        \leq 
        \varepsilon^2\,.
    \end{align*}
\end{theorem}

\begin{proof}
    The proof is given in \Cref{sec:noiseless-sc-rate-proof}.
\end{proof}

Theorem~\ref{thm:noiseless-wc-rate} states the sublinear convergence of FB--GVI for a convex $V$ (in terms of objective gap) and Theorem~\ref{thm:noiseless-sc-rate} states the linear convergence of FB--GVI for a strongly convex $V$. The convergence rates are of the same order as the convergence rates of the deterministic forward-backward algorithm over $\RR^d$~\citep{gorbunov2020unified}.

Finally, we also extend our results to the non-convex case, where we obtain a stationary point guarantee.

\begin{theorem}[Non-convex case, FB--GVI]
    \label{thm:noiseless-smooth-rate}
    Suppose that $V$ is $\beta$-smooth, and that $0 < \eta \leq \frac{1}{\beta}$.
    Let 
    $\Delta \defeq \mF(\iterate_0) - \mF(\hat{\pi})$.
    Then,
    \begin{align*}
        \min_{k \in \left\{ 0, \ldots, N - 1 \right\}} \norm{\nabla_{\BW} \mF(\iterate_k)}_{\iterate_k}^2 \leq \frac{150\Delta}{\eta N}\, .
    \end{align*}
    In particular, when $\eta = \frac{1}{\beta}$ and $N\gtrsim \frac{\beta\Delta}{\varepsilon^2}$,
    we obtain the guarantee
    \begin{align}
    \label{eq:foopt}
        \min_{k \in \left\{ 0, \ldots, N - 1 \right\}} \norm{\nabla_{\BW} \mF(\iterate_k)}_{\iterate_k}^2 \leq \varepsilon^2\, .
    \end{align}
\end{theorem}
\begin{proof}
    The proof is given in \Cref{sec:noiseless-smooth-rate-proof}.
\end{proof}

To the best of our knowledge, this is the first stationary point guarantee for Gaussian VI\@. The relevance of this result is that according to~\citet{katsevich2023approximation}, the favorable statistical properties of Gaussian VI arise, not due to the global minimization of the objective in~\eqref{eq:original}, but rather from the 
first-order optimality~\eqref{eq:foopt}. Hence, Theorem~\ref{thm:noiseless-smooth-rate} can be viewed as an algorithmic result for posterior approximation, even in the non-log-concave setting.

We also remark that in Theorem~\ref{thm:noiseless-smooth-rate}, although we assume that $V$ is smooth, it does \emph{not} imply that the objective $\mF$ is smooth over the Bures{--}Wasserstein space, due to the presence of the entropy term $\mH$.
In fact, the proof of Theorem~\ref{thm:noiseless-smooth-rate} requires a careful control of the eigenvalues of the iterates of FB{--}GVI\@.

\subsection{Convergence of Stochastic FB--GVI}

In this section, $(\iterate_k)_{k\in\NN}$ is the sequence of iterates defined by \eqref{eq:bwsfb}. To use Lemma~\ref{lemma:discrete-evi}, we first prove a bound on  $\sigma_k^2$, the variance of the BW stochastic gradient.

\begin{lemma}\label{lemma:grad-error-bound-text}
    If $V$ is convex and $\beta$-smooth, then
    \begin{align*}
        \sigma_k^2 
        \leq
        6\beta d + 12\beta^3 \lambda_{\max}(\hat{\Sigma})\, W_2^2(\iterate_k, \hat{\pi})\,.
    \end{align*}
    Moreover, if $V$ is $\alpha$-strongly convex, the bound above becomes
    \begin{align*}
        \sigma_k^2
        &\leq
        6\beta d + \frac{12\beta^3}{\alpha}\, W_2^2(\iterate_k, \hat{\pi})\,.
    \end{align*}
\end{lemma}
\begin{proof}
See Appendix~\ref{sec:error-bnd-sto-oracle}.
\end{proof}

The bound on $\sigma_k^2$ is reminiscent of the common assumption made in the literature on stochastic gradient algorithms over $\RR^d$, that the stochastic gradient has sublinear growth~\citep{kushner2003stochastic, bottou2018optimization}. We emphasize that we do not assume this sublinear growth. Instead, Lemma~\ref{lemma:grad-error-bound-text} proves the sublinear growth for the BW stochastic gradient used in Stochastic FB--GVI.  Next, we obtain corollaries of Lemma~\ref{lemma:discrete-evi} for Stochastic FB--GVI by controlling $\sigma_k^2$ with Lemma~\ref{lemma:grad-error-bound-text}.

\begin{theorem}[Convex case, Stochastic FB--GVI]
    \label{thm:stoch-wc-rate}
    Suppose that $V$ is convex and $\beta$-smooth and that $0 < \eta \leq \frac{1}{2\beta}$.
    Define $c \defeq 24\beta^3\lambda_{\max}(\hat{\Sigma})$.
    Then,
    \begin{align*}
        \EE 
        \bigl[ 
        \min_{k \in \left\{ 1, \ldots, N \right\}}
        \mF(p_k)
        \bigr]
        - \mF(\hat{\pi})\leq
        \frac{2W_2^2(p_0, \hat{\pi})}{N\eta}
        + 2c\eta\, W_2^2(p_0, \hat{\pi})
        + 12\beta \eta d\,.
    \end{align*}
    In particular, for sufficiently small values of $\varepsilon^2/d$ and with
    \begin{align*}
        \eta \asymp \frac{\varepsilon^2}{c W_2^2(p_0, \hat{\pi}) \lor \beta d}\,,\quad\text{and}\quad        N \gtrsim \frac{W_2^2(p_0, \hat{\pi})}{\varepsilon^4}\, \bigl( cW_2^2(p_0, \hat{\pi}) \lor \beta d \bigr)\,,
    \end{align*}
    we obtain the guarantee
    \begin{align*}
        \EE 
        \bigl[ 
        \min_{k \in \left\{ 1, \ldots, N \right\}}
        \mF(p_k)
        \bigr]
        - \mF(\hat{\pi})
        &\leq 
        \varepsilon^2\,.
    \end{align*}
\end{theorem}

\begin{proof}
    See \Cref{section:stoch-wc-rate-proof}.
\end{proof}

\begin{theorem}[Strongly convex case, Stochastic FB--GVI]
    \label{thm:stoch-sc-rate}
    Suppose that $V$ is $\alpha$-strongly convex and $\beta$-smooth, and that $\eta \leq \frac{\alpha^2}{48\beta^3}$. Then, %
    \begin{align*}
        \EE W_2^2(p_{N}, \hat{\pi})
        &\leq
        \exp\Bigl( -\frac{\alpha N \eta}{2} \Bigr)\,
        W_2^2(p_0, \hat{\pi})
        + \frac{24\beta\eta d}{\alpha}\,
        .
    \end{align*}
    In particular, for sufficiently small values of $\varepsilon^2/d$ and with
    \begin{align*}
        \eta \asymp \frac{\varepsilon^2}{\beta d}\,,\quad\text{and}\quad
        N \gtrsim \frac{\beta d}{\alpha\varepsilon^2} \log \frac{\alpha W_2^2(p_0, \hat{\pi})}{\varepsilon^2}\,,
    \end{align*}
    we obtain the guarantees
    \begin{align*}
        \alpha\, \EE W_2^2(p_N, \hat{\pi}) \leq \varepsilon^2\,,\quad\text{and}\quad
        \EE 
        \bigl[ 
        \min_{k \in \left\{ 1, \ldots, 2N \right\}}
        \mF(\iterate_k)
        \bigr]
        - \mF(\hat{\pi})
        \leq 
        \varepsilon^2\,.
    \end{align*}
\end{theorem}

\begin{proof}
    See \Cref{section:stoch-sc-rate-proof}.
\end{proof}

To our knowledge, Theorem~\ref{thm:stoch-wc-rate} is the first result to provide a complexity result in terms of the objective gap in Problem~\eqref{eq:original}, for log-smooth log-concave target distributions. Moreover, Theorem~\ref{thm:stoch-sc-rate} improves upon the state-of-the-art obtained in~\cite{lambert2022variational} for strongly log-concave target distributions. In particular, ignoring logarithmic factors, their iteration complexity (when written in a scale-invariant way) reads $\widetilde O(\frac{\beta^2 d}{\alpha^2 \varepsilon^2})$, whereas ours reads $\widetilde O(\frac{\beta d}{\alpha \varepsilon^2})$. Note that the linear dependence on the condition number $\beta/\alpha$ is to be expected for gradient descent methods.
We remark that our analysis crucially makes use of the proximal operator (the BW JKO) on the non-smooth entropy in order to obtain our improved rates.

\section{Simulations}\label{scn:exp}

In this section, via elementary simulations\footnote{Code for our experiments can be found at \url{https://github.com/mzydiao/FBGVI/blob/main/FBGVI-Experiments.ipynb}.}, we demonstrate that FBGVI is implementable, practical and competitive with the Bures–Wasserstein gradient descent (BWGD) method of~\cite{lambert2022variational}. We consider two examples:
\begin{enumerate}
    \item \textbf{Gaussian targets.} For the first experiment,
        we consider a scenario where the target density is $\pi(x) \propto \exp( -\half\, \langle (x - \mu), \Sigma^{-1}\, (x - \mu) \rangle )$,
        where
        $\mu \sim \text{Unif}([0, 1]^{10})$
        and $\Sigma^{-1} = U \diag \begin{bmatrix}
            10^{-9} & 10^{-8} & \cdots & 1
        \end{bmatrix}U^\T$, with $U \in \RR^{10\times 10}$ chosen as a uniformly random orthogonal matrix.
        In this case, we have that $\pi \in \BW(\RR^{10})$, so the solution to Problem~\eqref{eq:original}
        is precisely $\pi$,
        and furthermore we have that $\pi$ is $10^{-9}$-strongly log-concave and $1$-log-smooth.

        We run FB--GVI and stochastic FB--GVI with target potential $\pi \propto \exp(-V)$ initialized at $\iterate_0 = \mN(0, I_{10})$, where $I_{10}$ is the $10\times 10$ identity matrix.
        The step size $\eta$ is varied, and the resulting
        plots of $\log \KL{\iterate_k}{\pi}$
        for different choices of $\eta$
        are displayed in \Cref{fig:gaussian}.
        \begin{figure}
            \centering
            \includegraphics{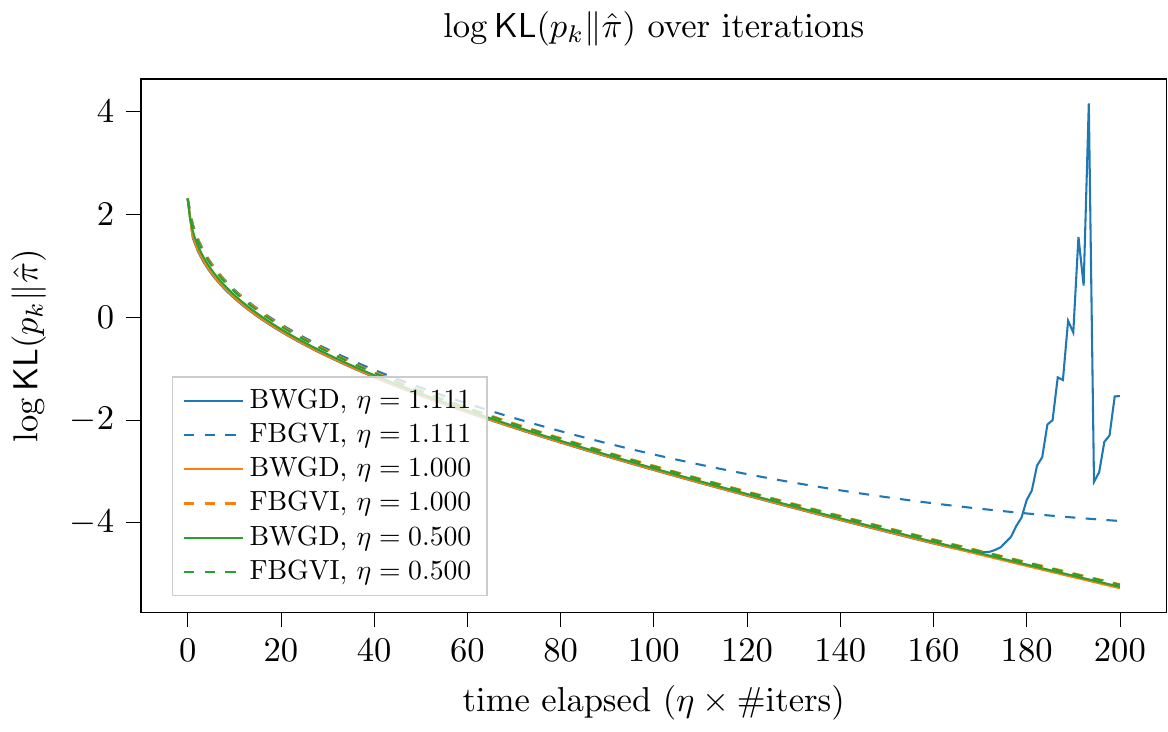}
            \includegraphics{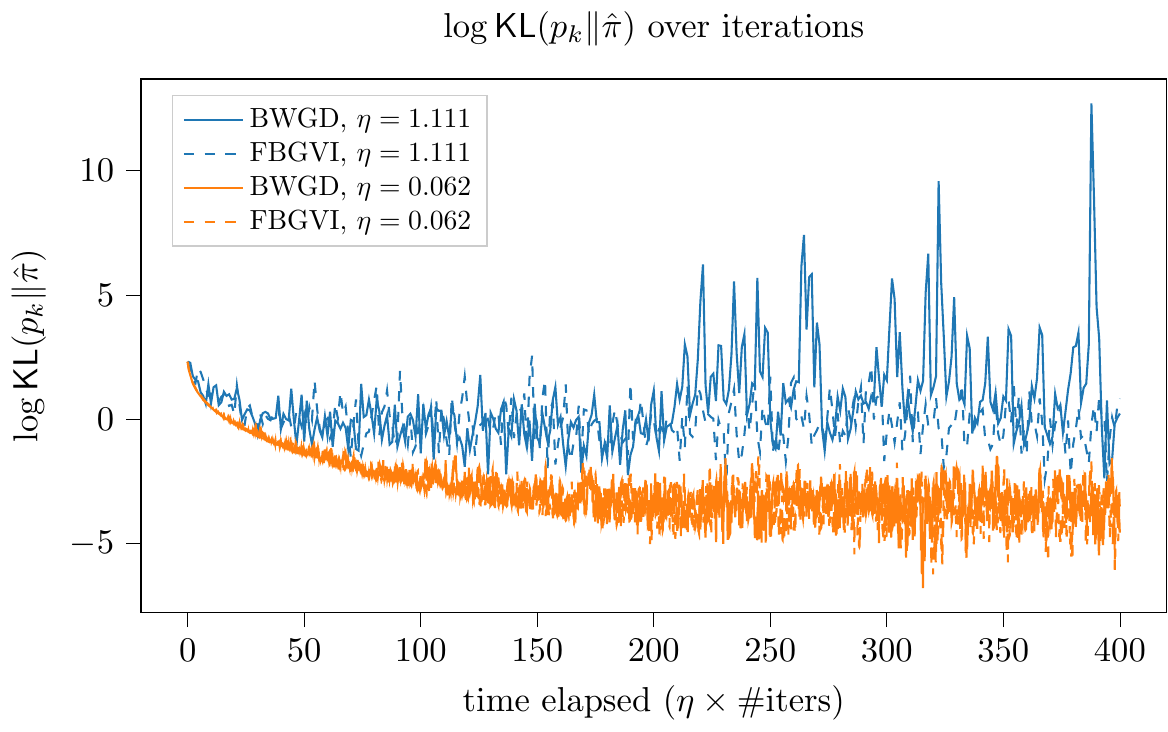}
            \caption{Gaussian target experiment: results for FB--GVI (top) and stochastic FB--GVI (bottom).}
            \label{fig:gaussian}
        \end{figure}
    \item \textbf{Bayesian logistic regression.}
        We consider the following generative model:
        given a parameter $\theta\in \RR^d$,
        we draw 
        i.i.d.\ samples $\{(X_i,Y_i)\}_{i=1}^n \in (\mathbb{R}^d\times\{0,1\})^n$
        with
        \[
        X_i \iid \mN(0, I_d)\,, 
        \quad
        Y_i \mid X_i \sim \text{Bern}(  e^{\left\langle \theta, X_i \right\rangle})\,.
        \]
        Given these samples $\left\{ (X_i, Y_i) \right\}_{i=1}^n$ and a uniform (improper) prior on $\theta$, the posterior on $\theta$ is given by
        \begin{align*}
            V(\theta) = \sum_{i=1}^{n}
            \,\bigl[ 
            \ln( 1 +  e^{\left\langle \theta, X_i \right\rangle}) - Y_i\, \left\langle \theta, X_i \right\rangle
            \bigr]\,.
        \end{align*}
        We run stochastic FB--GVI with $\pi \propto \exp (-V)$ initialized at $\iterate_0 = \mN\left( 0, I_d \right)$ with varying step sizes $\eta$.
        Since in this scenario we do not know the true
        minimizer $\hat \pi$ nor the normalization
        constant of $\pi$, we cannot directly
        compute $\KL{\iterate_k}{\pi}$ nor $W_2^2(\iterate_k, \hat \pi)$.
        However, we can still estimate the objective function 
        $
        \mF(\iterate_k)
        $
        as well as the squared BW gradient norm
        $
        \EE_{\iterate_k} \norm{\nabla_{\mathsf{BW}} \mF(\iterate_k)}^2
        $
        empirically by drawing samples from $\iterate_k$.
        For each choice of step size $\eta$, we plot our empirical estimates of $\mF(\iterate_k)$ and $\EE_{\iterate_k} \norm{\nabla_{\mathsf{BW}} \mF(\iterate_k)}^2$ over iterations
        in \Cref{fig:logistic}.
        \begin{figure}
            \centering
            \includegraphics{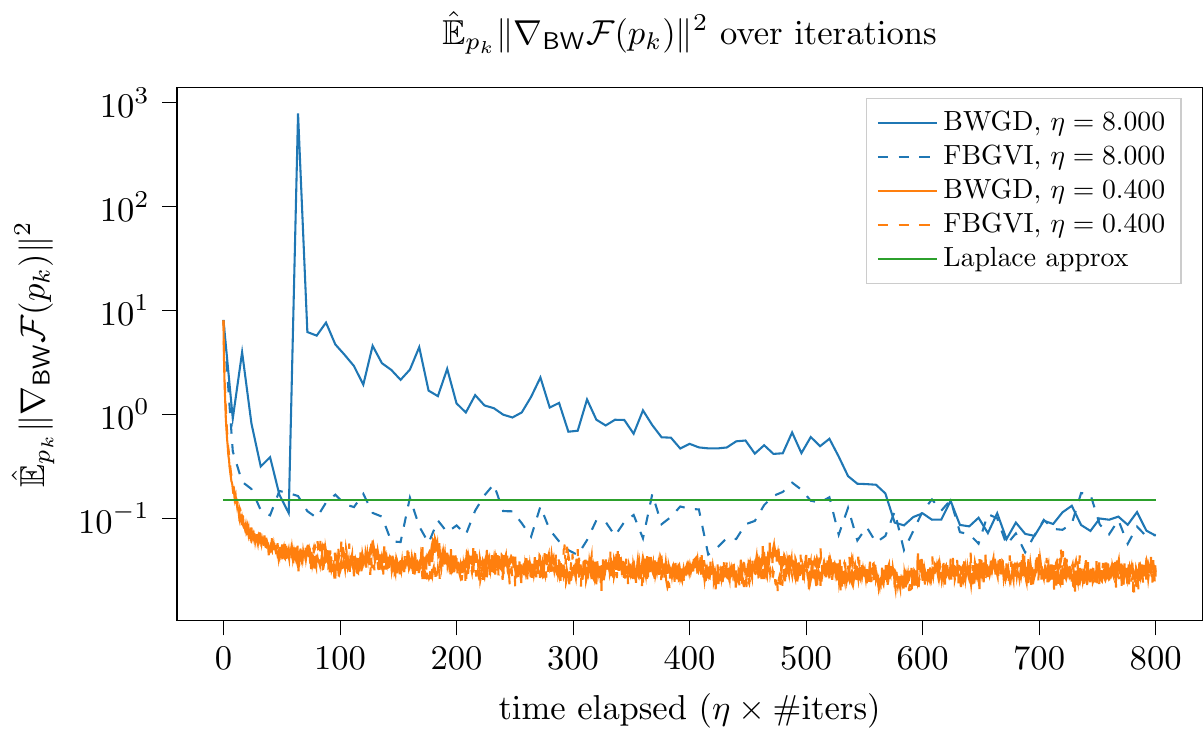}
            \includegraphics{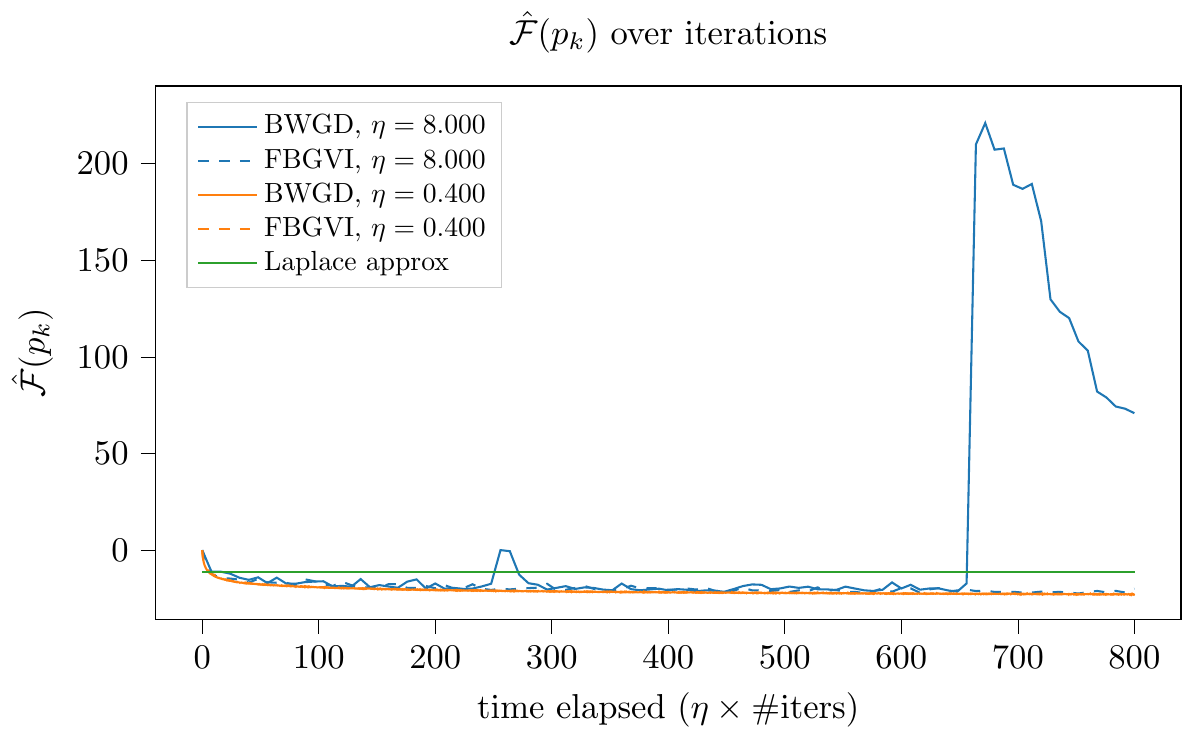}
            \caption{Bayesian logistic regression experiment: plots of $\log \hat{\EE}_{\iterate_k}\norm{\nabla_{\mathsf{BW}} \mF(\iterate_k)}^2$ (top) and of $\hat{\mF}(\iterate_k)$ (bottom) for stochastic FB--GVI.}
            \label{fig:logistic}
        \end{figure}
\end{enumerate}
Our results are provided in \Cref{fig:gaussian} and \Cref{fig:logistic}. Based on the plots, we make the following observations:
\begin{enumerate}
    \item (Stochastic) FB--GVI performs as well as BWGD, if not better. In addition, for sufficiently small $\eta$, both FB--GVI and BWGD attain lower objective than the Laplace approximation as seen in \Cref{fig:logistic}. This observation was also made in~\cite{katsevich2023approximation} for BWGD.
    \item FB--GVI is stable up to much larger step sizes than BWGD, mirroring the comparative stability of proximal gradient methods versus gradient descent methods in Euclidean space, especially when the step size is large (see, \emph{e.g.}, \cite{toulis2016towards,toulis2021proximal}).
\end{enumerate}
Our empirical results, combined with our theoretical guarantees, lend convincing evidence in support of using FB--GVI for Gaussian variational inference.

\section{Related works}\label{scn:related}

We now discuss streams of work that are closely related to ours, and place our work in the context of the larger literature on sampling and variational inference.

\paragraph{Optimization algorithms for Gaussian VI.} Algorithms for solving Gaussian VI have been considered in~\citet{paisley2012variational,ranganath2014black,lambert2022recursive}. The general approach is to parametrize the set of Gaussian distributions and to apply Euclidean optimization.
In particular, \citet{alquier2020concentration} noticed that when $\pi$ is the posterior distribution in a Bayesian logistic regression, Problem~\eqref{eq:original} becomes convex with a certain choice of parametrization. In this case, they also characterized the statistical properties of the iterates of gradient descent. 
Other settings in which the corresponding optimization problem is convex are provided in~\citet{challis2013gaussian,domke2020provable}. In particular,~\cite{domke2020provable} showed under the parameterization of~\cite{alquier2020concentration}, the Euclidean smoothness/convexity properties of the negative log-density of the model give rise to the same smoothness/convexity properties on the parameter space. However, to obtain convergence rates in the stochastic setting, one needs to control the variance of the stochastic gradient. This non-trivial task is not carried out in~\cite{domke2020provable}. In our work, the required variance control is established in Lemma~\ref{lemma:grad-error-bound-text}, and forms a crucial step in obtaining our convergence rates.

Algorithms based on natural gradient methods~\citep{zhang2018noisy,lin2019fast, lin2020handling} and normalizing flows~\citep{rezende2015variational,kingma2016improved,caterini2021variational} have also been proposed for variational inference. However, to the best of our knowledge, convergence results for such methods are lacking in the literature.

Finally, the closest related work to ours in this literature is that of~\citet{lambert2022variational}, who similarly proposed an optimization algorithm over the BW space, called Bures--Wasserstein Stochastic Gradient Descent (BW--SGD), to solve Problem~\eqref{eq:original}. Their algorithm BW--SGD relies on taking the gradient of the non-smooth entropy, and in particular they were only able to provide a (suboptimal) rate of convergence when $\pi$ is strongly log-concave.
In this work, we not only improve upon their convergence rate in the strongly log-concave case,
but also demonstrate a convergence rate for the log-concave case as well.

\paragraph{Minimization of $\mathsf{KL}$ over the Wasserstein space.} As mentioned previously, our approach has roots in the recent literature on viewing sampling methods as optimization algorithms over the Wasserstein space. %

For example, the Langevin Monte Carlo (LMC) algorithm~\citep{dalalyan2017theoretical} is an MCMC algorithm to sample from the target distribution $\pi$. The theory of Wasserstein gradient flows~\citep{ambrosio2008gradient} provides the mathematical tools to view LMC (and its many variants) as an optimization algorithm over Wasserstein space. In the case of LMC, the objective to minimize is $\KL{\cdot}{\pi}$. Therefore, one can use optimization analysis (over the Wasserstein space) to show convergence bounds for LMC~\citep{sampling-as-opt,durmus2019analysis,balasubramanian2022towards, chen2022improved, sampling-book}. 

Stein Variational Gradient Descent~\citep{SVGD,liu2017stein} is another method that can be seen as an optimization algorithm for minimizing $\KL{\cdot}{\pi}$. SVGD is a deterministic algorithm that drives the empirical distribution of a set of particles to fit $\pi$. The iterations of SVGD are computed by iterating a well chosen map $T$ such that $T - I$ belongs to a Reproducing Kernel Hilbert Space (RKHS). Little is known about the convergence rates of SVGD~\citep{duncan2019geometry, lu2019scaling, 
 chewi2020svgd, korba2020non, salim2022convergence, he2022regularized, shi2022finite}. However, there is an interesting connection between SVGD and BW--GD~\citep{lambert2022variational}: when the number of particles of SVGD tends to infinity (the ``mean-field'' limit), the iterations of BW--GD are equivalent to the iterations of SVGD if the RKHS is the set of affine functions with symmetric linear part. 
We also remark that other heuristic algorithms for particle-based optimization over Wasserstein spaces have been proposed, for example,  in~\cite{carrillo2019blob, AlvSchMro22icnn, backhoff2022stochastic, wang2022projected,yao2022mean} without any non-asymptotic convergence guarantees.

Another closely related work to ours is that of~\citet{WPGA}. In the same vein as our work, they view the objective $\KL{\cdot}{\pi}$ as a composite functional over the Wasserstein space. They propose a forward-backward algorithm, involving the JKO of the entropy, with strong convergence properties. However, they do not discuss the implementation of the JKO of the entropy. Therefore, to our knowledge, their algorithm is not implementable. On the contrary, our algorithm relies on the JKO of the entropy over the BW space, which is shown to admit a closed form.

\paragraph{(Non--smooth) manifold optimization.} Our work is also morally related to recent works developing efficient algorithms for solving non-smooth optimization problems over certain manifolds. 
For example, in the deterministic setting,~\citet{li2021weakly} analyzed the sub-gradient method,~\citet{chen2020proximal, huang2022riemannian} analyzed the proximal gradient method, ~\citet{chen2021manifold} analyzed the proximal point method, and~\citet{wang2022manifold} analyzed the proximal linear method. Stochastic versions were considered in~\citet{li2022stochastic, wang2022manifold}. We also refer to~\citet{zhang2021riemannian, hu2022constraint, peng2022riemannian, zhang2022riemannian} for other recent advances in non--smooth manifold optimization. Several of the above works consider the general setting of a Riemannian manifold.
While the BW space is a Riemannian manifold, it is also a subset of the Wasserstein space, a structure we leverage in our work to prove our convergence results.

The geometry of the BW space was investigated in~\citet{modin2017matrixdecomposition, malmonpis2018bw, bhatiajainlim2019bures}, and optimization over this space has proven to be fruitful for various applications, see, for example,~\citet{gd-bw-barycenters, averaging-on-bw, hanetal21bwgeometry, lambert2022variational, luotri22matrixfactorization, maulegrig22bwlowrank}.

\section{Conclusion}
\label{sec:ccl}
We proposed a novel optimization algorithm, (Stochastic) FB--GVI, for solving the Gaussian VI problem in~\eqref{eq:original}. We view this algorithm as performing optimization over the Bures--Wasserstein space, echoing a stream of successful works on optimization-inspired design and analysis of sampling and variational inference algorithms. Using this perspective, we also provided new or state-of-the-art convergence rates for solving~\eqref{eq:original}, depending on the regularity assumptions on $\pi$. As immediate future work, it is intriguing to study the statistical properties (consistency, normal approximation bounds, moment estimation bounds, and robustness properties) of the proposed (Stochastic) FW--GVI algorithm on various specific practical problems of interest. 

At a broader level, our work opens the door to the following question: Can we develop a rigorous algorithmic framework for general VI, \textit{i.e.}, Problem~\eqref{eq:original} where $\PG(\RR^d)$ is replaced by a different or larger set of distributions (for example, mixtures of Gaussians)? We believe that this paper provides a concrete step toward this general goal. 

\subsection*{Acknowledgements}
We thank anonymous reviewers for suggesting Theorem~\ref{thm:noiseless-wc-no-rate} and its proof, and for bringing the work of~\cite{domke2020provable} to our attention. KB was supported in part by National Science Foundation (NSF) grant DMS-2053918. SC was supported
by the NSF TRIPODS program (award DMS-2022448).

\bibliography{refs}

\begin{thebibliography}{84}
\providecommand{\natexlab}[1]{#1}
\providecommand{\url}[1]{\texttt{#1}}
\expandafter\ifx\csname urlstyle\endcsname\relax
  \providecommand{\doi}[1]{doi: #1}\else
  \providecommand{\doi}{doi: \begingroup \urlstyle{rm}\Url}\fi

\bibitem[Ahn and Chewi(2021)]{ahnchewi2021mirrorlangevin}
K.~Ahn and S.~Chewi.
\newblock Efficient constrained sampling via the mirror-{L}angevin algorithm.
\newblock In M.~Ranzato, A.~Beygelzimer, K.~Nguyen, P.~S. Liang, J.~W. Vaughan,
  and Y.~Dauphin, editors, \emph{Advances in Neural Information Processing
  Systems}, volume~34, pages 28405--28418. Curran Associates, Inc., 2021.

\bibitem[Alquier and Ridgway(2020)]{alquier2020concentration}
P.~Alquier and J.~Ridgway.
\newblock Concentration of tempered posteriors and of their variational
  approximations.
\newblock \emph{The Annals of Statistics}, 48\penalty0 (3):\penalty0
  1475--1497, 2020.

\bibitem[Altschuler et~al.(2021)Altschuler, Chewi, Gerber, and
  Stromme]{averaging-on-bw}
J.~Altschuler, S.~Chewi, P.~R. Gerber, and A.~Stromme.
\newblock Averaging on the {B}ures--{W}asserstein manifold: dimension-free
  convergence of gradient descent.
\newblock In M.~Ranzato, A.~Beygelzimer, Y.~Dauphin, P.~Liang, and J.~W.
  Vaughan, editors, \emph{Advances in Neural Information Processing Systems},
  volume~34, pages 22132--22145. Curran Associates, Inc., 2021.

\bibitem[Alvarez-Melis et~al.(2022)Alvarez-Melis, Schiff, and
  Mroueh]{AlvSchMro22icnn}
D.~Alvarez-Melis, Y.~Schiff, and Y.~Mroueh.
\newblock Optimizing functionals on the space of probabilities with input
  convex neural networks.
\newblock \emph{Transactions on Machine Learning Research}, 2022.

\bibitem[Ambrosio et~al.(2008)Ambrosio, Gigli, and
  Savar{\'e}]{ambrosio2008gradient}
L.~Ambrosio, N.~Gigli, and G.~Savar{\'e}.
\newblock \emph{Gradient flows: in metric spaces and in the space of
  probability measures}.
\newblock Springer Science \& Business Media, 2008.

\bibitem[Atchad{\'e} et~al.(2017)Atchad{\'e}, Fort, and
  Moulines]{atchade2017perturbed}
Y.~F. Atchad{\'e}, G.~Fort, and E.~Moulines.
\newblock On perturbed proximal gradient algorithms.
\newblock \emph{The Journal of Machine Learning Research}, 18\penalty0
  (1):\penalty0 310--342, 2017.

\bibitem[Backhoff-Veraguas et~al.(2022)Backhoff-Veraguas, Fontbona, Rios, and
  Tobar]{backhoff2022stochastic}
J.~Backhoff-Veraguas, J.~Fontbona, G.~Rios, and F.~Tobar.
\newblock Stochastic gradient descent in {W}asserstein space.
\newblock \emph{arXiv preprint arXiv:2201.04232}, 2022.

\bibitem[Balasubramanian et~al.(2022)Balasubramanian, Chewi, Erdogdu, Salim,
  and Zhang]{balasubramanian2022towards}
K.~Balasubramanian, S.~Chewi, M.~A. Erdogdu, A.~Salim, and S.~Zhang.
\newblock Towards a theory of non-log-concave sampling: first-order
  stationarity guarantees for {Langevin Monte Carlo}.
\newblock In \emph{Conference on Learning Theory}, pages 2896--2923. PMLR,
  2022.

\bibitem[Barber and Bishop(1997)]{barber1997ensemble}
D.~Barber and C.~Bishop.
\newblock Ensemble learning for multi-layer networks.
\newblock \emph{Advances in Neural Information Processing Systems}, 10, 1997.

\bibitem[Bauschke et~al.(2011)Bauschke, Combettes, et~al.]{bauschke2011convex}
H.~H. Bauschke, P.~L. Combettes, et~al.
\newblock \emph{Convex analysis and monotone operator theory in {H}ilbert
  spaces}, volume 408.
\newblock Springer, 2011.

\bibitem[Bernton(2018)]{bernton2018jko}
E.~Bernton.
\newblock {L}angevin {M}onte {C}arlo and {JKO} splitting.
\newblock In S.~Bubeck, V.~Perchet, and P.~Rigollet, editors, \emph{Proceedings
  of the 31st Conference on Learning Theory}, volume~75 of \emph{Proceedings of
  Machine Learning Research}, pages 1777--1798. PMLR, 06--09 Jul 2018.

\bibitem[Bhatia et~al.(2019)Bhatia, Jain, and Lim]{bhatiajainlim2019bures}
R.~Bhatia, T.~Jain, and Y.~Lim.
\newblock On the {B}ures--{W}asserstein distance between positive definite
  matrices.
\newblock \emph{Expo. Math.}, 37\penalty0 (2):\penalty0 165--191, 2019.

\bibitem[Bianchi et~al.(2019)Bianchi, Hachem, and Salim]{bianchi2019constant}
P.~Bianchi, W.~Hachem, and A.~Salim.
\newblock A constant step forward-backward algorithm involving random maximal
  monotone operators.
\newblock \emph{J. Convex Anal.}, 26\penalty0 (2):\penalty0 397--436, 2019.

\bibitem[Blei et~al.(2017)Blei, Kucukelbir, and McAuliffe]{blei2017variational}
D.~M. Blei, A.~Kucukelbir, and J.~D. McAuliffe.
\newblock Variational inference: a review for statisticians.
\newblock \emph{Journal of the American Statistical Association}, 112\penalty0
  (518):\penalty0 859--877, 2017.

\bibitem[Bottou et~al.(2018)Bottou, Curtis, and
  Nocedal]{bottou2018optimization}
L.~Bottou, F.~E. Curtis, and J.~Nocedal.
\newblock Optimization methods for large-scale machine learning.
\newblock \emph{SIAM Review}, 60\penalty0 (2):\penalty0 223--311, 2018.

\bibitem[Brascamp and Lieb(1976)]{brascamplieb}
H.~J. Brascamp and E.~H. Lieb.
\newblock On extensions of the {B}runn--{M}inkowski and
  {P}r\'{e}kopa--{L}eindler theorems, including inequalities for log concave
  functions, and with an application to the diffusion equation.
\newblock \emph{J. Functional Analysis}, 22\penalty0 (4):\penalty0 366--389,
  1976.

\bibitem[Carrillo et~al.(2019)Carrillo, Craig, and
  Patacchini]{carrillo2019blob}
J.~A. Carrillo, K.~Craig, and F.~S. Patacchini.
\newblock A blob method for diffusion.
\newblock \emph{Calculus of Variations and Partial Differential Equations},
  58:\penalty0 1--53, 2019.

\bibitem[Caterini et~al.(2021)Caterini, Cornish, Sejdinovic, and
  Doucet]{caterini2021variational}
A.~Caterini, R.~Cornish, D.~Sejdinovic, and A.~Doucet.
\newblock Variational inference with continuously-indexed normalizing flows.
\newblock In \emph{Uncertainty in Artificial Intelligence}, pages 44--53. PMLR,
  2021.

\bibitem[Challis and Barber(2013)]{challis2013gaussian}
E.~Challis and D.~Barber.
\newblock Gaussian {K}ullback--{L}eibler approximate inference.
\newblock \emph{Journal of Machine Learning Research}, 14\penalty0 (8), 2013.

\bibitem[Chen et~al.(2020)Chen, Ma, Man-Cho~So, and Zhang]{chen2020proximal}
S.~Chen, S.~Ma, A.~Man-Cho~So, and T.~Zhang.
\newblock Proximal gradient method for nonsmooth optimization over the
  {S}tiefel manifold.
\newblock \emph{SIAM Journal on Optimization}, 30\penalty0 (1):\penalty0
  210--239, 2020.

\bibitem[Chen et~al.(2021)Chen, Deng, Ma, and So]{chen2021manifold}
S.~Chen, Z.~Deng, S.~Ma, and A.~M.-C. So.
\newblock Manifold proximal point algorithms for dual principal component
  pursuit and orthogonal dictionary learning.
\newblock \emph{IEEE Transactions on Signal Processing}, 69:\penalty0
  4759--4773, 2021.

\bibitem[Chen et~al.(2022)Chen, Chewi, Salim, and Wibisono]{chen2022improved}
Y.~Chen, S.~Chewi, A.~Salim, and A.~Wibisono.
\newblock Improved analysis for a proximal algorithm for sampling.
\newblock In \emph{Conference on Learning Theory}, pages 2984--3014. PMLR,
  2022.

\bibitem[Ch{\'e}rief-Abdellatif et~al.(2019)Ch{\'e}rief-Abdellatif, Alquier,
  and Khan]{cherief2019generalization}
B.-E. Ch{\'e}rief-Abdellatif, P.~Alquier, and M.~E. Khan.
\newblock A generalization bound for online variational inference.
\newblock In \emph{Asian Conference on Machine Learning}, pages 662--677. PMLR,
  2019.

\bibitem[Chewi(2023)]{sampling-book}
S.~Chewi.
\newblock \emph{Log-concave sampling}.
\newblock 2023.
\newblock Available at \url{https://chewisinho.github.io/}.

\bibitem[Chewi et~al.(2020{\natexlab{a}})Chewi, Le~Gouic, Lu, Maunu, and
  Rigollet]{chewi2020svgd}
S.~Chewi, T.~Le~Gouic, C.~Lu, T.~Maunu, and P.~Rigollet.
\newblock {SVGD} as a kernelized {W}asserstein gradient flow of the chi-squared
  divergence.
\newblock \emph{Advances in Neural Information Processing Systems},
  33:\penalty0 2098--2109, 2020{\natexlab{a}}.

\bibitem[Chewi et~al.(2020{\natexlab{b}})Chewi, Maunu, Rigollet, and
  Stromme]{gd-bw-barycenters}
S.~Chewi, T.~Maunu, P.~Rigollet, and A.~J. Stromme.
\newblock Gradient descent algorithms for {B}ures--{W}asserstein barycenters.
\newblock In J.~Abernethy and S.~Agarwal, editors, \emph{Proceedings of Thirty
  Third Conference on Learning Theory}, volume 125 of \emph{Proceedings of
  Machine Learning Research}, pages 1276--1304. PMLR, 09--12 Jul
  2020{\natexlab{b}}.

\bibitem[Dalalyan(2017)]{dalalyan2017theoretical}
A.~S. Dalalyan.
\newblock Theoretical guarantees for approximate sampling from smooth and
  log-concave densities.
\newblock \emph{Journal of the Royal Statistical Society: Series B (Statistical
  Methodology)}, 79\penalty0 (3):\penalty0 651--676, 2017.

\bibitem[Domke(2020)]{domke2020provable}
J.~Domke.
\newblock Provable smoothness guarantees for black-box variational inference.
\newblock In \emph{International Conference on Machine Learning}, pages
  2587--2596. PMLR, 2020.

\bibitem[Duncan et~al.(2019)Duncan, N{\"u}sken, and
  Szpruch]{duncan2019geometry}
A.~Duncan, N.~N{\"u}sken, and L.~Szpruch.
\newblock On the geometry of {S}tein variational gradient descent.
\newblock \emph{arXiv preprint arXiv:1912.00894}, 2019.

\bibitem[Durmus et~al.(2019)Durmus, Majewski, and
  Miasojedow]{durmus2019analysis}
A.~Durmus, S.~Majewski, and B.~Miasojedow.
\newblock Analysis of {L}angevin {M}onte {C}arlo via convex optimization.
\newblock \emph{The Journal of Machine Learning Research}, 20\penalty0
  (1):\penalty0 2666--2711, 2019.

\bibitem[Gorbunov et~al.(2020)Gorbunov, Hanzely, and
  Richt{\'a}rik]{gorbunov2020unified}
E.~Gorbunov, F.~Hanzely, and P.~Richt{\'a}rik.
\newblock A unified theory of {SGD}: variance reduction, sampling, quantization
  and coordinate descent.
\newblock In \emph{International Conference on Artificial Intelligence and
  Statistics}, pages 680--690. PMLR, 2020.

\bibitem[Han et~al.(2021)Han, Mishra, Jawanpuria, and Gao]{hanetal21bwgeometry}
A.~Han, B.~Mishra, P.~Jawanpuria, and J.~Gao.
\newblock On {R}iemannian optimization over positive definite matrices with the
  {B}ures--{W}asserstein geometry.
\newblock In A.~Beygelzimer, Y.~Dauphin, P.~Liang, and J.~W. Vaughan, editors,
  \emph{Advances in Neural Information Processing Systems}, 2021.

\bibitem[He et~al.(2022)He, Balasubramanian, Sriperumbudur, and
  Lu]{he2022regularized}
Y.~He, K.~Balasubramanian, B.~K. Sriperumbudur, and J.~Lu.
\newblock Regularized {S}tein variational gradient flow.
\newblock \emph{arXiv preprint arXiv:2211.07861}, 2022.

\bibitem[Honkela and Valpola(2004)]{honkela2004unsupervised}
A.~Honkela and H.~Valpola.
\newblock Unsupervised variational {B}ayesian learning of nonlinear models.
\newblock \emph{Advances in Neural Information Processing Systems}, 17, 2004.

\bibitem[Hu et~al.(2022)Hu, Xiao, Liu, and Toh]{hu2022constraint}
X.~Hu, N.~Xiao, X.~Liu, and K.-C. Toh.
\newblock A constraint dissolving approach for nonsmooth optimization over the
  {S}tiefel manifold.
\newblock \emph{arXiv preprint arXiv:2205.10500}, 2022.

\bibitem[Huang and Wei(2022)]{huang2022riemannian}
W.~Huang and K.~Wei.
\newblock Riemannian proximal gradient methods.
\newblock \emph{Mathematical Programming}, 194\penalty0 (1-2):\penalty0
  371--413, 2022.

\bibitem[Jordan et~al.(1998)Jordan, Kinderlehrer, and
  Otto]{jordan1998variational}
R.~Jordan, D.~Kinderlehrer, and F.~Otto.
\newblock The variational formulation of the {F}okker--{P}lanck equation.
\newblock \emph{SIAM Journal on Mathematical Analysis}, 29\penalty0
  (1):\penalty0 1--17, 1998.

\bibitem[Kasprzak et~al.(2022)Kasprzak, Giordano, and
  Broderick]{kasprzak2022good}
M.~J. Kasprzak, R.~Giordano, and T.~Broderick.
\newblock How good is your {G}aussian approximation of the posterior?
  {F}inite-sample computable error bounds for a variety of useful divergences.
\newblock \emph{arXiv preprint arXiv:2209.14992}, 2022.

\bibitem[Katsevich and Rigollet(2023)]{katsevich2023approximation}
A.~Katsevich and P.~Rigollet.
\newblock On the approximation accuracy of {G}aussian variational inference.
\newblock \emph{arXiv preprint arXiv:2301.02168}, 2023.

\bibitem[Kingma et~al.(2016)Kingma, Salimans, Jozefowicz, Chen, Sutskever, and
  Welling]{kingma2016improved}
D.~P. Kingma, T.~Salimans, R.~Jozefowicz, X.~Chen, I.~Sutskever, and
  M.~Welling.
\newblock Improved variational inference with inverse autoregressive flow.
\newblock \emph{Advances in Neural Information Processing Systems}, 29, 2016.

\bibitem[Knoblauch et~al.(2022)Knoblauch, Jewson, and
  Damoulas]{knoblauch2022optimization}
J.~Knoblauch, J.~Jewson, and T.~Damoulas.
\newblock An optimization-centric view on {B}ayes’ rule: reviewing and
  generalizing variational inference.
\newblock \emph{Journal of Machine Learning Research}, 23\penalty0
  (132):\penalty0 1--109, 2022.

\bibitem[Korba et~al.(2020)Korba, Salim, Arbel, Luise, and
  Gretton]{korba2020non}
A.~Korba, A.~Salim, M.~Arbel, G.~Luise, and A.~Gretton.
\newblock A non-asymptotic analysis for {S}tein variational gradient descent.
\newblock \emph{Advances in Neural Information Processing Systems},
  33:\penalty0 4672--4682, 2020.

\bibitem[Kushner and Yin(2003)]{kushner2003stochastic}
H.~Kushner and G.~Yin.
\newblock Stochastic approximation and recursive algorithms.
\newblock In \emph{Stochastic Modelling and Applied Probability}, volume~35.
  Springer-Verlag NY, 2003.

\bibitem[Lambert et~al.(2022{\natexlab{a}})Lambert, Bonnabel, and
  Bach]{lambert2022recursive}
M.~Lambert, S.~Bonnabel, and F.~Bach.
\newblock The recursive variational {G}aussian approximation ({R-VGA}).
\newblock \emph{Statistics and Computing}, 32\penalty0 (1):\penalty0 10,
  2022{\natexlab{a}}.

\bibitem[Lambert et~al.(2022{\natexlab{b}})Lambert, Chewi, Bach, Bonnabel, and
  Rigollet]{lambert2022variational}
M.~Lambert, S.~Chewi, F.~Bach, S.~Bonnabel, and P.~Rigollet.
\newblock Variational inference via {W}asserstein gradient flows.
\newblock In A.~H. Oh, A.~Agarwal, D.~Belgrave, and K.~Cho, editors,
  \emph{Advances in Neural Information Processing Systems}, 2022{\natexlab{b}}.

\bibitem[Li et~al.(2022)Li, Balasubramanian, and Ma]{li2022stochastic}
J.~Li, K.~Balasubramanian, and S.~Ma.
\newblock Stochastic zeroth-order {R}iemannian derivative estimation and
  optimization.
\newblock \emph{Mathematics of Operations Research}, 2022.

\bibitem[Li et~al.(2021)Li, Chen, Deng, Qu, Zhu, and Man-Cho~So]{li2021weakly}
X.~Li, S.~Chen, Z.~Deng, Q.~Qu, Z.~Zhu, and A.~Man-Cho~So.
\newblock Weakly convex optimization over {S}tiefel manifold using {R}iemannian
  subgradient-type methods.
\newblock \emph{SIAM Journal on Optimization}, 31\penalty0 (3):\penalty0
  1605--1634, 2021.

\bibitem[Lin et~al.(2019)Lin, Khan, and Schmidt]{lin2019fast}
W.~Lin, M.~E. Khan, and M.~Schmidt.
\newblock Fast and simple natural-gradient variational inference with mixture
  of exponential-family approximations.
\newblock In \emph{International Conference on Machine Learning}, pages
  3992--4002. PMLR, 2019.

\bibitem[Lin et~al.(2020)Lin, Schmidt, and Khan]{lin2020handling}
W.~Lin, M.~Schmidt, and M.~E. Khan.
\newblock Handling the positive-definite constraint in the {B}ayesian learning
  rule.
\newblock In \emph{International Conference on Machine Learning}, pages
  6116--6126. PMLR, 2020.

\bibitem[Liu(2017)]{liu2017stein}
Q.~Liu.
\newblock Stein variational gradient descent as gradient flow.
\newblock \emph{Advances in Neural Information Processing Systems}, 30, 2017.

\bibitem[Liu and Wang(2016)]{SVGD}
Q.~Liu and D.~Wang.
\newblock Stein variational gradient descent: a general purpose {B}ayesian
  inference algorithm.
\newblock In D.~Lee, M.~Sugiyama, U.~Luxburg, I.~Guyon, and R.~Garnett,
  editors, \emph{Advances in Neural Information Processing Systems}, volume~29.
  Curran Associates, Inc., 2016.

\bibitem[Lu et~al.(2019)Lu, Lu, and Nolen]{lu2019scaling}
J.~Lu, Y.~Lu, and J.~Nolen.
\newblock Scaling limit of the {S}tein variational gradient descent: the mean
  field regime.
\newblock \emph{SIAM Journal on Mathematical Analysis}, 51\penalty0
  (2):\penalty0 648--671, 2019.

\bibitem[Luo and Trillos(2022)]{luotri22matrixfactorization}
Y.~Luo and N.~G. Trillos.
\newblock Nonconvex matrix factorization is geodesically convex: global
  landscape analysis for fixed-rank matrix optimization from a {R}iemannian
  perspective.
\newblock \emph{arXiv preprint arXiv:2209.15130}, 2022.

\bibitem[Malag\`o et~al.(2018)Malag\`o, Montrucchio, and
  Pistone]{malmonpis2018bw}
L.~Malag\`o, L.~Montrucchio, and G.~Pistone.
\newblock Wasserstein {R}iemannian geometry of {G}aussian densities.
\newblock \emph{Inf. Geom.}, 1\penalty0 (2):\penalty0 137--179, 2018.

\bibitem[Maunu et~al.(2022)Maunu, Le~Gouic, and Rigollet]{maulegrig22bwlowrank}
T.~Maunu, T.~Le~Gouic, and P.~Rigollet.
\newblock Bures--{W}asserstein barycenters and low-rank matrix recovery.
\newblock \emph{arXiv preprint arXiv:2210.14671}, 2022.

\bibitem[Modin(2017)]{modin2017matrixdecomposition}
K.~Modin.
\newblock Geometry of matrix decompositions seen through optimal transport and
  information geometry.
\newblock \emph{J. Geom. Mech.}, 9\penalty0 (3):\penalty0 335--390, 2017.

\bibitem[Naldi and Savar\'{e}(2021)]{naldi2021weak}
E.~Naldi and G.~Savar\'{e}.
\newblock Weak topology and {O}pial property in {W}asserstein spaces, with
  applications to gradient flows and proximal point algorithms of geodesically
  convex functionals.
\newblock \emph{Atti Accad. Naz. Lincei Rend. Lincei Mat. Appl.}, 32\penalty0
  (4):\penalty0 725--750, 2021.

\bibitem[Olkin and Pukelsheim(1982)]{olkpuk1982otgaussian}
I.~Olkin and F.~Pukelsheim.
\newblock The distance between two random vectors with given dispersion
  matrices.
\newblock \emph{Linear Algebra Appl.}, 48:\penalty0 257--263, 1982.

\bibitem[Opper and Archambeau(2009)]{opper2009variational}
M.~Opper and C.~Archambeau.
\newblock The variational {G}aussian approximation revisited.
\newblock \emph{Neural Computation}, 21\penalty0 (3):\penalty0 786--792, 2009.

\bibitem[Otto(2001)]{porous-medium-eqn}
F.~Otto.
\newblock The geometry of dissipative evolution equations: the porous medium
  equation.
\newblock \emph{Comm. Partial Differential Equations}, 26\penalty0
  (1-2):\penalty0 101--174, 2001.

\bibitem[Paisley et~al.(2012)Paisley, Blei, and Jordan]{paisley2012variational}
J.~W. Paisley, D.~M. Blei, and M.~I. Jordan.
\newblock Variational {B}ayesian inference with stochastic search.
\newblock In \emph{Proceedings of the 29th International Conference on Machine
  Learning, {ICML} 2012, Edinburgh, Scotland, UK, June 26 - July 1, 2012}.
  Omnipress, 2012.

\bibitem[Peng et~al.(2022)Peng, Wu, Hu, and Deng]{peng2022riemannian}
Z.~Peng, W.-H. Wu, J.~Hu, and K.-K. Deng.
\newblock Riemannian smoothing gradient type algorithms for nonsmooth
  optimization problem on manifolds.
\newblock \emph{arXiv preprint arXiv:2212.03526}, 2022.

\bibitem[Pleiss et~al.(2020)Pleiss, Jankowiak, Eriksson, Damle, and
  Gardner]{pleiss2020fast}
G.~Pleiss, M.~Jankowiak, D.~Eriksson, A.~Damle, and J.~Gardner.
\newblock {Fast matrix square roots with applications to Gaussian processes and
  Bayesian optimization}.
\newblock \emph{Advances in Neural Information Processing Systems},
  33:\penalty0 22268--22281, 2020.

\bibitem[Quiroz et~al.(2022)Quiroz, Nott, and Kohn]{quiroz2022gaussian}
M.~Quiroz, D.~J. Nott, and R.~Kohn.
\newblock Gaussian variational approximations for high-dimensional state space
  models.
\newblock \emph{Bayesian Analysis}, 1\penalty0 (1):\penalty0 1--28, 2022.

\bibitem[Ranganath et~al.(2014)Ranganath, Gerrish, and
  Blei]{ranganath2014black}
R.~Ranganath, S.~Gerrish, and D.~Blei.
\newblock Black--box variational inference.
\newblock In \emph{Artificial Intelligence and Statistics}, pages 814--822.
  PMLR, 2014.

\bibitem[Rezende and Mohamed(2015)]{rezende2015variational}
D.~Rezende and S.~Mohamed.
\newblock Variational inference with normalizing flows.
\newblock In \emph{International Conference on Machine Learning}, pages
  1530--1538. PMLR, 2015.

\bibitem[Salim et~al.(2020)Salim, Korba, and Luise]{WPGA}
A.~Salim, A.~Korba, and G.~Luise.
\newblock The {W}asserstein proximal gradient algorithm.
\newblock In H.~Larochelle, M.~Ranzato, R.~Hadsell, M.~Balcan, and H.~Lin,
  editors, \emph{Advances in Neural Information Processing Systems}, volume~33,
  pages 12356--12366. Curran Associates, Inc., 2020.

\bibitem[Salim et~al.(2022)Salim, Sun, and Richtarik]{salim2022convergence}
A.~Salim, L.~Sun, and P.~Richtarik.
\newblock A convergence theory for {SVGD} in the population limit under
  {T}alagrand’s inequality {T1}.
\newblock In \emph{International Conference on Machine Learning}, pages
  19139--19152. PMLR, 2022.

\bibitem[Santambrogio(2015)]{santambrogio-opt-applied}
F.~Santambrogio.
\newblock \emph{Optimal transport for applied mathematicians}, volume~87 of
  \emph{Progress in Nonlinear Differential Equations and their Applications}.
\newblock Birkh\"{a}user/Springer, Cham, 2015.
\newblock Calculus of variations, PDEs, and modeling.

\bibitem[Seeger(1999)]{seeger1999bayesian}
M.~Seeger.
\newblock Bayesian model selection for support vector machines, {G}aussian
  processes and other kernel classifiers.
\newblock \emph{Advances in Neural Information Processing Systems}, 12, 1999.

\bibitem[Shi and Mackey(2022)]{shi2022finite}
J.~Shi and L.~Mackey.
\newblock A finite-particle convergence rate for {S}tein variational gradient
  descent.
\newblock \emph{arXiv preprint arXiv:2211.09721}, 2022.

\bibitem[Song et~al.(2022)Song, Sebe, and Wang]{song2022fast}
Y.~Song, N.~Sebe, and W.~Wang.
\newblock Fast differentiable matrix square root and inverse square root.
\newblock \emph{IEEE Transactions on Pattern Analysis and Machine
  Intelligence}, 2022.

\bibitem[Spokoiny(2022)]{spokoiny2022dimension}
V.~Spokoiny.
\newblock Dimension free non-asymptotic bounds on the accuracy of high
  dimensional {L}aplace approximation.
\newblock \emph{arXiv preprint arXiv:2204.11038}, 2022.

\bibitem[Toulis et~al.(2016)Toulis, Tran, and Airoldi]{toulis2016towards}
P.~Toulis, D.~Tran, and E.~Airoldi.
\newblock Towards stability and optimality in stochastic gradient descent.
\newblock In \emph{Artificial Intelligence and Statistics}, pages 1290--1298.
  PMLR, 2016.

\bibitem[Toulis et~al.(2021)Toulis, Horel, and Airoldi]{toulis2021proximal}
P.~Toulis, T.~Horel, and E.~M. Airoldi.
\newblock The proximal {Robbins--Monro} method.
\newblock \emph{Journal of the Royal Statistical Society Series B: Statistical
  Methodology}, 83\penalty0 (1):\penalty0 188--212, 2021.

\bibitem[Van~der Vaart(2000)]{van2000asymptotic}
A.~W. Van~der Vaart.
\newblock \emph{Asymptotic statistics}, volume~3.
\newblock Cambridge University Press, 2000.

\bibitem[Villani(2003)]{villani2003}
C.~Villani.
\newblock \emph{Topics in optimal transportation}, volume~58 of \emph{Graduate
  Studies in Mathematics}.
\newblock American Mathematical Society, Providence, RI, 2003.

\bibitem[Wang et~al.(2022{\natexlab{a}})Wang, Chen, and Li]{wang2022projected}
Y.~Wang, P.~Chen, and W.~Li.
\newblock Projected {W}asserstein gradient descent for high-dimensional
  {B}ayesian inference.
\newblock \emph{SIAM/ASA Journal on Uncertainty Quantification}, 10\penalty0
  (4):\penalty0 1513--1532, 2022{\natexlab{a}}.

\bibitem[Wang et~al.(2022{\natexlab{b}})Wang, Liu, Chen, Ma, Xue, and
  Zhao]{wang2022manifold}
Z.~Wang, B.~Liu, S.~Chen, S.~Ma, L.~Xue, and H.~Zhao.
\newblock A manifold proximal linear method for sparse spectral clustering with
  application to single-cell {RNA} sequencing data analysis.
\newblock \emph{INFORMS Journal on Optimization}, 4\penalty0 (2):\penalty0
  200--214, 2022{\natexlab{b}}.

\bibitem[Wibisono(2018)]{sampling-as-opt}
A.~Wibisono.
\newblock Sampling as optimization in the space of measures: the {L}angevin
  dynamics as a composite optimization problem.
\newblock In S.~Bubeck, V.~Perchet, and P.~Rigollet, editors, \emph{Proceedings
  of the 31st Conference on Learning Theory}, volume~75 of \emph{Proceedings of
  Machine Learning Research}, pages 2093--3027. PMLR, 06--09 Jul 2018.

\bibitem[Yao and Yang(2022)]{yao2022mean}
R.~Yao and Y.~Yang.
\newblock Mean field variational inference via {W}asserstein gradient flow.
\newblock \emph{arXiv preprint arXiv:2207.08074}, 2022.

\bibitem[Zhang et~al.(2021)Zhang, Chen, and Ma]{zhang2021riemannian}
C.~Zhang, X.~Chen, and S.~Ma.
\newblock A {R}iemannian smoothing steepest descent method for non-{L}ipschitz
  optimization on submanifolds.
\newblock \emph{arXiv preprint arXiv:2104.04199}, 2021.

\bibitem[Zhang and Davanloo~Tajbakhsh(2022)]{zhang2022riemannian}
D.~Zhang and S.~Davanloo~Tajbakhsh.
\newblock Riemannian stochastic variance-reduced cubic regularized {N}ewton
  method for submanifold optimization.
\newblock \emph{Journal of Optimization Theory and Applications}, pages 1--38,
  2022.

\bibitem[Zhang et~al.(2018)Zhang, Sun, Duvenaud, and Grosse]{zhang2018noisy}
G.~Zhang, S.~Sun, D.~Duvenaud, and R.~Grosse.
\newblock Noisy natural gradient as variational inference.
\newblock In \emph{International Conference on Machine Learning}, pages
  5852--5861. PMLR, 2018.

\end{thebibliography}

\pagebreak

\appendix

\section{Differential calculus over the BW space}
\label{sec:bw-grad-comp}
In this section, we derive a formula for the BW gradient of a generic functional,
giving us the tools necessary to define the updates of our forward-backward algorithm.
In doing so, we demonstrate computation rules for differentiating a functional $\mF \colon \PG(\RR^d)\to \RR$ along a curve of measures $(\mu_t)_{t\geq 0 } \subseteq \PG(\RR^d)$,
which will be helpful for our proofs of convexity and smoothness inequalities later on.
Our derivation relies on specializing Otto calculus, which deals with the Wasserstein space $\mP_2(\RR^d)$, to the BW space $\PG(\RR^d)$.

\subsection{Background on Otto calculus}

We first give an informal overview of the computation rules of Otto calculus \citep{porous-medium-eqn},
which endows the Wasserstein space $\mP_2(\RR^d)$ with a formal Riemannian structure.
We refer to~\citet{ambrosio2008gradient} for a more rigorous development of the mathematical theory.

Let $\mu$ be an arbitrary element of $\mP_2(\RR^d)$ admitting a density w.r.t.\ Lebesgue measure.
The tangent space $\Tan_\mu \mP_2(\RR^d)$ is identified as the space of gradients of scalar functions on $\RR^d$, i.e.,
\[
\Tan_\mu \mP_2(\RR^d) = \overline{\{\nabla \psi \mid \psi \in \mathcal C_{\rm c}^\infty(\RR^d)\}}^{L^2(\mu)}\,.
\]
For a functional $\mF \colon \mP_2(\RR^d) \to \RR$,
we can formally define its $W_2$ gradient at $\mu$
as the %
mapping $\nabla_{W_2} \mF(\mu) \in \Tan_{\mu} \mP_2(\RR^d)$
satisfying
\begin{align*}
    \partial_t|_{t=0} \mF(\mu_t) = \left\langle \nabla_{W_2} \mF(\mu), v_0 \right\rangle_{\mu}\,,
\end{align*}
for any
sufficiently regular
curve of measures $(\mu_t)_{t\in\RR} \subseteq \mP_2(\RR^d)$
with $\mu_0 = \mu$
and velocity vector fields $(v_t)_{t\in\RR}$ with $v_t \in L^2(\mu_t)$ for a.e.\ $t$ satisfying the continuity
equation
\begin{align}
    \partial_t \mu_t + \Div(\mu_t v_t) = 0\,.
    \label{eqn:continuity}
\end{align}
In fact,
we can compute this $W_2$ gradient via direct identification.
Let $\delta \mF(\mu) \colon \RR^d \to \RR$ denote a \textit{first variation} of $\mF$ at $\mu$~\citep[see][Chapter 7]{santambrogio-opt-applied},
for which
\begin{align*}
    \partial_{t} |_{t=0} \mF(\mu_t) = \int \delta \mF(\mu)\, \partial_t |_{t=0} \mu_t\,.
\end{align*}
Then, by \Cref{eqn:continuity} and integration by parts,
\begin{align*}
    \partial_{t} |_{t=0} \mF(\mu_t) 
    &= \int \delta \mF(\mu)\, \partial_t |_{t=0} \mu_t
    = -\int \delta \mF(\mu) \Div(\mu v_0)
    = \int \left\langle \nabla \delta \mF(\mu), v_0 \right\rangle \dd{\mu}
    = \left\langle \nabla \delta \mF(\mu), v_0 \right\rangle_{\mu}\,.
\end{align*}
Hence, we conclude that 
\begin{align}
    \nabla_{W_2} \mF(\mu) &\equiv \nabla \delta \mF(\mu)\,.
    \label{eqn:w2-first-variation}
\end{align}

Now we turn our attention to the \ac{BW} space.
The \ac{BW} space $\mathsf{BW}(\RR^d)$ is a submanifold of $\mP_2(\RR^d)$ \citep{porous-medium-eqn, lambert2022variational},
and hence inherits the formal Riemannian structure described above.

Let $\mu$ be an arbitrary element of $\mathsf{BW}(\RR^d)$.
The tangent space $\Tan_\mu \PG(\RR^d)$ is identified as the space of affine functions on $\RR^d$ with symmetric linear term, i.e.,
\[
\Tan_\mu \PG(\RR^d) = \{x \mapsto b + S\,(x - m_\mu) \mid b \in \RR^d,\, S\in \mathbf{S}^d\}\,.
\]
In analogy to the above,
for a functional $\mF \colon \mathsf{BW}(\RR^d) \to \RR      $,
we can formally define its \ac{BW} gradient at $\mu$
as the %
element $\nabla_{\mathsf{BW}} \mF(\mu) \in \Tan_{\mu} \mathsf{BW}(\RR^d)$
satisfying
\begin{align*}
    \partial_t|_{t=0} \mF(\mu_t) = \left\langle \nabla_{\mathsf{BW}} \mF(\mu), v_0 \right\rangle_{\mu}\,,
\end{align*}
for any curve of measures $(\mu_t)_{t\in\RR}\subseteq \mathsf{BW}(\RR^d)$
with $\mu_0 = \mu$
and velocity vector fields $(v_t)_{t\in \RR}$, with each $v_t$ an affine map, satisfying
\Cref{eqn:continuity}.
Using \Cref{eqn:w2-first-variation} and integration by parts, we
compute an expression for the BW gradient of $\mF$ in the following subsection.

\subsection{BW gradient calculation}

The BW gradient of a functional $\mF \colon \mathsf{BW}(\RR^d)\to \RR      $ can be derived analogously to~\citet[Section~C.1]{lambert2022variational}.
We present the derivation here for completeness, and in doing so we obtain a formula for the rate of change of $\mF$ along a curve of Gaussians for which the corresponding velocity vector fields are affine maps with linear parts which are not necessarily symmetric; this will play a role in later proofs.
The key idea is to use integration by parts repeatedly, exploiting
the fact that the gradient of a Gaussian density is simply
that same density multiplied by an affine term.

\begin{lemma}\label{lemma:bw-grad-F}
     Let $\mF \colon \mP_2(\RR^d)\to \RR$
     be a functional on the Wasserstein space
     with first variation $\delta \mF$.
     Then, for $\mu \in \PG(\RR^d)$, we have that $\nabla_{\mathsf{BW}} \mF(\mu)$ is given by
     \begin{align*}
         \nabla_{\mathsf{BW}} \mF(\mu) \colon x \mapsto 
         (\EE_{\mu} \nabla^2 \delta \mF)(x - m_\mu) + \EE_\mu \nabla \delta \mF\,.
     \end{align*}
 \end{lemma}

\begin{proof}
    Let $(\mu_t)_{t \in \RR} \subseteq \mathsf{BW}(\RR^d)$ be a
    regular curve of Gaussians with $\mu_0 = \mu$ and $(v_t)_{t\in\RR}$ be a family of affine maps satisfying \Cref{eqn:continuity}.
    Furthermore, suppose that $v_0$ is given by
    \[v_0 \colon x \mapsto a + M\,(x - m_{\mu})\,, \qquad (a, M) \in \RR^d \times \RR^{d\times d}\,,\]
    and that
    $\nabla_{\mathsf{BW}} \mF(\mu) \in \Tan_\mu \mathsf{BW}(\RR^d)$ is given by
    \[\nabla_{\mathsf{BW}} \mF(\mu) \colon x \mapsto b_\mF + S_\mF\,(x - m_\mu)\,, \qquad (b_\mF, S_\mF) \in \RR^d \times \mathbf{S}^d\,.\]
    Letting $X\sim \mu$, we find that
    \begin{align*}
        \left\langle \nabla_{\mathsf{BW}} \mF(\mu), v_0 \right\rangle_\mu
        &=
        \EE \left\langle b_\mF + S_\mF\,(X - m_{\mu}), a + M\,(X - m_{\mu}) \right\rangle\\
        &= 
        \left\langle b_\mF, a \right\rangle + 
        \EE \left\langle S_\mF\,(X - m_{\mu}), M\,(X - m_{\mu}) \right\rangle\\
        &= 
        \left\langle b_\mF, a \right\rangle + 
        \EE\langle S_\mF, M\,(X - m_{\mu})(X - m_{\mu})^\T \rangle
        \\
        &= 
        \left\langle b_\mF, a \right\rangle + 
        \left\langle S_\mF, M\Sigma_{\mu} \right\rangle
        \\
        &= 
        \left\langle b_\mF, a \right\rangle + 
        \langle S_\mF, \Sigma_{\mu} M^\T \rangle\,.
        \ctag{since $S_\mF = S_\mF^\T$ and $\langle A, B \rangle = \langle A^\T, B^\T \rangle$}
    \end{align*}
    On the other hand,
    from the definition of the $W_2$ gradient, we obtain that
    \begin{align*}
        \partial_t|_{t=0} \mF(\mu_t)
        &=
        \left\langle \nabla_{W_2} \mF(\mu), v_0 \right\rangle_\mu
        \ctag{definition of $\nabla_{W_2} \mF$}
        \\
        &=
        \left\langle \nabla \delta \mF(\mu), v_0 \right\rangle_{\mu}
        \ctag{by \Cref{eqn:w2-first-variation}}
        \\
        &= 
        \EE \left\langle \nabla \delta \mF(X), a+ M\,(X - \EE X) \right\rangle\\
        &= 
        \EE \left\langle \nabla \delta \mF(X), a \right\rangle 
        +
        \EE\langle \Sigma_\mu M^\T \nabla \delta \mF(X), \Sigma_\mu^{-1}\, (X - \EE X)\rangle
        \\
        &= 
        \left\langle \EE \nabla \delta \mF(X), a \right\rangle 
        -\int
        \langle \Sigma_\mu M^\T \nabla \delta \mF, \nabla \mu \rangle
        \ctag{since $\nabla \mu(x) = -\mu(x)\,\Sigma_\mu\,(x - \EE X)$}
        \\
        &= 
        \left\langle \EE \nabla \delta \mF(X), a \right\rangle 
        +
        \EE
        [\Div(\Sigma_\mu M^\T \nabla \delta \mF)(X)]
        \ctag{integration by parts}\\
        &= 
        \left\langle \EE \nabla \delta \mF(X), a \right\rangle 
        +
        \langle \EE_\mu \nabla^2 \delta \mF(X), \Sigma_\mu M^\T \rangle
        \,.
    \end{align*}
    Hence, by direct identification, we conclude that
    \begin{align*}
        (b_\mF, S_\mF) = (\EE_\mu \nabla \delta \mF,\, \EE_\mu \nabla^2 \delta \mF)\,,
    \end{align*}
    proving our desired result.
\end{proof}

\subsection{Examples of BW gradients and stationary condition for Problem~\eqref{eq:original}}
\label{sec:BWgradpot}

Consider the functional $\mF = \mV + \mH$ defined by the sum of the potential (associated to the function $V$) and the entropy, and recall that Problem~\eqref{eq:original} is equivalent to minimizing $\mF$ over the BW space, \textit{i.e.}, solving Problem~\eqref{eq:original2}. Using \citet[Section C.1]{lambert2022variational}, we have the following formulas for the BW gradients of $\mV$ and $\mH$.
\begin{equation}
    \begin{split}
    \nabla_{\BW} \mV(\mu) &: x \mapsto \EE_\mu\nabla V + (\EE_\mu \nabla^2 V)(x-m_\mu)\,, \\
    \nabla_{\BW} \mH(\mu) &: x \mapsto -\Sigma_\mu^{-1}\,(x-m_\mu)\,.
    \end{split}
    \label{eqn:bw-grads}
\end{equation}
We can also derive the above formulas from Lemma~\ref{lemma:bw-grad-F}.

Moreover, by the proof of~\Cref{lemma:bw-grad-F}, we can compute $\partial_t \mF(\mu_t) = \langle \nabla_{\BW}\mF(\mu_t), v_t\rangle_{\mu_t}$ along any curve of Gaussians $(\mu_t)_{t\in\RR}$ and any family of affine maps $(v_t)_{t\in\RR}$ which together satisfy the continuity equation.

In particular, if $\hat \pi$ is a minimizer of~\eqref{eq:original}, the first-order stationary condition $\nabla_{\BW} \mF(\hat{\pi}) = 0$ for Problem~\eqref{eq:original} reads as
\begin{align}\label{eq:firstorderstationarity}
\EE_{\hat \pi} \nabla V = 0~\qquad\text{and}\qquad~\EE_{\hat \pi} \nabla^2 V = \hat\Sigma^{-1}\,,
\end{align}
where $ \hat\Sigma$ is the covariance matrix corresponding to the distribution $\hat \pi$.

\section{Convexity and smoothness inequalities in the \ac{BW} space for the potential and the entropy}
\label{section:gg-ineqs-proof}

Having derived a formula for the BW gradient of a generic functional $\mF \colon \PG(\RR^d)\to \RR$
in \Cref{sec:bw-grad-comp},
we may now proceed to prove
\Cref{lem:smooth} (for the potential) and \Cref{lem:cvx} (for the entropy).

For both lemmas, the key idea is to differentiate a functional $\mF \colon \PG(\RR^d) \to \RR$ along a curve
$(\mu_t)_{t\in [0, 1]}$ with velocity vector fields $(v_t)_{t \in [0, 1]}$ satisfying the continuity equation \eqref{eqn:continuity},
utilizing our calculation rules laid out in \Cref{sec:bw-grad-comp}.
In particular,
we will use that
\begin{align}
    \mF(\mu_1)
    - \mF(\mu_0)
    &=
    \int_0^1 \partial_t \mF(\mu_t) \dd{t}
    \nonumber\\
    &= 
    \partial_t |_{t=0} \mF(\mu_t) + \int_0^1 \int_0^t \partial^2_s F(\mu_s) \dd{s} \dd{t}
    \nonumber\\
    &=
    \left\langle \nabla_{\mathsf{BW}} \mF(\mu_0), v_0 \right\rangle_{\mu_0} + \int_0^1 (1 - t)\, \partial^2_t \mF(\mu_t) \dd{t},
    \label{eqn:2nd-order-expansion-pre}
\end{align}
for both the entropy and the potential.

\subsection{Proof of \Cref{lem:smooth}}
\label{sec:proof-lem-pot}
We prove the following result for the potential. This result is stronger than \Cref{lem:smooth}, and will be useful in our subsequent analysis.
\begin{lemma}\label{lemma:gg-ineqs-V}
    Suppose that $\alpha I \preceq \nabla^2 V \preceq \beta I$. Let $\mu \in \PG(\RR^d)$ and let $h \colon \RR^d \to \RR^d$ be an affine map. Then the following inequalities hold:
    \begin{align*}
        \mV((\id + h)_\# \mu)
        - \mV(\mu)
        &\geq
        \left\langle \nabla_{\mathsf{BW}} \mV(\mu), h \right\rangle_{\mu}
        +
        \frac{\alpha}{2}\, \norm{h}_{\mu}^2\,, \\
        \mV((\id + h)_\# \mu)
        - \mV(\mu)
        &\leq
        \left\langle \nabla_{\mathsf{BW}} \mV(\mu), h \right\rangle_{\mu}
        +
        \frac{\beta}{2}\, \norm{h}_{\mu}^2\,.
    \end{align*}
\end{lemma}
\begin{proof}
    Let $X \sim \mu$. Note that regardless of $\mu$, we have that $\delta \mV(\mu) = V$.
    Hence, $\nabla_{W_2} \mV(\mu) = \nabla V$.
    We thus compute that
    \begin{align*}
        \mV((\id + h)_\# \mu)
        - \mV(\mu)
        &=
        \EE[V(X + h(X)) - V(X)]\\
        &\geq
        \EE\bigl[
        \left\langle \nabla V(X), h(X) \right\rangle
        + \frac{\alpha}{2}\,\norm{h(X)}^2
        \bigr]
        \ctag{since $\nabla^2 V \succeq \alpha I$}\\
        &=
        \left\langle \nabla_{\mathsf{W_2}} \mV(\mu), h \right\rangle_{\mu}
        +
        \frac{\alpha}{2}\, \norm{h}_{\mu}^2
        \\
        &=
        \left\langle \nabla_{\mathsf{BW}} \mV(\mu), h \right\rangle_{\mu}
        +
        \frac{\alpha}{2}\, \norm{h}_{\mu}^2\,,
    \end{align*}
    proving the first inequality. The second inequality follows similarly, using the fact that $\nabla^2 V \preceq \beta I$.
\end{proof}

\Cref{lem:smooth} then follows as a corollary of the above lemma.

\begin{proof}[Proof of \Cref{lem:smooth}]
    Note that if $V$ is $\beta$-smooth, then
    we have by definition that $-\beta I \preceq \nabla^2 V \preceq \beta I$.
    Hence,
    applying \Cref{lemma:gg-ineqs-V} with $\alpha = -\beta$,
    we obtain that
    \begin{align*}
        \bigl\lvert \mV((\id + h)_\# \mu)
        - \mV(\mu)
        -
        \left\langle \nabla_{\mathsf{BW}} \mV(\mu), h \right\rangle_{\mu}
        \bigr\rvert
        \leq
        \frac{\beta}{2}\, \norm{h}_{\mu}^2.
    \end{align*}
    Moreover, we have shown in \Cref{sec:BWgradpot}
    that $\nabla_{\BW} \mV(\mu)$ is given by
    \[
    \nabla_{\mathsf{BW}}\mV(\mu): x \mapsto \EE_\mu \nabla V + (\EE_\mu \nabla^2 V)(x - m_{\mu})\,,
    \]
    completing the proof of our desired result.
\end{proof}

\subsection{Proof of \Cref{lem:cvx}}
\label{sec:proof-lem-ent}

For the entropy, we follow the same strategy as in the previous proof, differentiating the
entropy $\mH$ along a particular curve.
This time, we will differentiate along the \textit{generalized geodesic}
$(\mu_t^\nu)_{t\in [0, 1]} \subseteq \mathsf{BW}(\RR^d)$, which we define as follows:

Let 
$T_0, T_1$
be the optimal transport maps
for which $T_0 - \id, T_1 - \id \in T_\nu \mathsf{BW}(\RR^d)$
and $(T_0)_\# \nu = \mu_0$ and $(T_1)_\# \nu = \mu_1$, respectively.
Defining $T_t \defeq (1 - t)\, T_0 + t\, T_1$, the generalized geodesic with basepoint $\nu$ and endpoints $\mu_0, \mu_1$ is then the curve of measures $(\mu_t^\nu)_{t\in [0, 1]} \subseteq \mathsf{BW}(\RR^d)$
with
$
\mu_t^\nu = (T_t)_\# \nu
$.
We note that $\mu_0^\nu = \mu_0$ and $\mu_1^\nu = \mu_1$, and that $(\mu_t^\nu)_{t\in [0, 1]}$ solves the continuity equation
\begin{align*}
    \partial_t \mu_t^\nu + \Div(\mu_t^\nu v_t) = 0, \qquad \text{where} \; v_t = (T_1 - T_0) \circ T_t^{-1}\,.
\end{align*}

Generalized geodesics were used in~\citet{ambrosio2008gradient} to study gradient flows in the Wasserstein space, and have since been useful for various applications of this theory, e.g.,~\citet{gd-bw-barycenters, ahnchewi2021mirrorlangevin, averaging-on-bw}.

\begin{proof}[Proof of \Cref{lem:cvx}]
    The JKO operator of $\mH$ over the Wasserstein space $\mP_2(\RR^d)$ is derived in \citet[Example~7]{sampling-as-opt} for a Gaussian measure $\mu = \mN(\mu, \Sigma)$,
    and takes the form
    $\mu' = \mN(\mu, \Sigma_1)$ where
    $\Sigma_1$ is defined in the same manner as \Cref{eq:BWjko}.
    Since $\mu'$ is also an element of $\PG(\RR^d)$, we conclude that $\mu'$ is also the result of applying the BW JKO operator to $\mu$, proving our desired closed form.

    Now we demonstrate the desired generalized geodesic convexity inequality for the entropy. In fact, this claim follows from general results on the Wasserstein space~\citep[see, e.g.,][\S 9.4]{ambrosio2008gradient}, but we give a proof here for completeness.
    As mentioned above, to do so we will differentiate $\mH$ along the generalized geodesic $(\mu_t^\nu)_{t\in [0, 1]}$ defined above.
    Abusing notation, we identify a distribution $\mu$ with its density with respect to Lebesgue measure.
    We then have that
    \begin{align*}
        \partial_t^2 \mH(\mu_t^\nu)
        &=
        \partial_t^2 \int \mu_t^\nu \ln \mu_t^\nu
        = 
        \partial_t^2 \int \nu \ln (\mu_t^\nu \circ T_t)
        = 
        \partial_t^2 \int \nu \ln \frac{\nu}{ \det \nabla T_t}
        \ctag{since $(T_t)_\# \nu = \mu_t^\nu$, change of variable}
        \\
        &= 
        -\int ( \partial_t^2 \ln \det \nabla T_t )\dd{\nu}
        = 
        -\int \partial_t \left\langle [\nabla T_t]^{-1}, \partial_t \nabla T_t \right\rangle \dd{\nu }
        \\
        &
        = 
        -\int \partial_t \left\langle [\nabla T_t]^{-1}, \nabla T_1 - \nabla T_0 \right\rangle\dd{\nu }
        = 
        \int \left\langle [\nabla T_t]^{-2}, (\nabla T_1 - \nabla T_0)^2 \right\rangle \dd{\nu}\\
        &= 
        \left\langle [\nabla T_t]^{-2}, (\nabla T_1 - \nabla T_0)^2 \right\rangle\,,
    \end{align*}
    where the last line follows since $T_t$ is an affine map, meaning that
    $\nabla T_t$ is constant on $\RR^d$.
    In addition, by Brenier's theorem~\citep[Theorem 2.12]{villani2003}, $T_t$ is the gradient of a convex function
    for all $t\in [0, 1]$.
    Hence, we know that $\nabla T_t \succeq 0$ for all $t\in [0, 1]$, meaning that
        $
        \left\langle [\nabla T_t]^{-2}, (\nabla T_1 - \nabla T_0)^2 \right\rangle
        \geq 
        0.
        $
    Hence, using \Cref{eqn:2nd-order-expansion-pre} applied to $\mH$, we obtain that
    \begin{align*}
        \mH(\mu_1) - \mH(\mu_0)
        &=
        \left\langle \nabla_{\mathsf{BW}} \mH(\mu_0) \circ T_0, T_1 - T_0 \right\rangle_\nu + \int_0^1 
        (1 - t)\,
        \left\langle [\nabla T_t]^{-2}, (\nabla T_1 - \nabla T_0)^2 \right\rangle
        \dd{t}\\
        &\geq
        \left\langle \nabla_{\mathsf{BW}} \mH(\mu_0) \circ T_0, T_1 - T_0 \right\rangle_\nu\,.
    \end{align*}
    This proves the desired inequality for the entropy, and we conclude our proof.
\end{proof}

\begin{remark}
    In fact, we can show a \textit{strong convexity} inequality for the entropy
    along generalized geodesics connecting distributions $\mu_0, \mu_1 \in \mathsf{BW}(\RR^d)$
    with the same mean.
    Let $m_0, m_1$ be the means of $\mu_0, \mu_1$ respectively,
    and suppose that $\Sigma_{\mu_0}, \Sigma_{\mu_1} \preceq \lambda I$.
    We compute that
    \begin{align*}
        \left\langle [\nabla T_t]^{-2}, (\nabla T_1 - \nabla T_0)^2 \right\rangle
        &=
        \left\langle I, [\nabla T_t]^{-1}\,(\nabla T_1 - \nabla T_0)^2\,[\nabla T_t]^{-1} \right\rangle
        \\
        &\geq
        \frac{1}{\norm{\Sigma_{\mu_t^\nu}}_{\mathrm{op}}}\,
        \left\langle \Sigma_{\mu_t}^\nu, [\nabla T_t]^{-1}\,(\nabla T_1 - \nabla T_0)^2\,[\nabla T_t]^{-1} \right\rangle
        \\
        &=
        \frac{1}{\norm{\Sigma_{\mu_t^\nu}}_{\mathrm{op}}}\,
        \left\langle [\nabla T_t]^{-1} \,\Sigma_{\mu_t^\nu}\,[\nabla T_t]^{-1}, (\nabla T_1 - \nabla T_0)^2 \right\rangle
        \\
        &=
        \frac{1}{\norm{\Sigma_{\mu_t^\nu}}_{\mathrm{op}}}\,
        \left\langle \Sigma_\nu, (\nabla T_1 - \nabla T_0)^2 \right\rangle.
    \end{align*}
        
    Since $T_0$ is an affine map, we know that $T_{0} (x) - (\nabla T_0)\, x$ is a constant for all $x\in \RR^d$, and similarly for $T_1$. Hence, we find that if $Y \sim \nu$, then
    \begin{align}
        \norm{T_1 - T_0}_\nu^2
        &=
        \Tr(\Cov_{\nu}[T_1 - T_0, T_1 - T_0])
        +
        \norm{\EE_{\nu} [T_1 - T_0]}^2
        \ctag{by bias-variance decomposition}
        \nonumber
        \\
        &=
        \Tr(\Cov[(\nabla T_1 - \nabla T_0)(Y), (\nabla T_1 - \nabla T_0)(Y)])
        +
        \norm{m_1 - m_0}^2
        \ctag{since $T_0, T_1$ are affine}
        \nonumber
        \\
        &=
        \left\langle \Sigma_\nu, (\nabla T_1 - \nabla T_0)^2 \right\rangle
        +
        \norm{m_1 - m_0}^2\,.
        \label{eqn:bias-variance}
    \end{align}
    In addition, from \citet[Lemma 10]{gd-bw-barycenters}, we know that the operator norm of
    the covariance matrix is convex along generalized geodesics in $\mathsf{BW}(\RR^d)$,
    implying that
    $\Sigma_{\mu_t^\nu} \preceq \lambda I$ for all $t\in [0, 1]$.
    Thus, we obtain
    \begin{align*}
        \left\langle [\nabla T_t]^{-2}, (\nabla T_1 - \nabla T_0)^2 \right\rangle
        &\leq
        \frac{1}{\norm{\Sigma_{\mu_t^\nu}}_{\mathrm{op}}}\,
        \left\langle \Sigma_\nu, (\nabla T_1 - \nabla T_0)^2 \right\rangle
        \\
        &= 
        \frac{1}{\norm{\Sigma_{\mu_t^\nu}}_{\mathrm{op}}}\,
        \bigl( 
        \norm{T_1 - T_0}_\nu^2
        -
        \norm{m_1 - m_0}^2
        \bigr)
        \ctag{by \Cref{eqn:bias-variance}}
        \\
        &\geq
        \frac{1}{\lambda}\,
        \bigl( 
        \norm{T_1 - T_0}_\nu^2
        -
        \norm{m_1 - m_0}^2
        \bigr)
        \ctag{by \citet[Lemma 10]{gd-bw-barycenters}}\,.
    \end{align*}
    Hence, using \Cref{eqn:2nd-order-expansion-pre} applied to $\mH$, we obtain that
    \begin{align*}
        \mH(\mu_1) - \mH(\mu_0)
        &=
        \left\langle \nabla_{\mathsf{BW}} \mH(\mu_0) \circ T_0, T_1 - T_0 \right\rangle_\nu + \int_0^1 
        (1 - t)\,
        \left\langle [\nabla T_t]^{-2}, (\nabla T_1 - \nabla T_0)^2 \right\rangle
        \dd{t}\\
        &\geq
        \left\langle \nabla_{\mathsf{BW}} \mH(\mu_0) \circ T_0, T_1 - T_0 \right\rangle_\nu 
        + \frac{1}{2\lambda}\, \bigl( \norm{T_1 - T_0}_\nu^2 - \norm{m_1 - m_0}^2 \bigr)\,.
    \end{align*}
    This implies that the entropy is strongly convex along generalized geodesics
    between two Gaussians $\mu_0, \mu_1 \in \mathsf{BW}(\RR^d)$ with the same mean.
    Similarly, the same computation
    can be used to show a \textit{smoothness} inequality for the
    entropy along \textit{geodesics}.
    As before, we compute that
    \begin{align*}
        \left\langle [\nabla T_t]^{-2}, (\nabla T_1 - \nabla T_0)^2 \right\rangle
        &=
        \left\langle I, [\nabla T_t]^{-1}\,(\nabla T_1 - \nabla T_0)^2\,[\nabla T_t]^{-1} \right\rangle
        \\
        &\leq
        \frac{1}{\lambda_{\min}(\Sigma_{\mu_t^\nu})}\,
        \left\langle \Sigma_{\mu_t}^\nu, [\nabla T_t]^{-1}\,(\nabla T_1 - \nabla T_0)^2\,[\nabla T_t]^{-1} \right\rangle
        \\
        &=
        \frac{1}{\lambda_{\min}(\Sigma_{\mu_t^\nu})}\,
        \left\langle [\nabla T_t]^{-1}\, \Sigma_{\mu_t^\nu}\,[\nabla T_t]^{-1}, (\nabla T_1 - \nabla T_0)^2 \right\rangle
        \\
        &=
        \frac{1}{\lambda_{\min}(\Sigma_{\mu_t^\nu})}\,
        \left\langle \Sigma_\nu, (\nabla T_1 - \nabla T_0)^2 \right\rangle
        \\
        &= 
        \frac{1}{\lambda_{\min}(\Sigma_{\mu_t^\nu})}\,
        \bigl( 
        \norm{T_1 - T_0}_\nu^2
        -
        \norm{m_1 - m_0}^2
        \bigr)
        \\
        &\leq
        \frac{1}{\lambda_{\min}(\Sigma_{\mu_t^\nu})}\,
        \norm{T_1 - T_0}_\nu^2\,.
    \end{align*}
    Once again using \Cref{eqn:2nd-order-expansion-pre} applied to $\mH$, we obtain that
    \begin{align}
        \mH(\mu_1) - \mH(\mu_0)
        &=
        \left\langle \nabla_{\mathsf{BW}} \mH(\mu_0) \circ T_0, T_1 - T_0 \right\rangle_\nu + \int_0^1 
        (1 - t)\,
        \left\langle [\nabla T_t]^{-2}, (\nabla T_1 - \nabla T_0)^2 \right\rangle
        \dd{t}
        \nonumber
        \\
        &\leq
        \left\langle \nabla_{\mathsf{BW}} \mH(\mu_0) \circ T_0, T_1 - T_0 \right\rangle_\nu + \int_0^1 
        \frac{1 - t}{\lambda_{\min}(\Sigma_{\mu_t^\nu})}\,
        \norm{T_1 - T_0}_\nu^2
        \dd{t}
        .
        \label[Inequality]{ineq:entropy-smoothness}
    \end{align}
\end{remark}

As a corollary of \Cref{ineq:entropy-smoothness}, we obtain a smoothness inequality for the entropy along geodesics, which will be useful for our subsequent analysis.

\begin{lemma}[Smoothness of entropy along geodesics]
    \label{lemma:entropy-smoothness}
    Suppose that $\mu_0, \mu_1 \in \PG(\RR^d)$
    satisfy $\Sigma_{\mu_0}^{-1}, \Sigma_{\mu_1}^{-1} \preceq \gamma I$.
    Then if $T$ is the optimal transport map from $\mu_0$ to $\mu_1$, we have that
    \begin{align*}
        \mH(\mu_1) - \mH(\mu_0)
        &\leq
        \left\langle \nabla_{\BW} \mH(\mu_0), T - \id \right\rangle_{\mu_0} + \frac{\gamma}{2}\, \norm{T - \id}_{\mu_0}^2\,.
    \end{align*}
\end{lemma}
\begin{proof}
    We apply
    \Cref{ineq:entropy-smoothness}
    with 
    $\nu = \mu_0$, noting in this case that $T_1 = T$ and $T_0 = \id$, and that $(\mu_t^\nu)_{t\in [0, 1]}$ is precisely the constant-speed geodesic $(\mu_t)_{t\in [0, 1]}$ connecting $\mu_0, \mu_1$.
    Furthermore, by \citet[Appendix B]{averaging-on-bw},
    we know that $\lambda_{\min}$ is concave along geodesics, so $\lambda_{\min}(\Sigma_{\mu_t}) \geq \gamma^{-1} I$ for all $t$.
    Hence,
    we obtain
    \begin{align*}
        \mH(\mu_1) - \mH(\mu_0)
        &\leq
        \left\langle \nabla_{\mathsf{BW}} \mH(\mu_0), T - \id \right\rangle_{\mu_0} + \int_0^1 
        \frac{1 - t}{\lambda_{\min}(\Sigma_{\mu_t})}\,
        \norm{T - \id}_{\mu_0}^2
        \dd{t}
        \\
        &\leq
        \left\langle \nabla_{\mathsf{BW}} \mH(\mu_0), T - \id \right\rangle_{\mu_0} + \int_0^1 
        \frac{1 - t}{\gamma^{-1}}\,
        \norm{T - \id}_{\mu_0}^2
        \dd{t}
        \\
        &= 
        \left\langle \nabla_{\mathsf{BW}} \mH(\mu_0), T - \id \right\rangle_{\mu_0} 
        +
        \frac{\gamma}{2}\, \norm{T - \id}_{\mu_0}^2\,,
    \end{align*}
    proving the desired result.
\end{proof}

\section{Proof of the one-step inequality (\Cref{lemma:discrete-evi})}\label{section:discrete-evi-proof}

The key idea of this proof is 
to decompose the difference 
$
\mF(\iterate_{k+1}) - \mF(\nu)
$
as the sum of three terms,
\begin{align*}
    \mF(\iterate_{k+1}) - \mF(\nu) 
    &= 
    [\mV(\iterate_{k+1}) - \mV(\iterate_{k+\half})]+
    [\mV(\iterate_{k+\half}) - \mV(\nu)]+
    [\mH(\iterate_{k+1}) - \mH(\nu)]\,,
\end{align*}
where each individual term may be controlled using the inequalities in \Cref{lemma:gg-ineqs-V,lem:cvx}.
Recalling that \Cref{lem:cvx} applies only to \textit{generalized geodesics},
we must take care in defining couplings between $\iterate_{k}, \iterate_{k+\half}, \iterate_{k+1}$ and $\nu$.
We detail the argument in the following proof.

\begin{proof}[Proof of \Cref{lemma:discrete-evi}]
    Recall from Section~\ref{sec:convergence} that we defined $\SF_k$ as the $\sigma$-algebra generated up to iteration $k$ (but not including the random sample $\hat{X}_k \sim p_k$ in Stochastic FB--GVI)). We also have 
    \begin{align*}
        e_k \colon x \mapsto (S_k - \EE_{\iterate_k} \nabla^2 V)(x - m_k) + (b_k - \EE_{\iterate_k} \nabla V)
    \end{align*}
    to be defined as the (random) error of the gradient estimate at iteration $k$ of (stochastic) FB--GVI, for which $\EE[e_k \mid \SF_k] =0$.
    Conditioned on the filtration $\SF_k$, we construct the following random variables $X_k, X_{k+\half}, X_{k+1}, Y_\mV$ and $Y_\mH$.

    Let $(X_k, Y_\mV)\sim (\iterate_k, \nu)$ be optimally coupled for the $W_2$ distance,
    and let $(X_k, Y_\mV) \indep e_k$.
    Since $\stepsize \leq \frac{1}{\beta}$ by assumption,
    we have that
    \[
    I - \stepsize S_k \succeq (1-\stepsize\beta)\, I\succeq 0\,.
    \]
    Recall that by Brenier's theorem~\citep[Theorem 2.12]{villani2003}, if $Y = \nabla \varphi(X)$ for a convex, proper, and lower-semicontinuous function $\varphi : \RR^d\to\RR \cup \{\infty\}$, then $(X,Y)$ is an optimal coupling for the $2$-Wasserstein distance.
    The condition $I- \stepsize S_k \succeq 0$ above therefore ensures that $(X_{k}, X_{k+\half}) \sim (\iterate_k, \iterate_{k+\half})$ is an optimal coupling for the $W_2$ distance, where we define
    \begin{align*}
        X_{k+\half}
        &\defeq 
        (I- \stepsize S_k) (X_k - m_k) + m_k - \stepsize b_k\,.
    \end{align*}
    On the other hand, defining $X_{k+1}$ such that
    \begin{align*}
        X_{k+1} 
        &\defeq X_{k+\half} - \eta\, \nabla_{\mathsf{BW}} \mH(\iterate_{k+1})[X_{k+1}]\\
        &= 
        (I+ \stepsize\, \Sigma_{k+1})^{-1} (X_{k+\half} - m_{k+1}) + m_{k+1}\,,
    \end{align*}
    we also get that $(X_{k+\half}, X_{k+1}) \sim (\iterate_{k+\half}, \iterate_{k+1})$ are optimally coupled.
    Finally, we construct the random variable $Y_\mH \sim \nu$
    for which $(X_{k+\half}, Y_\mH)$ are optimally coupled for the $W_2$ distance.

    First, we bound the difference in energy.
    From Brenier's theorem, we know that $Y_{\mH}$ and $X_{k+1}$ can both be expressed as
    an affine functions of $X_k$,
    thereby enabling the application of \Cref{lemma:gg-ineqs-V}.
    Doing so, we obtain that
    \begin{align*}
        \EE[
        \mV(\iterate_{k+1}) - \mV(\nu)
        ]
        &=
        \EE[\mV(\iterate_{k+1}) - \mV(\iterate_k)] + \EE[\mV(\iterate_k) - \mV(\nu)]\\
        &\leq
        \EE \left\langle \nabla_{\mathsf{BW}} \mV(\iterate_k)(X_k), X_k - Y_\mV \right\rangle - \frac{\alpha}{2}\, \EE \norm{X_k - Y_\mV}^2
        \\
        &\qquad
        +
        \EE \left\langle \nabla_{\mathsf{BW}} \mV(\iterate_k)(X_k), X_{k+1} - X_k \right\rangle + \frac{\beta}{2}\, \EE \norm{X_{k+1} - X_k}^2
        \ctag{by \Cref{lemma:gg-ineqs-V}}
        \\
        &=
        -\frac{\alpha}{2}\, \EE \norm{X_k - Y_\mV}^2
        + \EE \left\langle \nabla_{\mathsf{BW}} \mV(\iterate_k)(X_k),
        X_{k+1} - Y_\mV \right\rangle
        \\
        &\qquad
        + \frac{1}{2\stepsize}\, \EE \norm{X_{k+1} - X_k}^2
        - 
        \Bigl( \frac{1}{2\stepsize} - \frac{\beta}{2} \Bigr)\,
        \EE \norm{X_{k+1} - X_k}^2
        \\
        &=
        -\frac{\alpha}{2}\, \EE \norm{X_k - Y_\mV}^2
        - \EE \left\langle e_k(X_k), X_{k+1} - Y_\mV \right\rangle
        -\frac{1}{\stepsize}\, \EE\langle 
        X_{k+\half} - X_{k},
        X_{k+1} - Y_\mV
        \rangle
        \\
        &\qquad
        + \frac{1}{2\stepsize}\, \EE \norm{X_{k+1} - X_k}^2
        -
        \Bigl( \frac{1}{2\stepsize} - \frac{\beta}{2} \Bigr)\,
        \EE \norm{X_{k+1} - X_k}^2
        \\
        &= 
        \frac{1}{2\stepsize}\,
        \left( 
        1 - \alpha\stepsize
        \right)\,
        \EE \norm{X_k - Y_\mV}^2
        - \EE \left\langle e_k(X_k), X_{k+1} - Y_\mV \right\rangle
        -
        \Bigl( \frac{1}{2\stepsize} - \frac{\beta}{2} \Bigr)\,
        \EE \norm{X_{k+1} - X_k}^2
        \\
        &\qquad
        +\frac{1}{2\stepsize}\, \EE 
        \bigl[ 
        \norm{X_{k+1} - X_k}^2
        -
        \norm{X_k - Y_\mV}^2
        -
        2
        \,\langle 
        X_{k+\half} - X_{k},
        X_{k+1} - Y_\mV
        \rangle
        \bigr]\,.
    \end{align*}
    Now we bound the difference in entropy.
    Since $Y_\mH$ and $X_{k+1}$ are both optimally coupled with $X_{k+\half}$,
    we know that $(Y_\mH, X_{k+1})$ are coupled along a generalized geodesic.
    Hence, we can apply \Cref{lem:cvx}
    to obtain that 
    \begin{align*}
        \EE[\mH(\iterate_{k+1}) - \mH(\nu)]
        &\leq
        \EE \left\langle \nabla_{\mathsf{BW}} \mH(\iterate_{k+1})[X_{k+1}], X_{k+1} - Y_\mH \right\rangle
        \\
        &=
        - \frac{1}{\stepsize}\,
        \EE \langle X_{k+1} - X_{k+\half}, X_{k+1} - Y_\mH \rangle
        \\
        &= 
        \frac{1}{2\stepsize}\,
        \EE
        \bigl[ 
        \|X_{k+\half} - Y_\mH\|^2 - \|X_{k+1} - X_{k+\half}\|^2 
        -
        \|X_{k+1} - Y_\mH\|^2
        \bigr]
        \\
        &\leq 
        \frac{1}{2\stepsize}\,
        \EE
        \bigl[ 
        \|X_{k+\half} - Y_\mV\|^2 - \|X_{k+1} - X_{k+\half}\|^2 
        - 
        \|X_{k+1} - Y_\mH\|^2
        \bigr]
        \ctag{since $(X_{k+\half}, Y_\mH)$ are optimally coupled}\,.
    \end{align*}
    Now, we sum the above inequalities
    to obtain our desired bound on $\EE[\mF(\iterate_{k+1}) - \mF(\nu)]$.
    We obtain that
    \begin{align}
        \EE[\mF(\iterate_{k+1}) - \mF(\nu)]
        &= 
        \EE[
        \mV(\iterate_{k+1}) - \mV(\nu)
        ]
        +
        \EE[
        \mH(\iterate_{k+1}) - \mH(\nu)
        ]
        \nonumber
        \\
        &\leq
        \frac{1}{2\stepsize}\,
        \EE 
        \bigl[
        \left( 
        1 - \alpha\stepsize
        \right)\,
        \norm{X_k - Y_\mV}^2
        - 
        \norm{X_{k+1} - Y_\mH}^2
        \bigr]
        \nonumber
        \\
        &\qquad
        + \frac{1}{2\stepsize}\,
        \EE \bigl[ 
        \norm{X_{k+1} - X_k}^2
        -
        \norm{X_k - Y_\mV}^2
        +
        \|X_{k+\half} - Y_\mV\|^2 
        - \|X_{k+1} - X_{k+\half}\|^2 
        \bigr]
        \nonumber
        \\
        &\qquad
        - \frac{1}{2\stepsize}\,
        \EE \bigl[
        2
        \,\langle 
        X_{k+\half} - X_{k},
        X_{k+1} - Y_\mV
        \rangle
        \bigr]
        \nonumber
        \\
        &\qquad
        - \EE \left\langle e_k(X_k), X_{k+1} - Y_\mV \right\rangle
        -
        \Bigl( \frac{1}{2\stepsize} - \frac{\beta}{2} \Bigr)\,
        \EE \norm{X_{k+1} - X_k}^2
        \nonumber
        \\
        &=
        \frac{1}{2\stepsize}\,\EE 
        \bigl[
        \left( 
        1 - \alpha\stepsize
        \right)\,
        \norm{X_k - Y_\mV}^2
        - 
        \norm{X_{k+1} - Y_\mH}^2
        \bigr]
        \nonumber
        \\
        &\qquad
        - \EE \left\langle e_k(X_k), X_{k+1} - Y_\mV \right\rangle
        -
        \Bigl( \frac{1}{2\stepsize} - \frac{\beta}{2} \Bigr)\,
        \EE \norm{X_{k+1} - X_k}^2\,.
        \label[Inequality]{ineq:pre-evi}
    \end{align}
    Finally, it remains to bound the error term on the last line.
    For this, we consider two cases based on whether or not the
    error term $e_k$ is identically zero:
    \begin{itemize}
        \ii 
        In the case of FB--GVI where we have access to the exact gradient $\nabla_{\BW} \mV(p_k)$, we have that $e_k \equiv 0$, so
    \begin{align*}
        - \EE \left\langle e_k(X_k), X_{k+1} - Y_\mV \right\rangle
        &= 0\,.
    \end{align*}
    Combining this with \Cref{ineq:pre-evi}, we obtain that
    with $\stepsize \leq \frac{1}{\beta}$,
    \begin{align*}
        \EE[\mF(\iterate_{k+1}) - \mF(\nu)]
        &\leq
        \frac{1}{2\stepsize}\,
        \EE \bigl[
        \left( 
        1 - \alpha\stepsize
        \right)\,
        \norm{X_k - Y_\mV}^2
        -
        \norm{X_{k+1} - Y_\mH}^2
        \bigr]
        -
        \Bigl( \frac{1}{2\stepsize} - \frac{\beta}{2} \Bigr)\,
        \EE \norm{X_{k+1} - X_k}^2\\
        &\leq
        \frac{1}{2\stepsize}\,
        \EE \bigl[
        \left( 
        1 - \alpha\stepsize
        \right)\,
        \norm{X_k - Y_\mV}^2
        -
        \norm{X_{k+1} - Y_\mH}^2
        \bigr]\,.
    \end{align*}
    Rearranging, we conclude that if $e_k \equiv 0$ and $\stepsize \leq \frac{1}{\beta}$,
    \begin{align}
    \label{eq:proof-no-rate}
        \EE W_2^2(\iterate_{k+1}, \nu)
        &\leq 
        \EE\norm{X_{k+1} - Y_\mH}^2
        \\
        &\leq
        (1 - \alpha\stepsize)\,\EE \norm{X_k - Y_\mV}^2 - 2\stepsize \,\EE[\mF(\iterate_{k+1}) - \mF(\nu)]
        \label[Inequality]{ineq:osi-precursor}
        \\
        &= 
        (1 - \alpha\stepsize)\,\EE W_2^2(\iterate_k, \nu) - 2\stepsize\, \EE[\mF(\iterate_{k+1}) - \mF(\nu)]\,.
        \ctag{since conditioned on $\SF_k$, $(X_k, Y_\mV)$ are optimally coupled}
        \nonumber
    \end{align}

    \ii
    Otherwise, if $e_k$ is not necessarily identically 0,
    we can still compute
    \begin{align*}
        - \EE \left\langle e_k(X_k), X_{k+1} - Y_\mV \right\rangle
        &=
        - \EE \left\langle e_k(X_k), X_{k+1} - X_k \right\rangle
        \ctag{since $e_k \indep (X_k, Y_\mV)$ by construction}
        \\
        &\leq
        \stepsize\, \EE \norm{e_k(X_k)}^2 + \frac{1}{4\stepsize} \,\EE \norm{X_{k+1} - X_k}^2\,.
        \ctag{Cauchy--Schwarz and Young's inequality}
    \end{align*}
    Hence, combining this with \Cref{ineq:pre-evi}, we obtain that
    for $\stepsize \leq \frac{1}{2\beta}$,
    \begin{align*}
        \EE[\mF(\iterate_{k+1}) - \mF(\nu)]
        &\leq
        \frac{1}{2\stepsize}\,
        \EE \bigl[
        \left( 
        1 - \alpha\stepsize
        \right)\,
        \norm{X_k - Y_\mV}^2
        -
        \norm{X_{k+1} - Y_\mH}^2
        \bigr]
        +
        \stepsize\, \EE \norm{e_k(X_k)}^2 \\
        &\qquad{} -
        \Bigl( \frac{1}{4\stepsize} - \frac{\beta}{2} \Bigr)\,
        \EE \norm{X_{k+1} - X_k}^2\\
        &\leq
        \frac{1}{2\stepsize}\,
        \EE \bigl[
        \left( 
        1 - \alpha\stepsize
        \right)\,
        \norm{X_k - Y_\mV}^2
        -
        \norm{X_{k+1} - Y_\mH}^2
        \bigr]
        +
        \stepsize\,\EE \sigma_k^2\,.
    \end{align*}
    Rearranging, we conclude that as long as $\stepsize \leq \frac{1}{2\beta}$,
    \begin{align*}
        \EE W_2^2(\iterate_{k+1}, \nu)
        &\leq 
        \EE\norm{X_{k+1} - Y_\mH}^2
        \\
        &\leq
        (1 - \alpha\stepsize)\,\EE \norm{X_k - Y_\mV}^2 - 2\stepsize\, \EE[\mF(\iterate_{k+1}) - \mF(\nu)] + 2\stepsize^2\, \EE \sigma_k^2
        \\
        &= 
        (1 - \alpha\stepsize)\,\EE W_2^2(\iterate_k, \nu) - 2\stepsize\, \EE[\mF(\iterate_{k+1}) - \mF(\nu)]+ 2\stepsize^2 \,\EE \sigma_k^2\,.
        \ctag{since $(X_k, Y_\mV)$ are optimally coupled}
    \end{align*}
    \end{itemize}
    Combining these two cases,
    we have demonstrated our desired inequality.
\end{proof}

\begin{remark}
    Consider specializing the above proof to the case where $\nu = \iterate_k$, for which $Y_\mV= Y_\mH= X_{k}$, so that $(X_k, X_{k+\half}) \sim (\iterate_k, \iterate_{k+\half})$ and $(X_{k+\half}, X_{k+1}) \sim (\iterate_{k+\half}, \iterate_{k+1})$ are optimally coupled for the $W_2$ distance.
    Then from \Cref{ineq:osi-precursor}, we obtain that
    \begin{align*}
        \EE \norm{X_{k+1} - Y_\mH}^2
        &\leq
        (1 - \alpha\stepsize)\,\EE \norm{X_k - Y_\mV}^2 - 2\stepsize \,\EE[\mF(\iterate_{k+1}) - \mF(\nu)]
        \ctag{\Cref{ineq:osi-precursor}}
        \\
        \implies
        \EE \norm{X_{k+1} - X_k}^2
        &\leq
        - 2\stepsize \,\EE[\mF(\iterate_{k+1}) - \mF(\iterate_k)]\,.
        \ctag{since $\nu = \iterate_k$ and $Y_{\mV} = Y_\mH= X_k$}
    \end{align*}
    As a corollary, we obtain the following lemma, which will be useful in subsequent analysis.
\end{remark}

\begin{lemma}
    \label{lemma:iterate-osi}
    Suppose that $V$ is $\beta$-smooth.
    Let $(\iterate_k)_{k \in \NN}$ be the iterates of FB--GVI~\eqref{eq:dfb1}--\eqref{eq:dfb2}.
    Let $\stepsize > 0$ be such that
    $\stepsize \leq \frac{1}{\beta}$.
    Let $(X_k, X_{k+\half}) \sim (\iterate_k, \iterate_{k+\half})$ and $(X_{k+\half}, X_{k+1}) \sim (\iterate_{k+\half}, \iterate_k)$ be optimally coupled for the $W_2$ distance.
    Then,
    \begin{align*}
        \EE
        \norm{X_{k+1} - X_k}^2
        &\leq 
        - 2\stepsize\,\EE [\mF(\iterate_{k+1}) - \mF(\iterate_k)]\,.
    \end{align*}
\end{lemma}

\section{Eigenvalue control of the iterates}

We will show the following eigenvalue bound result:

\begin{lemma}
    \label{lemma:iterate-eigbounds}
    At the $k$-th iteration of \Cref{alg:FBGVI},
    suppose that we have
    $\gamma_0 I\preceq \Sigma_k^{-1} \preceq \gamma_1 I$.
    As long as
    $0\leq \stepsize \leq \frac{1}{\gamma_1}$
    and $\gamma_0 I\preceq S_k \preceq \gamma_1 I$,
    we then have that
    \[
    \gamma_1^{-1} I\preceq \Sigma_{k+1} \preceq \gamma_0^{-1} I\,.
    \]
\end{lemma}

\begin{proof}
    Define the monotonically increasing function $f_\eta \colon \RR_{\geq0} \to \RR_{\geq0}$ such that
    \[
    f_\eta(x) = \half\,
    \bigl( 
    x + 2\eta + \sqrt{x\,(x+4\eta)}
    \bigr)\,.
    \]
    First, we make note of the following algebraic identity.
    Define $x_\gamma \defeq (1 - \eta\gamma)^2 / \gamma$. Then we have that
    \begin{align}
        f_\eta(x_\gamma) 
        &=
        \half\, \biggl( 
        \frac{(1 - \eta\gamma)^2}{\gamma} + 2\eta + \sqrt{\Bigl( \frac{(1 - \eta\gamma)^2}{\gamma} \Bigr)\, \Bigl( \frac{(1 - \eta\gamma)^2}{\gamma} + 4\eta \Bigr)}
        \biggr)
        \nonumber
        \\
        &= 
        \frac{1}{2\gamma}\, \bigl( 1 + \eta^2 \gamma^2 + \sqrt{(1 - \eta\gamma)^2\,(1 + \eta\gamma)^2} \bigr)
        \nonumber
        \\
        &= 
        \frac{1}{2\gamma}\, \bigl( 1 + \eta^2 \gamma^2 + (1-\eta\gamma)\,(1+\eta\gamma) \bigr)
        \nonumber
        \\
        &= \frac{1}{\gamma}\,.
        \label{eqn:algebraic-identity}
    \end{align}
    Now,
    let $\lambda_{\min}(M), \lambda_{\max}(M)$ denote the 
    minimum and maximum eigenvalues of a matrix $M \in \mathbf{S}^d$.
    The conditions $\eta \leq \gamma_1^{-1}$ and $S_k \preceq \gamma_1 I$ then
    imply that $I- \eta S_k \succeq 0$.
    Hence, we then have that
    \begin{align*}
        \lambda_{\min}(\Sigma_{k+\half})
        &= 
        \lambda_{\min}\bigl( (I- \eta S_k)\, \Sigma_k\, (I- \eta S_k) \bigr)\\
        &\geq
        \lambda_{\min}^2(I- \eta S_k)\, \lambda_{\min}(\Sigma_k)\\
        &\geq
        (1 - \eta \gamma_1)^2\, \lambda_{\min}(\Sigma_k)\\
        &\geq
        \frac{(1 - \eta \gamma_1)^2}{\gamma_1}\\
        &=
        x_{\gamma_1}\,.
    \end{align*}
    Now, we also note that $\Sigma_{k+\half}$ and $\Sigma_{k+1}$ commute
    by construction, so since $f_\eta$ is a monotonically increasing function,
    \begin{align*}
        \lambda_{\min}(\Sigma_{k+1})
        &=
        f_\eta\bigl(\lambda_{\min}(\Sigma_{k+\half})\bigr)
        \geq
        f_\eta(x_{\gamma_1})
        =
        \frac{1}{\gamma_1}\,,
    \end{align*}
    where the last equality follows from \Cref{eqn:algebraic-identity}.

    Similarly, for the upper bound, we have that
    \begin{align*}
        \lambda_{\max}(\Sigma_{k+\half})
        &= 
        \lambda_{\max}\bigl( (I- \eta S_k)\, \Sigma_k\, (I- \eta S_k) \bigr)\\
        &\leq
        \lambda_{\max}^2(I- \eta S_k)\, \lambda_{\max}(\Sigma_k)
        \ctag{since $I- \eta S_k \succeq 0$}
        \\
        &\leq
        (1 - \eta \gamma_0)^2\, \lambda_{\max}(\Sigma_k)\\
        &\leq
        \frac{(1 - \eta \gamma_0)^2 }{\gamma_0}\,.
    \end{align*}
    Thus, we similarly obtain
    \begin{align*}
        \lambda_{\max}(\Sigma_{k+1})
        &=
        f_\eta\bigl(\lambda_{\max}(\Sigma_{k+\half})\bigr)
        \leq
        f_\eta(x_{\gamma_0})
        =
        \frac{1}{\gamma_0}\,.
    \end{align*}
    Combining the above results, this proves that $\gamma_1^{-1} I\preceq \Sigma_{k+1} \preceq \gamma_0^{-1} I$ which is what we set out to show.
\end{proof}

Note that for (stochastic) FB--GVI, we have $\alpha I \preceq S_k \preceq \beta I$, so \Cref{lemma:iterate-eigbounds} holds with $\gamma_0 = \alpha$ and $\gamma_1 = \beta$. Hence, we obtain the following corollary:

\begin{corollary}
    \label{lemma:iterate-smoothness}
    Suppose that \Cref{alg:FBGVI} is initialized with a matrix $\Sigma_0$ such that
    $\beta^{-1} I \preceq \Sigma_0$, that $V$ is $\beta$-smooth, and that the step size satisfies $\eta \le \frac{1}{\beta}$.
    Then $\beta^{-1} I \preceq \Sigma_k$ for all $k$.
\end{corollary}

\section{Proofs of the noiseless algorithm convergence rates}

We obtain the desired convergence rates for FB--GVI
by rearranging and iterating the one-step inequality of \Cref{lemma:discrete-evi}.
First, we derive inequalities that hold for both the convex and strongly convex cases.

For FB--GVI,
we can apply \Cref{lemma:discrete-evi} with $\nu = \hat{\pi}$, $\eta \leq \frac{1}{\beta}$
and $\sigma_k = 0$.
Furthermore, FB--GVI is deterministic, so we may remove the expectations in \Cref{lemma:discrete-evi}.
In this case, the inequality in \Cref{lemma:discrete-evi} implies that for all $k$,
\begin{align}
    W_2^2(\iterate_{k+1}, \hat{\pi})
    &\leq 
    (1 - \alpha \eta)\, W_2^2(\iterate_k, \hat{\pi})
    - 2\eta\, (\mF(\iterate_{k+1}) - \mF(\hat{\pi}))
    \ctag{by \Cref{lemma:discrete-evi}}
    \nonumber
    \\
    &\leq
    \exp\left( -\alpha\eta \right)\, W_2^2(\iterate_k, \hat{\pi})
    - 2\eta\, (\mF(\iterate_{k+1}) - \mF(\hat{\pi}))
    \,.
    \label{ineq:noiseless-evi-sc}
\end{align}
Rearranging \cref{ineq:noiseless-evi-sc}, we obtain
\begin{align}
    \mF(\iterate_{k+1}) - \mF(\hat{\pi}) \leq
    \frac{\exp\left( -\alpha \eta \right)\,W_2^2(\iterate_{k}, \hat{\pi}) - W_2^2(\iterate_{k+1}, \hat{\pi})}{2\eta}\,.
    \label[Inequality]{ineq:KL-bound}
\end{align}
On the other hand, we can also apply \Cref{lemma:discrete-evi} with $\nu = \iterate_k$, $\eta \leq \frac{1}{\beta}$
and $\sigma_k^2 = 0$ to obtain that
\begin{align*}
    W_2^2(\iterate_{k+1}, \iterate_k) &\leq (1 - \alpha \eta)\, W_2^2(\iterate_k, \iterate_k) - 2\eta\, (\mF(\iterate_{k+1}) - \mF(\iterate_k)) = - 2\eta\, (\mF(\iterate_{k+1}) - \mF(\iterate_k))\,.
\end{align*}
Hence, rearranging this inequality, we obtain that
\begin{align}
    \mF(\iterate_{k+1}) - \mF(\iterate_k) \leq -\frac{W_2^2(\iterate_{k+1}, \iterate_k)}{2\eta} \leq 0\,,
    \label[Inequality]{ineq:obj-decrease}
\end{align}
meaning that the objective value decreases with each iteration of the algorithm.

\subsection{Proof of \Cref{thm:noiseless-wc-rate}}
\label{sec:noiseless-wc-rate-proof}

\begin{proof}
    Since $V$ is convex, \Cref{ineq:KL-bound} holds with the choice $\alpha = 0$,
    from which we obtain that
    \begin{align*}
        \mF(\iterate_{k+1}) - \mF(\hat{\pi}) \leq
        \frac{W_2^2(\iterate_k, \hat{\pi}) - W_2^2(\iterate_{k+1}, \hat{\pi})}{2\eta}\,.
    \end{align*}
    Telescoping this inequality, we obtain that
    \begin{align*}
        \mF(\iterate_{N}) - \mF(\hat{\pi})
        \leq
        \frac{1}{N} \sum_{k=1}^{N}
        [\mF(\iterate_{k}) - \mF(\hat{\pi})]
        \leq
        \frac{1}{2\eta N} \sum_{k=0}^{N-1}
        [W_2^2(\iterate_{k}, \hat{\pi}) - W_2^2(\iterate_{k+1}, \hat{\pi})]
        \le 
        \frac{W_2^2(\iterate_0, \hat{\pi})}{2\eta N}\,,
    \end{align*}
    where the first inequality holds by \Cref{ineq:obj-decrease}.
    Hence, with the choice
    \begin{align*}
        \eta = \frac{1}{\beta}\,,\quad\text{and}\quad
        N \gtrsim 
        \frac{\beta W_2^2(\iterate_0, \hat{\pi})}{\varepsilon^2}\,,
    \end{align*}
    we obtain the guarantee
        $
        \mF(\iterate_N) - \mF(\hat{\pi})
        \leq 
        \varepsilon^2,
        $
    proving our desired result.
\end{proof}

\subsection{Proof of \Cref{thm:noiseless-wc-no-rate}}
\label{sec:noiseless-wc-no-rate-proof}
    \begin{proof}
    Our proof makes use of the recent results by~\cite{naldi2021weak}. For every continuous function $\zeta : \RR^d \to \RR$, denote $F_\zeta$ the map defined over $\mP_2(\RR^d)$ by $F_\zeta : \mu \mapsto \int \zeta\, \dd \mu$. We define the NS-topology over $\mP_2(\RR^d)$ to be the initial topology induced by the family $F_{\zeta}$ where $\zeta$ is a continuous function such that $\frac{\zeta(x)}{1+\|x\|^2} \longrightarrow_{\|x\|\to \infty} 0.$ Denoting $\mC$ the set of such $\zeta$ functions, a sequence $(\mu_k)_{k\ge 0}$ converges to $\hat{\mu}$ for the NS-topology iff for every $\zeta \in \mC$, $F_\zeta(\mu_k) \to F_\zeta(\hat{\mu})$. 
    
    The set of bounded continuous functions over $\RR^d$ is included in $\mC$, therefore NS-convergence implies weak-convergence. Moreover, the set of continuous functions $\zeta$ such that $\sup_{x \in \RR^d} \frac{\zeta(x)}{1+\|x\|^2} < \infty$ contains $\mC$, therefore the convergence in the Wasserstein distance implies NS-convergence, see \citet[Equations 1.9 and 1.10]{naldi2021weak}.

    To prove \Cref{thm:noiseless-wc-no-rate}, we apply \citet[Theorem 6.9]{naldi2021weak}. This theorem implies the NS-convergence of FB--GVI which in turn implies weak convergence. We now check the assumptions of \citet[Theorem 6.9]{naldi2021weak}:
    \begin{itemize}[noitemsep,leftmargin=0.2in]
        \item First, the set $A = \BW(\RR^d) \cup \mD$, where $\mD$ is the set of Dirac measures over $\RR^d$ is NS-closed because it is weakly closed. 
        \item Then, denote $T: A \to A$ the map defining FB--GVI (see Equation~\eqref{eq:dfb2}), i.e., $T(\mu) = \jko_{\eta \mH}( (\id - \eta\, \nabla_{\mathsf{BW}}\mV(\mu))_{\#} \mu)$ and $\iterate_{k+1} = T(\iterate_k)$ where $\iterate_0 \in A$. We show that $T$ is asymptotically regular. Using Lemma~\ref{lemma:discrete-evi} with  $\alpha = 0$, if $\eta \leq 1 / \beta$, then
    \begin{align}
        W_2^2(T(\iterate_{k}), \nu) 
        &\leq W_2^2(\iterate_k, \nu) - 2\stepsize\,[\mF(T(\iterate_{k})) - \mF(\nu)]\,. 
    \end{align}
    Inspecting the proof of Lemma~\ref{lemma:discrete-evi}, especially Equation~\eqref{eq:proof-no-rate}, one can see that we actually have the following stronger inequality if $\eta < 1/\beta$: 
    \begin{align}
    \label{eq:bounded1}
        W_2^2(T(\iterate_{k}), \nu) 
        &\leq 
         W_2^2(\iterate_k, \nu) - 2\stepsize\,[\mF(T(\iterate_{k})) - \mF(\nu)] - (1 - \eta \beta)\,W_2^2(T(\iterate_{k}), \iterate_{k})\,. 
    \end{align}
    Besides, if $\nu$ is a minimizer of Problem~\eqref{eq:original}, then $\mF(T(\iterate_{k})) - \mF(\nu) \geq 0$, therefore
    \begin{align}
    \label{eq:bounded}
        W_2^2(T(\iterate_{k}), \nu) 
        &\leq 
         W_2^2(\iterate_k, \nu) - (1 - \eta \beta)\,W_2^2(T(\iterate_{k}), \iterate_{k})\,. 
    \end{align}
    Iterating the last inequality, we obtain $\sum_{k\ge 0} W_2^2(T(\iterate_{k}), \iterate_{k}) < \infty$. In particular, $W_2^2(T(\iterate_{k}), \iterate_{k}) \to 0$, for any initial measure $\iterate_0 \in A$. Therefore, $T$ is asymptotically regular.

    \item The proof of the non-expansiveness of $T$ is similar to the proof of the asymptotic regularity: one can adapt Lemma~\ref{lemma:discrete-evi} to show a one-step inequality where $\nu$ is varying and follows the FB--GVI algorithm. We skip this part of the proof as it is similar to the proof of Lemma~\ref{lemma:discrete-evi}. 

    \item Finally, $(\iterate_k)_{k\ge 0}$ is bounded using Equation~\eqref{eq:bounded}.

    \end{itemize}

    Therefore, using~\citet[Theorem 6.9]{naldi2021weak}, $(\iterate_k)_{k\ge 0}$ converges weakly to some $\pi_\star \in A$.  Besides, using Theorem~\ref{thm:noiseless-wc-rate}, $\mF(\iterate_k) - \min_{\BW(\RR^d)} \mF \to 0$. Therefore, using the weak lower semicontinuity of the KL divergence we have $\mF(\pi_\star) - \min_{\BW(\RR^d)} \mF \leq 0$. In particular, $\pi_\star \in \BW(\RR^d)$ ($\pi_\star$ cannot be in $\mD$ otherwise $\mF(\pi_\star) = +\infty$). Therefore, we also have $\mF(\pi_\star) - \min_{\BW(\RR^d)} \mF \geq 0$ which implies $\mF(\pi_\star) - \min_{\BW(\RR^d)} \mF = 0$. Finally, $\pi_\star$ is a minimizer of Problem~\eqref{eq:original}.

    To conclude the proof, observe that $(p_k)_{k\ge 0}$ converges weakly to $\pi_\star$ and that all these distributions are Gaussian. Therefore the convergence actually happens in Wasserstein distance.
    \end{proof}

\subsection{Proof of \Cref{thm:noiseless-sc-rate}}
\label{sec:noiseless-sc-rate-proof}

\begin{proof}
    Since $\mF(\hat{\pi}) \leq \mF(\iterate_{k+1})$ as $\hat{\pi}$ achieves the minimum
    of $\mF$ among Gaussians, we may
    iterate \Cref{ineq:KL-bound} to obtain 
    \begin{align*}
        W_2^2(\iterate_{N}, \hat{\pi})
        &\leq
        \exp\left( -N \alpha \eta \right)\,
        W_2^2(\iterate_0, \hat{\pi})\,.
    \end{align*}
    Hence, with the choice
    \begin{align*}
        \eta = \frac{1}{\beta} \,, \quad\text{and}\quad
        N \gtrsim 
        \frac{1}{\alpha\eta} \log \frac{\alpha W_2^2(\iterate_0, \hat{\pi})}{\varepsilon^2}
        \asymp
        \frac{\beta}{\alpha} \log \frac{\alpha W_2^2(\iterate_0, \hat{\pi})}{\varepsilon^2}\,,
    \end{align*}
    we obtain the guarantee
    $
        \alpha W_2^2(\iterate_N, \hat{\pi}) 
        \leq
        \varepsilon^2.
        $

    Now, for the guarantee in KL divergence,
    we ``reinitialize''
    the algorithm with distribution $\iterate_N$
    and apply the convex result of
    \Cref{thm:noiseless-wc-rate}.
    With the same choice of $N$ and $\eta$ and assuming $\varepsilon$ is sufficiently small,
    we can apply \Cref{thm:noiseless-wc-rate} to obtain the guarantee
    \begin{align*}
        \mF(\iterate_{2N}) - \mF(\hat{\pi})
        &\leq
        \frac{W_2^2(\iterate_N, \hat{\pi})}{2\eta N}
        \leq
        \frac{\varepsilon^2}{2 \alpha \eta N}
        \lesssim
        \frac{\varepsilon^2}{\log \frac{\alpha W_2^2(\iterate_0, \hat{\pi})}{\varepsilon^2}}
        \lesssim
        \varepsilon^2\,,
    \end{align*}
    proving our desired result.
\end{proof}

\subsection{Proof of \Cref{thm:noiseless-smooth-rate}}
\label{sec:noiseless-smooth-rate-proof}

First, we need a lemma.
\begin{lemma}
    \label{lemma:fake-coercivity}
    Let $\mu_0, \mu_1 \in \PG(\RR^d)$
    be such that $\Sigma_{\mu_0}, \Sigma_{\mu_1} \succeq \beta^{-1} I$. Then if $(X_0, X_1) \sim (\mu_0, \mu_1)$ are optimally coupled for the $W_2$ distance, we have that
    \begin{align*}
        \EE \norm{\nabla_{\BW} \mH(\mu_1)[X_1] - \nabla_{\BW} \mH(\mu_0)[X_0]}^2
        \leq
        20\beta^2\, W_2^2(\mu_0, \mu_1)\,.
    \end{align*}
\end{lemma}

The proof proceeds as follows. First, we apply
the triangle inequality and the Cauchy--Schwarz inequality to decompose the LHS into two terms which we will control separately.
For the first term, we appeal to the Lipschitzness of $\nabla_{\BW} \mH(\mu_1)$, which is possible since $\Sigma_{\mu_1}^{-1} \preceq \beta I$.
Then for the second term, we will utilize \Cref{lemma:gg-ineqs-V} and \Cref{lemma:entropy-smoothness} to derive a bound in terms of $\KL{\mu_0}{\mu_1}$, which we can then further bound in terms of $W_2^2(\mu_0, \mu_1)$.
Combining these bounds, we obtain our desired result.

\begin{proof}
    Applying the triangle inequality and Cauchy--Schwarz, we obtain that
    \begin{align*}
        \half\,
        \EE \norm{\nabla_{\BW} \mH(\mu_1)[X_1] - \nabla_{\BW} \mH(\mu_0) [X_0]}^2
        &\leq
        \EE \norm{\nabla_{\BW} \mH(\mu_1)[X_1] - \nabla_{\BW} \mH(\mu_1) [X_0]}^2
        \\
        &\qquad
        +
        \EE \norm{\nabla_{\BW} \mH(\mu_1)[X_0] - \nabla_{\BW} \mH(\mu_0) [X_0]}^2\,.
    \end{align*}
    For the first term, we note that
    since $\Sigma_{\mu_1}^{-1} \preceq \beta I$ by assumption,
    we have that
    \begin{align}
        \EE \norm{\nabla_{\BW} \mH(\mu_1)[X_1] - \nabla_{\BW} \mH(\mu_1) [X_0]}^2
        &=
        \EE \|\Sigma_{\mu_1}^{-1} (X_1 - X_0)\|^{2}
        \ctag{by \Cref{eqn:bw-grads}}
        \nonumber
        \\
        &\leq
        \beta^2\, \EE \norm{X_1 - X_0}^{2}
        \ctag{since $\Sigma_{\mu_1}^{-1} \preceq \beta I$}
        \nonumber
        \\
        &= 
        \beta^2\, W_2^2(\mu_0, \mu_1)\,,
        \label[Inequality]{ineq:x11-x10}
    \end{align}
    where the last equality holds since $(X_0, X_1) \sim (\mu_0, \mu_1)$ are optimally coupled by assumption.
    Now, we bound the second term.
    Define the functionals $\mV_1, \mF_1 : \BW(\RR^d)\to \RR$
    such that
    \begin{align*}
        \mV_1(\mu) &\coloneqq -\int \log \mu_1(x) \dd{\mu(x)}\,,\\
        \mF_1(\mu) &\coloneqq \mV_1(\mu) + \mH(\mu)\,.
    \end{align*}
    Note that by \Cref{eqn:bw-grads},
    $\nabla_{\BW} \mV_1(\mu) = - \nabla \log \mu_1 = -\nabla_{\BW} \mH(\mu_1)$,
    so that
    \begin{align*}
        \nabla_{\BW} \mF_1(\mu) 
        = 
        \nabla_{\BW} \mV_1(\mu) 
        +
        \nabla_{\BW} \mH(\mu) 
        =
        \nabla_{\BW} \mH(\mu) 
        -
        \nabla_{\BW} \mH(\mu_1) \,.
    \end{align*}
    Furthermore, we also note that
    \begin{equation}
        \KL{\mu}{\mu_1} = \mF_1(\mu) - \mF_1(\mu_1)\,.
        \label{eqn:kl-eq}
    \end{equation}
    Therefore, the second term that we want to control above can be interpreted as the squared norm of $\nabla_{\BW} \mF_1(\mu_0)$. We will show that $\mF_1$ is smooth, which will allow us to bound the squared gradient norm by a multiple of $\mF_1(\mu_0) - \mF_1(\mu_1) = \KL{\mu_0}{\mu_1}$ by the descent lemma from optimization.

    Let $\gamma \defeq c^{-1} \beta$, where $c \in (0, 1)$ is chosen to satisfy $c \leq (1 - c)^{2}$.
    Define the random variable $X_0'$ as follows:
    \begin{align*}
        X_0' 
        &\defeq
        X_0 - \frac{1}{\gamma}\,\nabla_\BW \mF_1(\mu_0)[X_0]
        \\
        &=
        X_0 - \frac{1}{\gamma}\,(\nabla_\BW \mH(\mu_0) - \nabla_\BW \mH(\mu_1))[X_0]
        \\
        &= 
        X_0 - \frac{1}{\gamma}\, \left( -\Sigma_{\mu_0}^{-1}(X_0 - m_{\mu_0}) + \Sigma_{\mu_1}^{-1}(X_0 - m_{\mu_1}) \right)
        \ctag{by \Cref{eqn:bw-grads}}
        \\
        &= 
        \underbrace{\Bigl( I + \frac{1}{\gamma}\, \Sigma_{\mu_0}^{-1} - \frac{1}{\gamma}\, \Sigma_{\mu_1}^{-1} \Bigr)}_{\defeq M_0}
        X_0 + \frac{1}{\gamma}\,( -\Sigma_{\mu_0}^{-1} m_{\mu_0} + \Sigma_{\mu_1}^{-1} m_{\mu_1} )\,.
    \end{align*}
    Let $\mu_0' \defeq \mathrm{law}(X_0')$.
    Since we have $0 \preceq \Sigma_{\mu_0}^{-1}, \Sigma_{\mu_1}^{-1} \preceq \beta I = c \gamma I$ by assumption,
    we have that
    \begin{align*}
        M_0 =
        I + \frac{1}{\gamma}\, \Sigma_{\mu_0}^{-1} - \frac{1}{\gamma}\, \Sigma_{\mu_1}^{-1}
        \succeq
        I - \frac{1}{\gamma}\, \Sigma_{\mu_1}^{-1}
        \succeq
        (1 - c)\, I
        \succeq 0\,,
    \end{align*}
    so $X_0'$ is equal to the gradient of a convex function of $X_0$.
    Hence, by Brenier's theorem, we conclude that
    $(X_0, X_0')\sim (\mu_0, \mu_0')$ are optimally coupled for the 
    $W_2$ distance.
    Thus, by \Cref{lemma:gg-ineqs-V} applied to the potential $\mV_1$, we find that
    \begin{align}
        \mV_1(\mu_0') - \mV_1(\mu_0)
        &\leq
        \EE \left\langle \nabla_{\BW} \mV_1(\mu_0)[X_0], X_0' - X_0\right\rangle + \frac{\beta}{2}\, \EE\norm{X_0' - X_0}^2
        \ctag{since $- \nabla^2 \log \mu_0 \preceq \beta I$}
        \nonumber
        \\
        &=
        -\EE \left\langle \nabla_{\BW} \mH(\mu_1)[X_0], 
        X_0' - X_0
        \right\rangle 
        + \frac{\beta}{2}\,\EE \norm{X_0' - X_0}^2\,.
        \label[Inequality]{ineq:fi-v}
    \end{align}
    Additionally, we note that since $\beta = c \gamma \leq (1 - c)^2\, \gamma$, we have that
    \begin{align*}
        \Sigma_{\mu_0'} 
        &= M_0 \Sigma_{\mu_0} M_0
        \succeq (1 - c)^2\, \Sigma_{\mu_0}
        \succeq \frac{(1 - c)^2}{\beta}\, I
        \succeq \frac{1}{\gamma}\, I\,.
    \end{align*}
    This implies that $\Sigma_{\mu_0'}^{-1}, \Sigma_{\mu_0}^{-1} \preceq \gamma I$.
    Hence, we can also apply the geodesic smoothness inequality of \Cref{lemma:entropy-smoothness}
    to obtain
    \begin{align}
        \mH(\mu_0') - \mH(\mu_0)
        \leq
        \EE \left\langle \nabla_\BW \mH(\mu_0)[X_0], X_0' - X_0 \right\rangle + \frac{\gamma}{2}\, \EE \norm{X_0' - X_0}^2\,.
        \label[Inequality]{ineq:fi-h}
    \end{align}
    Hence,
    combining \Cref{eqn:kl-eq} with \Cref{ineq:fi-v} and \Cref{ineq:fi-h},
    we obtain that
    \begin{align}
        - \KL{\mu_0}{\mu_1}
        &\leq
        \KL{\mu_0'}{\mu_1} - \KL{\mu_0}{\mu_1}
        \ctag{since $\KL{\mu_0'}{\mu_1} \geq 0$}
        \nonumber
        \\
        &=
        \mF_1(\mu_0') - \mF_1(\mu_0)
        \ctag{by \Cref{eqn:kl-eq}}
        \nonumber
        \\
        &=
        [
        \mV_1(\mu_0') - \mV_1(\mu_0)
        ]
        + 
        [
        \mH(\mu_0') - \mH(\mu_0)
        ]
        \nonumber
        \\
        &\leq
        -\EE \left\langle \nabla_{\BW} \mH(\mu_1)[X_0], 
        X_0' - X_0
        \right\rangle 
        +
        \frac{\beta}{2}\,\EE \norm{X_0' - X_0}^2
        \ctag{by \Cref{ineq:fi-v}}
        \nonumber
        \\
        &\qquad
        +
        \EE \left\langle \nabla_\BW \mH(\mu_0)[X_0], X_0' - X_0 \right\rangle 
        +
        \frac{\gamma}{2}\, \EE \norm{X_0' - X_0}^2
        \ctag{by \Cref{ineq:fi-h}}
        \nonumber
        \\
        &=
        \Bigl( -\frac{1}{\gamma} + \frac{\beta}{2\gamma^2} + \frac{1}{2\gamma} \Bigr)\,
        \EE \norm{\nabla_{\BW} \mH(\mu_0) [X_0] - \nabla_{\BW} \mH(\mu_1) [X_0]}^2
        \ctag{definition of $X_0'$}
        \nonumber
        \\
        &=
        - \frac{1 - c}{2\gamma}\,
        \EE \norm{\nabla_{\BW} \mH(\mu_0) [X_0] - \nabla_{\BW} \mH(\mu_1) [X_0]}^2\,.
        \label[Inequality]{ineq:kl-fi}
    \end{align}
    To bound the LHS of this inequality, we again apply
    \Cref{lemma:gg-ineqs-V} to the potential $\mV_1$ as well as \Cref{lemma:entropy-smoothness} to $\mH$ to obtain
    \begin{align}
        \KL{\mu_0}{\mu_1}
        &=
        \mF_1(\mu_0) - \mF_1(\mu_1)
        \nonumber
        \\
        &= 
        [
        \mV_1(\mu_0) - \mV_1(\mu_1)
        ]
        + 
        [
        \mH(\mu_0) - \mH(\mu_1)
        ]
        \nonumber
        \\
        &\leq
        \EE \left\langle 
        \nabla_{\BW} \mV(\mu_1)[X_1], 
        X_0 - X_1
        \right\rangle 
        +
        \frac{\beta}{2} \,\EE\norm{X_0 - X_1}^2
        \ctag{by \Cref{lemma:gg-ineqs-V} since $-\nabla^2 \log \mu_1 \preceq \beta I$}
        \nonumber
        \\
        &\qquad
        +
        \EE \left\langle \nabla_\BW \mH(\mu_1)[X_1], X_0 - X_1 \right\rangle 
        +
        \frac{\beta}{2}\, \EE \norm{X_0 - X_1}^2
        \ctag{by \Cref{lemma:entropy-smoothness} since $\Sigma_{\mu_0}^{-1}, \Sigma_{\mu_1}^{-1} \preceq \beta I$}
        \nonumber
        \\
        &=
        \beta\, \EE \norm{X_0 - X_1}^2
        \ctag{since $\nabla_{\BW} \mV_1(\mu_1) + \nabla_{\BW} \mH(\mu_1) = \nabla_{\BW} \mF_1(\mu_1) = 0 $}
        \nonumber\\
        &=
        \beta\, W_2^2(\mu_0, \mu_1)\,.
        \label[Inequality]{ineq:kl-w2}
    \end{align}
    Finally, choosing $c = \frac{1}{3}$ so that $c \leq (1 - c)^2$ and combining our above inequalities, we find that
    \begin{align*}
        \half\,
        \EE \norm{\nabla_{\BW} \mH(\mu_{1})[X_1] - \nabla_{\BW}\mH (\mu_0)[X_0]}^2
        &\leq
        \EE \norm{\nabla_{\BW} \mH(\mu_{1})[X_1] - \nabla_{\BW}\mH (\mu_1)[X_0]}^2
        \\
        &\qquad
        +
        \EE \norm{\nabla_{\BW} \mH(\mu_0) [X_0] - \nabla_{\BW} \mH(\mu_1) [X_0]}^2
        \\
        &\leq
        \beta^2\, W_2^2(\mu_0, \mu_1)
        +
        \frac{2\gamma}{1-c}\, \KL{\mu_0}{\mu_1}
        \ctag{by \Cref{ineq:x11-x10} and \Cref{ineq:kl-fi}}
        \\
        &\leq
        10\beta^2\, W_2^2(\mu_0, \mu_1)\,.
        \ctag{by \Cref{ineq:kl-w2}}
    \end{align*}
    Rearranging, we obtain our desired result.
\end{proof}

With this result in mind, we are ready to prove our desired stationary point guarantee.
\begin{proof}
    Let $(X_k, X_{k+\half}) \sim (\iterate_k, \iterate_{k+\half})$ and $(X_{k+\half}, X_{k+1}) \sim (\iterate_{k+\half}, \iterate_k)$ be optimally coupled for the $W_2$ distance,
    noting
    as in the proof of \Cref{lemma:discrete-evi}
    that by construction,
    \begin{align*}
        \frac{X_k - X_{k+1}}{\eta}
        =
        \nabla_{\BW} \mV(\iterate_k)[X_k] + \nabla_{\BW} \mH(\iterate_{k+1})[X_{k+1}]\,.
    \end{align*}
    Applying \Cref{lemma:iterate-osi}, we obtain that
    \begin{align*}
        \EE \norm{X_{k+1} - X_{k}}^2
        \leq 
        -2 \eta\, \EE [\mF(\iterate_{k+1}) - \mF(\iterate_k)]\,.
    \end{align*}
    Telescoping this inequality, we find that
    \begin{align}
        \min_{k \in \{0, \ldots, N-1\}}
        \EE \norm{X_{k+1} - X_k}^2
        \nonumber
        &\leq
        \frac{1}{N} \sum_{k=0}^{N-1} 
        \EE \norm{X_{k+1} - X_k}^2
        \nonumber
        \\
        &\leq 
        -\frac{2\eta}{N} \sum_{k=0}^{N - 1}
        \EE [\mF(\iterate_{k+1}) - \mF(\iterate_k)]
        \nonumber
        \\
        &= -\frac{2\eta}{N}\, \EE [\mF(\iterate_{N}) - \mF(\iterate_0)]
        \nonumber
        \\
        &\leq \frac{2 \eta \Delta}{N}\,.
        \label[Inequality]{ineq:distance-bound}
    \end{align}
    Now, let $(X_k, X_{k+1}^\star) \sim (\iterate_k, \iterate_{k+1})$ be optimally coupled for the $W_2$ distance.
    By \Cref{lemma:iterate-smoothness}, we have that $\Sigma_k^{-1} \preceq \beta I$
    for all $k$, meaning that we can apply \Cref{lemma:fake-coercivity}
    with $\mu_0 = \iterate_k$ and $\mu_1 = \iterate_{k+1}$
    to obtain that
    \begin{align}
        \EE \norm{\nabla_{\BW} \mH(\iterate_k)[X_k] - \nabla_{\BW} \mH(\iterate_{k+1})[X_{k+1}^\star]}^2
        &\leq 20\beta^2\, W_2^2(\iterate_k, \iterate_{k+1})\,.
        \label[Inequality]{ineq:triangle-bound-1}
    \end{align}
    Furthermore, we have that
    \begin{align}
        \EE \norm{\nabla_{\BW} \mH(\iterate_{k+1})[X_{k+1}^\star] - \nabla_{\BW} \mH(\iterate_{k+1})[X_{k+1}]}^2
        &=
        \EE \norm{\Sigma_{k+1}^{-1}\,(X_{k+1}^\star - X_{k+1})}^2
        \nonumber
        \\
        &\leq
        \beta^2\, \EE \norm{X_{k+1}^\star - X_{k+1}}^2
        \nonumber
        \\
        &\leq
        2\beta^2\, \EE \norm{X_{k+1}^\star - X_k}^2 + 2\beta^2\,\EE \norm{X_{k+1} - X_k}^2
        \nonumber
        \\
        &=
        2\beta^2\, W_2^2(\iterate_k, \iterate_{k+1}) + 2\beta^2\,\EE \norm{X_{k+1} - X_k}^2\,.
        \label[Inequality]{ineq:triangle-bound-2}
    \end{align}
    With these inequalities in mind, we obtain that
    \begin{align*}
        \frac{1}{3}\,
        \norm{\nabla_{\BW} \mF(\iterate_k)}_{\iterate_k}^{2}
        &= 
        \frac{1}{3}\,
        \EE \norm{\nabla_{\BW} \mV(\iterate_k)[X_k] + \nabla_{\BW} \mH(\iterate_k)[X_k]}^2
        \\
        &\leq
        \EE \norm{\nabla_{\BW} \mV(\iterate_k)[X_k] + \nabla_{\BW} \mH(\iterate_{k+1})[X_{k+1}]}^2
        +
        \EE \norm{\nabla_{\BW} \mH(\iterate_k)[X_k] - \nabla_{\BW} \mH(\iterate_{k+1})[X_{k+1}^\star]}^2
        \\
        &\qquad
        +
        \EE \norm{\nabla_{\BW} \mH(\iterate_{k+1})[X_{k+1}^\star] - \nabla_{\BW} \mH(\iterate_{k+1})[X_{k+1}]}^2
        \ctag{by triangle inequality}\\
        &\leq
        \frac{1}{\eta^2}\, \EE \norm{X_{k+1} - X_k}^2
        +
        22\,\beta^2
        W_2^2(p_k, p_{k+1})
        +
        2\beta^2\, \EE \norm{X_{k+1} - X_k}^2
        \ctag{by \Cref{ineq:triangle-bound-1} and \Cref{ineq:triangle-bound-2}}\\
        &\leq
        \Bigl( 
        \frac{1}{\eta^2} 
        +
        24 \beta^2
        \Bigr)\,
        \EE \norm{X_{k+1} - X_k}^2
        \ctag{since $(X_k, X_{k+1})$ is a coupling of $(\iterate_k, \iterate_{k+1})$}\\
        &\leq
        \frac{25}{\eta^2}\, 
        \EE \norm{X_{k+1} - X_k}^2\,.
        \ctag{since $\beta \leq \eta^{-1}$}
    \end{align*}
    Combining the above with \Cref{ineq:distance-bound}, we obtain that
    \begin{align*}
        \min_{k \in \left\{ 0, \ldots, N-1 \right\}}
        \norm{\nabla_{\BW} \mF(\iterate_k)}_{\iterate_k}^{2}
        &\leq
        \min_{k \in \left\{ 0, \ldots, N-1 \right\}}
        \frac{75}{\eta^2}\,
        \EE \norm{X_{k+1} - X_k}^2
        \leq
        \frac{150\Delta}{\eta N}\,.
    \end{align*}
    Finally, taking $\eta = \frac{1}{\beta}$ and $N \geq \frac{150\beta \Delta}{\varepsilon^2}$, we obtain that
    \[
        \min_{k \in \left\{ 0, \ldots, N-1 \right\}}
        \norm{\nabla_{\BW} \mF(\iterate_k)}_{\iterate_k}^{2} \leq \varepsilon^2\,,
    \]
    as desired.
\end{proof}

\section{Proofs of the noisy algorithm convergence rates}

We once again utilize \Cref{lemma:discrete-evi} to obtain our desired
rates of convergence.
First, we must prove the bound
on $\sigma_k$ for Stochastic FB--GVI given in \Cref{lemma:grad-error-bound-text}.

\subsection{Proof of \Cref{lemma:grad-error-bound-text}}
\label{sec:error-bnd-sto-oracle}

\begin{proof}
    Let $\mu = \mN(m, \Sigma)$ be an element of
    $\PG(\RR^d)$.
    We first note that if $X\sim \mu$, then
    by integration by parts,
    \begin{align}
        \Sigma\, \EE \nabla^2 V (X)
        &=
        \Sigma \int \nabla^2 V \dd{\mu}
        \nonumber\\
        &= -
        \Sigma \int \nabla \mu \otimes \nabla V \ctag{integration by parts}
        \nonumber\\
        &= -\Sigma 
        \int  \nabla \ln \mu \otimes \nabla V \dd{\mu}
        \nonumber\\
        &=
        \int (x - m) \otimes \nabla V \dd{\mu(x)} \ctag{since $-\Sigma \,\nabla \ln \mu(x) = x - m$}
        \nonumber\\
        &=
        \EE [(X - m)\otimes \nabla V(X)]\,.
        \label{eqn:gaussian-ibp}
    \end{align}
    Hence,
    \begin{align*}
        \langle \EE \nabla^2 V (X), \Sigma \rangle
        &=
        \langle \EE [ \Sigma^{-1}\, (X - m) \otimes \nabla V(X)], \Sigma \rangle
        \ctag{by \Cref{eqn:gaussian-ibp}}
        \\
        &=
        \EE 
        \langle 
        \Sigma^{-1}\,(X - m) \otimes \nabla V(X), \Sigma
        \rangle
        \ctag{linearity of expectation and trace}
        \\
        &=
        \EE \langle \nabla V(X), X - m \rangle
        \ctag{cyclicity of trace}\,.
    \end{align*}
    Now, let $(X_k, Z) \sim (\iterate_k, \hat{\pi})$ be optimally coupled for the $W_2$
    distance and independent of $\hat{X}_k$.
    Recall also the Brascamp--Lieb inequality~\citep{brascamplieb}: if $\mu$ is a measure on $\RR^d$ with density $\mu\propto\exp(-W)$, where $W$ is twice continuously differentiable and strictly convex, then for any smooth test function $f : \RR^d\to\RR$ it holds that $\Var_\mu(f) \le \EE\langle \nabla f, (\nabla^2 W)^{-1}\,\nabla f\rangle$.
    In particular, if we take $f = \langle \nabla V, e\rangle$ for a unit vector $e$ and $\mu = p_k$, it follows that $\Var_{p_k}{\langle \nabla V, e\rangle} \le \EE_{p_k}\langle e, \nabla^2 V\,\Sigma_k\,\nabla^2 V\,e \rangle$.
    Summing this inequality as $e$ ranges over an orthonormal basis of $\RR^d$, we obtain
    \begin{align*}
        \EE_{p_k}\|\nabla V - \EE_{p_k}\nabla V\|^2
        &\le \EE_{p_k}\langle [\nabla^2 V]^2, \Sigma_k\rangle\,.
    \end{align*}
    Thus, we get that
    \begin{align*}
        \half\, \sigma_k^2
        &\leq
        \EE \|(\nabla^2 V(\hat{X}_k) - \EE_{\iterate_k} \nabla^2 V)(X_k - m_k)\|^2
        +
        \EE \|\nabla V(\hat{X}_k) - \EE_{\iterate_k} \nabla V\|^2
        \ctag{by triangle inequality}
        \\
        &=
        \bigl\langle 
        \EE_{\iterate_k}[ (\nabla^2 V - \EE_{\iterate_k} \nabla^2 V)^2],
        \Sigma_k 
        \bigr\rangle 
        +
        \EE_{\iterate_k} \norm{\nabla V - \EE_{\iterate_k} \nabla V}^2
        \ctag{since $X_k \indep \hat{X}_k$}
        \\
        &=
        \EE_{\iterate_k}
        \langle \nabla^2 V, \Sigma_k\, \nabla^2 V \rangle
        -
        \langle 
        \EE_{\iterate_k} [\nabla^2 V]^2,
        \Sigma_k 
        \rangle 
        +
        \EE_{\iterate_k} \norm{\nabla V - \EE_{\iterate_k} \nabla V}^2
        \\
        &\leq
        \EE_{\iterate_k}
        \langle \nabla^2 V, \Sigma_k\, \nabla^2 V \rangle
        +
        \EE_{\iterate_k} \norm{\nabla V - \EE_{\iterate_k} \nabla V}^2
        \ctag{
        since $
        \langle 
        \EE_{\iterate_k}[ (\nabla^2 V)^2],
        \Sigma_k 
        \rangle \geq 0$}
        \\
        &\leq
        2\,\EE_{\iterate_k} \langle \nabla^2 V, \Sigma_{k}\, \nabla^2 V \rangle
        \ctag{by Brascamp--Lieb}
        \\
        &\leq
        2\beta\, \EE_{\iterate_k}\langle \nabla^2 V, \Sigma_k \rangle
        \ctag{since $\nabla^2 V \preceq \beta I$ and $\nabla^2 V, \Sigma_k \succeq 0$}
        \\
        &=
        2\beta\, \EE\left\langle \nabla V(X_k), X_k - m_k \right\rangle
        \ctag{by \Cref{eqn:gaussian-ibp}}\\
        &=
        {2\beta\, 
        \underbrace{
        \EE \left\langle \nabla V(Z), Z - \hat{m} \right\rangle
        }_{\mathsf{err}_1}}
        +
        {2\beta\, 
        \underbrace{
        \EE \left\langle \nabla V(X_k) - \nabla V(Z), (X_k - m_k) - (Z - \hat{m}) \right\rangle
        }_{\mathsf{err}_2}}
        \\
        &\qquad{}
        +
        {2\beta\,
        \underbrace{
        \EE \left\langle \nabla V(Z), (X_k - m_k) - (Z - \hat{m}) \right\rangle
        }_{\mathsf{err}_3}}
        +
        {2\beta\, 
        \underbrace{
        \EE \left\langle \nabla V(X_k) - \nabla V(Z), Z - \hat{m} \right\rangle
        }_{\mathsf{err}_4}}\,.
    \end{align*}
    Now, we have the following:
    \begin{align*}
        \mathsf{err}_1
        &= 
        \EE \left\langle \nabla V(Z), Z - \hat{m} \right\rangle
        =
        \langle \EE \nabla^2 V(Z), \hat{\Sigma} \rangle
        \ctag{by \Cref{eqn:gaussian-ibp} and the stationarity conditions in~\eqref{eq:firstorderstationarity}}
        =
        \Tr(I)\\
        &= d\,,
        \\
        \mathsf{err}_2
        &= 
        \EE \left\langle \nabla V(X_k) - \nabla V(Z), (X_k - m_k) - (Z - \hat{m}) \right\rangle
        \\
        &\leq 
        \frac{1}{2\beta}\,
        \EE \norm{\nabla V(X_k) - \nabla V(Z)}^2
        + 
        \frac{\beta}{2}\,
        \EE \norm{(X_k - m_k) - (Z - \hat{m})}^2
        \ctag{Young's inequality}\\
        &\leq
        \beta\, \EE \norm{X_k - Z}^2
        \ctag{since $\nabla V$ is $\beta$-Lipschitz}\\
        &=
        \beta\,
        W_2^2(\mu_k, \hat{\pi})\,,
        \ctag{since $(X_k, Z)$ are optimally coupled}
        \\
        \mathsf{err}_3
        &= 
        \EE \left\langle \nabla V(Z), (X_k - m_k) - (Z - \hat{m}) \right\rangle
        \\
        &\leq
        \frac{1}{4\beta}\,
        \EE \norm{\nabla V(Z)}^2
        +
        \beta\,
        \EE \norm{(X_k - m_k) - (Z - \hat{m})}^2
        \ctag{Young's inequality}
        \\
        &\leq
        \frac{1}{4\beta}\,
        \EE \langle \nabla^2 V(Z)^2, \hat{\Sigma} \rangle
        + 
        \beta\,
        W_2^2(\mu_k, \hat{\pi})
        \ctag{Brascamp--Lieb, optimal coupling of $(X_k, Z)$}
        \\
        &\leq
        \frac{d}{4}
        + 
        \beta\, W_2^2(\mu_k, \hat{\pi})\,,
        \ctag{since $\EE_{\hat{\pi}} \nabla^2 V = \hat{\Sigma}^{-1}$ by \Cref{eq:firstorderstationarity} and $\nabla^2 V\preceq \beta I$}
        \\
        \mathsf{err}_4
        &=
        \EE \left\langle \nabla V(X_k) - \nabla V(Z), Z - \hat{m} \right\rangle
        \\
        &\leq
        \frac{\Tr(\hat{\Sigma})}{d}\,
        \EE \norm{\nabla V(X_k) - \nabla V(Z)}^2
        +
        \frac{d}{4\Tr(\hat{\Sigma})}\,
        \EE \norm{Z - \hat{m}}^2
        \ctag{Young's inequality}
        \\
        &\leq
        \frac{\beta^2\Tr(\hat{\Sigma})}{d}\,
        \EE \norm{X_k - Z}^2
        +
        \frac{d}{4\Tr(\hat{\Sigma})}
        \Tr(\hat{\Sigma})
        \ctag{since $\nabla V$ is $\beta$-Lipschitz}\\
        &\leq
        \frac{\beta^2\Tr(\hat{\Sigma})}{d}\,
        W_2^2(\mu_k, \hat{\pi})
        +
        \frac{d}{4}\,.
    \end{align*}
    Combining these, we obtain that
    \begin{align*}
        \sigma_k^2
        &\leq
        4\beta \sum_{i=1}^{4} \mathsf{err}_i
        \leq
        6\beta d + \Bigl( 8\beta^2 + \frac{4\beta^3 \Tr(\hat{\Sigma})}{d} \Bigr)\, W_2^2(\mu_k, \hat{\pi})
        \leq
        6\beta d + 12\beta^3\lambda_{\max}(\hat{\Sigma})\, W_2^2(\mu_k, \hat{\pi})\,.
        \ctag{since $\hat{\Sigma}^{-1} = \EE_{\hat{\pi}} \nabla^2 V \preceq \beta I$ so $\lambda_{\max}(\hat{\Sigma}) \geq 1/\beta$}
    \end{align*}
    Note that in the strongly convex case, by \Cref{eq:firstorderstationarity}, we obtain that
    \begin{align*}
        \lambda_{\max}(\hat{\Sigma}) = \lambda_{\max}(\EE_{\hat{\pi}}[\nabla^2 V]^{-1}) \leq \frac{1}{\alpha}\,,
    \end{align*}
    so this bound simplifies to
    \begin{align*}
        \sigma_k^2
        &\leq
        6\beta d + \frac{12\beta^3}{\alpha}\, W_2^2(\mu_k, \hat{\pi})\,.
    \end{align*}
    This concludes our proof.
\end{proof}

\subsection{One-step inequality using the bound on \texorpdfstring{$\sigma_k$}{the variance}}

We apply the error bound in \Cref{lemma:grad-error-bound-text}
along with the \acl{OSI} of \Cref{lemma:discrete-evi} with $\nu = \hat{\pi}$ and $\eta \leq \frac{1}{2\beta}$.
This gives us the inequality
\begin{align}
    \EE W_2^2(\iterate_{k+1}, \hat{\pi})
    &\leq
    (1 - \alpha \eta) \EE W_2^2(\iterate_k, \hat{\pi}) - 2\eta\, (\EE \mF(\iterate_{k+1}) - \mF(\hat{\pi})) + 
    2\eta^2\,\EE \sigma_k^2
    \nonumber
    \\
    &\leq
    \bigl(1 - \alpha \eta + 24\beta^3 \eta^2 \lambda_{\max}(\hat{\Sigma}) \bigr)\, \EE W_2^2(\iterate_k, \hat{\pi}) - 2\eta\, (\EE \mF(\iterate_{k+1}) - \mF(\hat{\pi})) + 12\beta \eta^2 d
    \nonumber
    \\
    &\leq
    \exp\bigl( 
    - \alpha \eta + 24\beta^3\eta^2 \lambda_{\max}(\hat{\Sigma})\,
    \bigr)
    \EE W_2^2(\iterate_k, \hat{\pi}) - 2\eta\, (\EE \mF(\iterate_{k+1}) - \mF(\hat{\pi})) + 12\beta\eta^2 d\,.
    \label[Inequality]{ineq:stoch-evi}
\end{align}

\subsection{Proof of \Cref{thm:stoch-wc-rate}}
\label{section:stoch-wc-rate-proof}

\begin{proof}
    Define
    $c \defeq 24\beta^3 \lambda_{\max}(\hat{\Sigma})$.
    Since $V$ is convex by assumption, we may take $\alpha = 0$ in 
    \Cref{ineq:stoch-evi} to obtain that
    \begin{align*}
        2\eta\,(\EE \mF(\iterate_{k+1}) - \mF(\hat{\pi}))
        &\leq
        e^{c\eta^2}\,
        \EE W_2^2 (\iterate_k, \hat{\pi})
        - \EE W_2^2(\iterate_{k+1}, \hat{\pi})
        + 12 \beta\eta^2 d\,.
    \end{align*}
    Define
    $S_N(\eta) \defeq \sum_{k=1}^{N} e^{-kc\eta^2}$.
    We then find that
    \begin{align*}
        \sum_{k=0}^{N-1}
        2\eta\, 
        e^{-(k+1)c\eta^2}\,
        (\EE \mF(\iterate_{k+1}) - \mF(\hat{\pi}))
        &\leq
        \sum_{k=0}^{N-1}
        e^{-(k+1)c\eta^2}
        \,\bigl( 
        e^{c\eta^2}\,
        \EE W_2^2 (\iterate_k, \hat{\pi})
        - \EE W_2^2(\iterate_{k+1}, \hat{\pi})
        + 12 \beta\eta^2 d
        \bigr)\\
        &=
        W_2^2 (\iterate_0, \hat{\pi})
        - 
        e^{-Nc\eta^2}\,
        \EE W_2^2(\iterate_N, \hat{\pi})
        + 12 \beta \eta^2 d
        \sum_{k=0}^{N-1}
        e^{-(k+1)c\eta^2}\\
        &\le
        W_2^2 (\iterate_0, \hat{\pi})
        + 12 \beta \eta^2 d
        S_N(\eta)\,.
    \end{align*}
    Let $\ol{\iterate}$ be drawn randomly from among $\left\{ \iterate_k \right\}_{k=1}^N$, with probability of choosing $\iterate_k$ proportional to $e^{-kc\eta^2}$.
    Then we have that
    \begin{align*}
        \EE \mF(\ol{\iterate}) - \mF(\hat{\pi})
        &=
        \frac{1}{2\eta S_N(\eta)}
        \sum_{k=0}^{N-1}
        2\eta\, e^{-(k+1)c\eta^2}\,
        (\EE \mF(\iterate_{k+1}) - \mF(\hat{\pi}))
        \\
        &\le
        \frac{1}{2\eta S_N(\eta)}\,
        \bigl( 
        W_2^2 (\iterate_0, \hat{\pi})
        + 12 \beta \eta^2 d S_N(\eta)
        \bigr)\\
        &= 
        \frac{W_2^2(\iterate_0, \hat{\pi})}{2\eta S_N(\eta)}
        + 6\beta \eta d\,.
    \end{align*}
    Now, we note that
    \begin{align*}
        S_N(\eta)
        &=  
        \sum_{k=1}^{N} e^{-kc\eta^2}
        \geq
        \sum_{k=1}^{N \land (c\eta^2)^{-1}} e^{-kc\eta^2}
        \geq
        \sum_{k=1}^{N \land (c\eta^2)^{-1}} e^{-1}
        \geq
        \frac{N \land \lfloor (c\eta^2)^{-1}\rfloor}{e}\,.
    \end{align*}
    Thus, we obtain the inequality
    \begin{align*}
        \EE 
        \bigl[ 
        \min_{k \in \left\{ 1, \ldots, N \right\}}
        \mF(\iterate_k)
        \bigr]
        - \mF(\hat{\pi})
        &\leq
        \EE \mF(\ol{\iterate}) - \mF(\hat{\pi})
        \\
        &\leq
        \frac{W_2^2(\iterate_0, \hat{\pi})}{2\eta S_N(\eta)}
        + 6\beta\eta d\\
        &\leq
        \frac{2 W_2^2(\iterate_0, \hat{\pi})}{\eta\, (N \land \lfloor (c\eta^2)^{-1} \rfloor)}
        + 6\beta \eta d\\
        &\lesssim
        \frac{W_2^2(\iterate_0, \hat{\pi})}{\eta N}
        + c\eta W_2^2(\iterate_0, \hat{\pi})
        + \beta \eta d\,.
    \end{align*}
    Hence, taking
    \begin{align*}
        \eta 
        &\asymp \frac{\varepsilon^2}{c W_2^2(\iterate_0, \hat{\pi}) \lor \beta d}
        \asymp \frac{\varepsilon^2}{\beta^3 \lambda_{\max}(\hat{\Sigma})\, W_2^2(\iterate_0, \hat{\pi}) \lor \beta d}\,,
        \\
        N
        &\gtrsim 
        \frac{W_2^2(\iterate_0, \hat{\pi})}{\eta\varepsilon^2}
        \asymp 
        \frac{W_2^2(\iterate_0, \hat{\pi})}{\varepsilon^4}\, \bigl( \beta^3 \lambda_{\max}(\hat{\Sigma})\, W_2^2(\iterate_0, \hat{\pi}) \lor \beta d \bigr)\,,
    \end{align*}
    we get the guarantee $
   \EE \bigl[ \min_{k \in \left\{ 1, \ldots, N \right\}} \mF(\iterate_k) \bigr] - \mF(\hat{\pi}) \leq \varepsilon^2\,$.
\end{proof}

\subsection{Proof of \Cref{thm:stoch-sc-rate}}
\label{section:stoch-sc-rate-proof}

\begin{proof}
    In the strongly convex case where $0 \prec \alpha I\preceq \nabla^2 V$, we
    have the eigenvalue guarantee $\lambda_{\max}(\hat{\Sigma}) \leq \frac{1}{\alpha} I$,
    since $\EE_{\pi}\nabla^2 V = \hat{\Sigma}^{-1}$
    by \eqref{eq:firstorderstationarity}. Hence, under the assumption that $\eta \leq \frac{\alpha^2}{48 \beta^3}$,
    \Cref{ineq:stoch-evi} implies that
    \begin{align}
        \EE W_2^2(\iterate_{k+1}, \hat{\pi})
        &\leq 
        \exp\Bigl( 
        - \alpha \eta + \frac{24\beta^3\eta^2}{\alpha} 
        \Bigr)\,
        \EE W_2^2(\iterate_k, \hat{\pi}) - 2\eta\, (\EE \mF(\iterate_{k+1}) - \mF(\hat{\pi})) + 12\beta \eta^2 d
        \nonumber
        \\
        &\leq
        \exp\left( -\frac{\alpha\eta}{2} \right)\, \EE  W_2^2(\iterate_k, \hat{\pi}) - 2\eta\, (\EE \mF(\iterate_{k+1}) - \mF(\hat{\pi})) + 12\beta\eta^2 d\,.
        \nonumber
    \end{align}
    Since $\mF(\hat{\pi}) \leq \mF(\iterate_{k+1})$, we may iterate this inequality
    to obtain that
    \begin{align*}
        \EE W_2^2(\iterate_{N}, \hat{\pi})
        &\leq
        \exp\Bigl( -\frac{N\alpha\eta}{2} \Bigr)\,
        W_2^2(\iterate_0, \hat{\pi}) + \frac{24 \beta \eta d}{\alpha}\,.
    \end{align*}
    Hence, with the choice
    \begin{align*}
        \eta &\asymp \frac{\varepsilon^2}{\beta d}\,,\quad 
\text{and} \quad      N \gtrsim 
        \frac{1}{\alpha\eta} \log \frac{\alpha W_2^2(\iterate_0, \hat{\pi})}{\varepsilon^2}
        \asymp
        \frac{\beta d}{\alpha \varepsilon^2} \log \frac{\alpha W_2^2(\iterate_0, \hat{\pi})}{\varepsilon^2}\,,
    \end{align*}
    we obtain the guarantee $        \alpha\, \EE W_2^2(\iterate_N, \hat{\pi}) \leq \varepsilon^2\,$.    Now, for the guarantee in KL divergence,
    we ``reinitialize''
    the algorithm with distribution $\iterate_N$
    and apply the convex result of
    \Cref{thm:stoch-wc-rate}.
    This argument is inspired by~\citet{durmus2019analysis}.
    Assuming $\varepsilon$ is sufficiently small, we get that
    \begin{align*}
        c\, \EE W_2^2(\iterate_N, \hat{\pi}) \leq \frac{c\varepsilon^2}{\alpha} \leq \beta d\,,
    \end{align*}
    meaning that for the above choice of $\eta$, we have
    \begin{align*}
        \eta \asymp \frac{\varepsilon^2}{\beta d} \asymp
        \frac{\varepsilon^2}{c \,\EE W_2^2(\iterate_N, \hat{\pi}) \lor \beta d}\,.
    \end{align*}
    Furthermore, for our choice of $N$, we have that
    \begin{align*}
        \frac{\EE W_2^2(\iterate_N, \hat{\pi})}{\varepsilon^4}\, \bigl( c W_2^2(\iterate_0, \hat{\pi}) \lor \beta d \bigr) \leq \frac{\beta d}{\alpha\varepsilon^2} \lesssim N\,.
    \end{align*}
    Thus, applying \Cref{thm:stoch-wc-rate} with our choice of step size $\eta$ and iteration count $N$, we obtain that
    \begin{align*}
        \EE\bigl[ 
        \min_{k \in \left\{ 1, \ldots, 2N \right\}} \mF(\iterate_k) 
        \bigr]
        - \mF(\hat{\pi})
        &\leq
        \EE\bigl[ 
        \min_{k \in \left\{ N+1, \ldots, 2N \right\}} \mF(\iterate_k) 
        \bigr]
        - \mF(\hat{\pi})
        \\
        &\lesssim
        \EE\Bigl[ 
        \frac{W_2^2(\iterate_N, \hat{\pi})}{\eta N} + c \eta\, W_2^2(\iterate_N, \hat{\pi}) + \beta\eta d
        \Bigr]\\
        &\lesssim
        \varepsilon^2\,,
    \end{align*}
    proving our desired result.
\end{proof}

\end{document}